\documentclass[11pt]{amsart}

\usepackage[centertags]{amsmath}
\usepackage{amsfonts}
\usepackage{amssymb}
\usepackage{amsthm}
\usepackage[english]{babel}
\usepackage[active]{srcltx}
\usepackage{enumerate}
\usepackage{mathrsfs}

\numberwithin{equation}{section}

\newtheorem{thm}{Theorem}[section]
\newtheorem{cor}[thm]{Corollary}
\newtheorem{lem}[thm]{Lemma}
\newtheorem{prop}[thm]{Proposition}

\def\N{{{\Bbb N}}}

\def\R{\mathbb{R}}
\def\T{\mathbb{T}}

\theoremstyle{remark}
\newtheorem{rem}[thm]{Remark}

\theoremstyle{example}

\newcommand{\vertiii}[1]{{\left\vert\kern-0.25ex\left\vert\kern-0.25ex\left\vert #1
    \right\vert\kern-0.25ex\right\vert\kern-0.25ex\right\vert}}

\begin{document}

\title[
New estimates for the maximal functions
]{New estimates for the maximal functions \\ and applications}
\author
{\'Oscar Dom\'inguez}

\address{O. Dom\'inguez, Departamento de An\'alisis Matem\'atico y Matem\'atica Aplicada, Facultad de Matem\'aticas, Universidad Complutense de Madrid\\
Plaza de Ciencias 3, 28040 Madrid, Spain.}
\email{oscar.dominguez@ucm.es}

\author{Sergey Tikhonov}

 \address{S. Tikhonov, Centre de Recerca Matem\`{a}tica\\
Campus de Bellaterra, Edifici C
08193 Bellaterra (Barcelona), Spain;
ICREA, Pg. Llu\'{i}v Companys 23, 08010 Barcelona, Spain,
 and Universitat Aut\`{o}noma de Barcelona.}
\email{ stikhonov@crm.cat}

\maketitle

\bigskip
\begin{abstract}
	In this paper we study sharp pointwise inequalities for maximal operators. In particular, we strengthen 
 DeVore's inequality for the moduli of smoothness and a logarithmic variant of  Bennett--DeVore--Sharpley's inequality for rearrangements. As a consequence, we improve the classical  Stein--Zygmund embedding
deriving
$		\dot{B}^{d/p}_\infty L_{p,\infty}(\R^d) \hookrightarrow \text{BMO}(\R^d)$ for
$1 < p < \infty$. Moreover, these results are also applied to establish new Fefferman--Stein inequalities, Calder\'on--Scott type inequalities, and extrapolation estimates. Our approach is based on the
 limiting interpolation  techniques. 

\end{abstract}

\section{Notation}

As usual, $\R^d$ denotes the Euclidean $d$-space, $\T = [0, 2 \pi]$ is the unit circle, 
  and $\N_0 = \N\cup \{0\}$.
For $1 \leq p \leq \infty$, $p'$ is defined by $\frac{1}{p} + \frac{1}{p'}=1$.

By $Q$ or  $Q_0$ we denote cubes in $\R^d$ with sides parallel to the coordinate axes. Moreover,  $2 Q$ stands for the cube, which has the same center as $Q$ whose side length is twice that of $Q$.

Let $\mu$ be a Radon 
 positive measure on $\R^d$. Throughout this paper we assume that $\mu$ satisfies the doubling condition, i.e., there exists a constant $c_\mu > 0$ such that
\begin{equation*}
	\mu (2Q) \leq c_\mu \,\mu(Q)
\end{equation*}
for all cubes $Q$. The distribution function of a scalar-valued measurable function $f$ defined on $\R^d$ is
 \begin{equation*}
 	\mu_f(t) = \mu \{x \in \R^d : |f(x)| > t\}, \quad t >0,
 \end{equation*}
the non-increasing rearrangement is given by
\begin{equation*}
	f^*_\mu(t) = \inf \{\lambda \geq 0 : \mu_f(\lambda) \leq t\}
\end{equation*}
and
\begin{equation*}
	f^{**}_\mu(t) = \frac{1}{t} \int_0^t f^*_\mu(s) \, ds.
\end{equation*}
In the special case that $\mu = |\cdot|_d$, the Lebesgue measure on $\R^d$, we write simply $f^*$ and $f^{**}$, respectively.

We will assume that $A\lesssim B$ means that $A\leq C B$ with a positive constant $C$ depending only on nonessential parameters.
If $A\lesssim B\lesssim A$, then $A\asymp B$.

\subsection{Maximal functions}

For $f \in L_1(Q_0, \mu)$, the local sharp maximal function 
 is defined by
\begin{equation*}
	f^{\#}_{Q_0;\mu}(x) = \sup_{x \in Q, Q \subset Q_0} \frac{1}{\mu(Q)} \int_Q |f(y) - f_{Q;\mu}| \, d\mu (y), \quad x \in Q_0,
\end{equation*}
where $f_{Q;\mu} = \frac{1}{\mu(Q)} \int_Q f d\mu$. The space $\text{BMO}(Q_0,\mu)$ consists of all $f \in L_1(Q_0,\mu)$ such that
\begin{equation*}
	\|f\|_{\text{BMO}(Q_0,\mu)} = \sup_{x \in Q_0} f^{\#}_{Q_0;\mu}(x) < \infty.
\end{equation*}
 Similarly, given a locally integrable function $f$ on $\R^d$, we define
\begin{equation*}
	f^{\#}_\mu(x) = \sup_{x \in Q} \frac{1}{\mu(Q)} \int_Q |f(y) - f_{Q;\mu}| \, d \mu(y), \quad x \in \R^d,
\end{equation*}
and $\|f\|_{\text{BMO}(\R^d, \mu)} = \sup_{x \in \R^d}f^{\#}_\mu(x) < \infty.$ When $\mu = |\cdot|_d$,  we use the notation $f_Q, f_{Q_0}^{\#}, f^{\#}, \text{BMO}(Q_0)$, and $\text{BMO}(\R^d)$.


Let $s \in (0,1)$. The Str\"omberg--Jawerth--Torchinsky local maximal operator \cite{Stromberg, JawerthTorchinsky} is defined by
\begin{equation*}
	M^{\#}_{s, Q_0;\mu} f (x) = \sup_{x \in Q, Q \subset Q_0} \inf_{c \in \R} \inf \big\{\alpha \geq 0 : \mu\{y \in Q: |f(y) - c| > \alpha\} < s \mu(Q)\big\}, \quad x \in Q_0.
\end{equation*}
Again if $\mu$ is the Lebesgue measure we write $M^{\#}_{s, Q_0} f$. 

The relationship between $f^{\#}_{Q_0;\mu}$ and $M^{\#}_{s, Q_0;\mu} f$ is given by the well-known equivalence
\begin{equation}\label{EquivMax}
	(f^{\#}_{\mu})^*_\mu (t) \asymp (M^{\#}_{s;\mu} f)^{**}_\mu(t)
\end{equation}
provided that $s$ is small enough; see \cite[Lemma 3.4]{JawerthTorchinsky}.

We  also consider the modification 
 of the Str\"omberg--Jawerth--Torchinsky maximal function given by
\begin{equation*}
	\overline{M}^{\#}_{s, Q_0} f (x) = \sup_{x \in Q, Q \subset Q_0}  \inf \big\{\alpha \geq 0 : \big|\{y \in Q: |f(y) - f_Q| > \alpha\}\big|_d \leq s |Q|_d\big\}, \quad x \in Q_0.
\end{equation*}

\subsection{Function spaces}\label{SectionFunctionSpaces}

 If $0 < p, q \leq \infty$ and $-\infty < b < \infty$, then $L_{p,q}(\log L)_b(\R^d,\mu)$ stands for the \emph{Lorentz--Zygmund space} formed by all (equivalence classes of) measurable functions $f$ defined on $\R^d$ such that
\begin{equation}\label{LZ}
	\|f\|_{L_{p,q}(\log L)_b(\R^d,\mu)} = \Big(\int_0^\infty (t^{1/p} (1 + |\log t|)^b f^*_\mu(t))^q \frac{dt}{t} \Big)^{1/q} < \infty
\end{equation}
(where the integral should be replaced by the supremum if $q=\infty$).
  When $p=\infty$ we assume that $b < -1/q \, (b \leq 0 \text{ if } q=\infty)$, otherwise $L_{\infty,q}(\log L)_b(\R^d,\mu)=\{0\}$. The quasi-norm \eqref{LZ} can be equivalently characterized in terms of $\mu_f(t)$ (cf. \cite[Proposition 2.2.5]{Carro}). Again the symbol $\mu$ will be omitted in the notation of Lorentz--Zygmund spaces when we deal with the Lebesgue measure.

The Lorentz--Zygmund spaces $L_{p,q}(\log L)_b(Q_0,\mu)$, as well as their periodic counterparts for the Lebesgue measure $L_{p,q}(\log L)_b(\T)$, are introduced in a similar way but now the integration in \eqref{LZ} extends over the interval $(0,1)$. For more details on Lorentz--Zygmund spaces, we refer to \cite{BennettRudnick, BennettSharpley, Carro}.

In particular,
setting $b=0$ in $L_{p,q}(\log L)_b(Q_0,\mu)$ we obtain the classical Lorentz spaces $L_{p,q}(Q_0,\mu)$ and
 \begin{equation*}
 	\|f\|_{L_{p,q}(Q_0,\mu)} \asymp \Big(\int_0^\infty t^q (\mu_f(t))^{q/p} \frac{d t}{t} \Big)^{1/q}.
 \end{equation*}
Moreover, letting $p=q < \infty$ in $L_{p,q}(\log L)_b(Q_0,\mu)$  we recover the Zygmund spaces $L_{p} (\log L)_b(Q_0,\mu)$ and
 \begin{equation*}
 	\|f\|_{L_{p} (\log L)_b(Q_0,\mu)} \asymp \Big(\int_{Q_0} [|f(x)| \log^b (e + |f(x)|)]^p \, d \mu(x) \Big)^{1/p}.
 \end{equation*}
For $b=0$, $L_{p}(\log L)_b(Q_0,\mu)= L_p(Q_0,\mu)$. If $p=\infty$ and $b < 0$, then $L_\infty (\log L)_b(Q_0,\mu)$ coincides with the Orlicz space of exponentially integrable functions $\text{exp}\, L^{-1/b}(Q_0,\mu)$, i.e.,
 \begin{equation*}
 	\text{exp}\, L^{-1/b}(Q_0,\mu) = \Big\{f : \int_{Q_0} \exp (\lambda |f(x)|^{-1/b}) \, d \mu (x) < \infty \quad \text{for some} \quad \lambda = \lambda(b, f) >0 \Big\}.
 \end{equation*}
 In the special case $b=-1$ we simply write $\text{exp}\, L (Q_0,\mu)$.

Let us now define the smooth function spaces that are useful in this paper.  Assume that  $1 < p < \infty, 1 \leq q \leq \infty$, and $-\infty < s, b < \infty$.
Let $\dot{H}^s L_{p,q} (\log L)_b(\R^d)$ be the \emph{Riesz-potential space}
 defined as  the completion of $C^\infty_0(\R^d)$ with respect to the functional
\begin{equation*}
	\|f\|_{\dot{H}^s L_{p,q}(\log L)_b(\R^d)} = \|(-\Delta)^{s/2} f\|_{L_{p,q}(\log L)_b(\R^d)}.
\end{equation*}
In particular, setting $q=p$ and $b=0$ one recovers the classical space $\dot{H}^s_p(\R^d) = \dot{H}^s L_p(\R^d)$. Note that if $k \in \N$ then $\dot{H}^k L_{p,q}(\log L)_b(\R^d)$ coincides (up to equivalence constants of semi-quasi-norms) with the \emph{(homogeneous) Lorentz--Zygmund--Sobolev space} $\dot{W}^k L_{p,q}(\log L)_b(\R^d)$, where
\begin{equation*}
	 \|f\|_{\dot{W}^k L_{p,q} (\log L)_b(\R^d)} = \|\, |\nabla^k f| \, \|_{L_{p,q}(\log L)_b(\R^d)}
\end{equation*}
and $|\nabla^k f (x) | = \Big(\sum_{|\alpha|_1 = k} |D^\alpha f (x)|^2 \Big)^{1/2}$. We shall also need their inhomogeneous counterparts $W^k L_{p,q} (\log L)_b(\mathcal{X}), \, \mathcal{X} \in \{\R^d, Q_0\},$ formed by all those $k$-times weakly differentiable functions $f$ on $\mathcal{X}$ such that
\begin{equation*}
	\|f\|_{W^k L_{p,q} (\log L)_b(\mathcal{X})} = \sum_{m=0}^k \|\, |\nabla^m f| \,\|_{ L_{p,q} (\log L)_b(\mathcal{X})} < \infty.
\end{equation*}
In particular, the choice $p=q$ and $b=0$ yields the classical Sobolev spaces $W^k_p(\mathcal{X})  = W^k L_p(\mathcal{X})$.

For $k \in \N$, we let $\omega_k(f,t)_{p,q,b; \R^d}$ denote the $k$-th order \emph{moduli of smoothness} of $f \in L_{p,q} (\log L)_b(\R^d)$ defined by
	\begin{equation*}
		\omega_k(f,t)_{p,q,b;\R^d} = \sup_{|h| \leq t} \| \Delta^k_h f\|_{L_{p,q}(\log L)_b (\R^d)}, \qquad t > 0,
	\end{equation*}
	where $\Delta^k_h f$ is the $k$-th  difference of $f$ with step $h$, that is,
	\begin{equation*}
		\Delta^1_h f(x) = \Delta_h f(x) = f(x+h)-f(x), \quad \Delta^k_h = \Delta_h \Delta^{k-1}_h, \quad h \in \R^d.
	\end{equation*}
	As usual, in the definition of $\omega_k(f,t)_{p,q,b;\T}$ in the case of $2 \pi$-periodic functions, the norm is taken over all of $\T$, while the modulus of smoothness on the cube $Q$ is given by
$
\omega_k(f,t)_{p,q,b;Q}= \sup_{|h| \leq t} \| \Delta^k_h f\|_{L_{p,q}(\log L)_b(Q_{kh})},
$
where the set $Q_{kh}:=\{x:x, x+kh\in Q\}.$ To simplify notation, we write $\omega_k(f,t)_{p,q,b}=\omega_k(f,t)_{p,q,b; \mathcal{X}}$ for $f \in L_{p,q} (\log L)_b(\mathcal{X}), \, \mathcal{X} \in \{\R^d, \T, Q\}$. Although we use the same notation for the moduli of smoothness of functions in $\R^d, \T$ and $Q$, this should hopefully cause no confusion, as the meaning should be clear from the context.
	
	Let $s > 0,$ $-\infty < b, \xi < \infty$, and $0 < r \leq \infty$. The \emph{(homogeneous) Besov-type space} $\dot{B}^{s,\xi}_r L_{p,q} (\log L)_b(\R^d)$ is defined
as  the completion of $C^\infty_0(\R^d)$ with respect to the (quasi-semi)-norm
	\begin{equation*}
		\|f\|_{\dot{B}^{s,\xi}_r L_{p,q} (\log L)_b(\R^d);k} = \Big(\int_0^\infty (t^{-s} (1 + |\log t|)^\xi \omega_k (f,t)_{p,q,b})^r \frac{dt}{t} \Big)^{1/r} < \infty
	\end{equation*}
	(with the usual modification if $r=\infty$). Let $B^{s,\xi}_r L_{p,q} (\log L)_b(\R^d)$ be the corresponding inhomogeneous  space, endowed with
	\begin{equation*}
		\|f\|_{B^{s,\xi}_r L_{p,q} (\log L)_b(\R^d);k} = \|f\|_{L_{p,q} (\log L)_b(\R^d)} +  \Big(\int_0^1 (t^{-s} (1 + |\log t|)^\xi \omega_k (f,t)_{p,q,b})^r \frac{dt}{t} \Big)^{1/r} < \infty
	\end{equation*}
	(with the usual modification if $r=\infty$). It is well known that these definitions are independent of $k$, in the sense that different choices of $k \in \N$ with $k > s$ give equivalent (quasi-semi-)norms on Besov spaces. Nevertheless, our notation is justified by the fact that the 
  equivalence constants may depend on $k$, which plays a key role to establish extrapolation estimates.

Letting $p=q$ and $b=0$ in $\dot{B}^{s,\xi}_r L_{p,q}(\log L)_b(\R^d)$ we recover the logarithmic Besov space $\dot{B}^{s,\xi}_r L_{p}(\R^d) = \dot{B}^{s,\xi}_{p,r}(\R^d)$ (cf. \cite{DominguezTikhonov} and the list of references given there).
    If $\xi=0$ and $b=0$ then $\dot{B}^{s,\xi}_r L_{p,q}(\log L)_b(\R^d) = \dot{B}^s_r L_{p,q}(\R^d)$, a Lorentz--Besov space (cf. \cite{GogatishviliOpicTikhonovTrebels, SeegerTrebels} and the references therein); if moreover $p=q$ then we arrive at  the classical Besov space $\dot{B}^s_{p,r}(\R^d)$ (see \cite{BennettSharpley, BerghLofstrom, Triebel83}).
	
	Similarly, one can define the periodic spaces $\dot{B}^{s,\xi}_r L_{p,q}(\log L)_b(\T), \, B^{s,\xi}_r L_{p,q}(\log L)_b(\T)$ and the spaces $\dot{B}^{s,\xi}_r L_{p,q}(\log L)_b(Q_0), \, B^{s,\xi}_r L_{p,q}(\log L)_b(Q_0)$ for cubes $Q_0 \subset \R^d$.

\subsection{Limiting interpolation}

Let $(A_0, A_1)$ be a compatible pair of quasi-Banach spaces. For all $f \in A_0+A_1$ and $t > 0$, the \emph{Peetre $K$-functional} is defined by
\begin{equation*}
	K(t,f) = K(t, f; A_0, A_1) = \inf_{\substack{f = f_0 + f_1 \\ f_i \in A_1, i=0,1}} \{\|f_0\|_{A_0} + t \|f_1\|_{A_1}\}.
\end{equation*}
Let $0 < \theta < 1, -\infty < b < \infty$, and $0 < q \leq \infty$. The \emph{real interpolation space} $(A_0,A_1)_{\theta,q;b}$ is the collection of all those $f \in A_0+A_1$ for which
\begin{equation}\label{interpolationnorm}
	\|f\|_{(A_0,A_1)_{\theta,q;b}} = \left(\int_0^\infty (t^{-\theta} (1 + |\log t|)^b K(t,f))^q \frac{dt}{t} \right)^{1/q} < \infty
\end{equation}
(with the usual modification if $q=\infty$). See \cite{BrudnyiKrugljak, EvansOpicPick, GogatishviliOpicTrebels, Gustavsson}. In particular, letting $b=0$ we recover the classical interpolation space $(A_0, A_1)_{\theta,q}$ (cf. \cite{BennettSharpley, BerghLofstrom, Triebel78}).

Given two (quasi-semi-)normed spaces $A_0$ and $A_1$, we write $A_0 \hookrightarrow A_1$ if $A_0 \subset A_1$ and the natural embedding from $A_0$ into $A_1$ is continuous. The space $A'$ is the dual space of the Banach space $A$.

It is easy to see that if $\theta=0$ or $\theta=1$ in \eqref{interpolationnorm} then we only obtain trivial spaces in general; that is why the definition of meaningful limiting interpolation spaces requires some modifications of the classical interpolation norm \eqref{interpolationnorm}. For simplicity,  assume that the pair $(A_0, A_1)$ is {ordered}, that is, $A_1 \hookrightarrow A_0$. It is not difficult to check that  \eqref{interpolationnorm} is equivalent to that obtained by replacing the interval $(0, \infty)$ to the smaller interval $(0,1)$. Namely, for $\theta\in (0,1)$,
\begin{equation*}
	\|f\|_{(A_0,A_1)_{\theta,q;b}} \asymp \left(\int_0^1 (t^{-\theta} (1 + |\log t|)^b K(t,f))^q \frac{dt}{t} \right)^{1/q}.
\end{equation*}
This basic observation together with the finer tuning given by logarithmic weights is the key to introduce \emph{limiting interpolation methods}. The space $(A_0, A_1)_{(1,b),q}$ is formed by all those $f \in A_0$ satisfying
\begin{equation}\label{DefLimInt}
	\|f\|_{(A_0,A_1)_{(1,b),q}} = \left(\int_0^1 (t^{-1} (1 + |\log t|)^{b} K(t,f))^q \frac{dt}{t} \right)^{1/q} < \infty.
\end{equation}
Here $b < -1/q \, (b \leq 0 \text{ if } q=\infty)$,  since otherwise $(A_0, A_1)_{(1,b),q}$ becomes the trivi\-al space, in the sense that it contains the zero element only. For our purposes, it is enough to work with ordered pairs $(A_0, A_1)$ and the limiting interpolation space with $\theta = 1$, but one may also consider limiting interpolation methods for general quasi-Banach pairs and $\theta=0$. For detailed information on limiting interpolation, we refer the reader to \cite{AstashkinLykovMilman, DominguezHaroskeTikhonov, DominguezTikhonov, EvansOpic, EvansOpicPick, GogatishviliOpicTrebels} and the references therein.

\bigskip
\section{New estimates of the maximal functions} 

\subsection{Estimates of the maximal functions in terms of measures of smoothness}\label{SectionEstimSmooth}
It is well known that there are strong relationships between
the oscillation of a function and its smoothness properties, see, e.g., \cite{DeVore}, \cite{DeVoreSharpley}, \cite{Garsia}, \cite{JohnNirenberg},  \cite{Kolyada86}, \cite{Kolyada87}, \cite{Kolyada87b}, \cite{Kolyada}, \cite{Kolyada99},  \cite{KolyadaLerner}, \cite{MartinMilman}, \cite{MartinMilman14}.

In this section we focus on the following well-known inequality for Besov functions on the real line
\begin{equation}\label{EmbD}
	\|f\|_{\text{BMO}(\R)} \lesssim \|f \|_{\dot{B}^{1/p}_{p,\infty}(\R);1}, \quad 1 < p < \infty.
\end{equation}
At the quantitative level, one may ask for pointwise estimates involving the maximal function $f^{\#}$ and the moduli of smoothness $\omega_1(f,t)_p$. The answer to this question was addressed by DeVore \cite[Theorem 1]{DeVore}. Namely, he showed that
\begin{equation}\label{DeVMax}
	f^{\# *}(t) \lesssim \sup_{t < u < \infty} u^{-1/p} \omega_1(f,u)_{p}, \qquad t > 0,
\end{equation}
for $f \in L_p(\R), \, 1 \leq p < \infty$. It is clear that \eqref{DeVMax} implies 
\eqref{EmbD}. Inequality \eqref{DeVMax} has been extended in several ways by Kolyada. The limiting case $p=1$ in \eqref{DeVMax} is of special interest, since 
  it yields  the quantitative version of the known embedding $\text{BV}(\R) \hookrightarrow \text{BMO}(\R)$. Note that in this case  inequality \eqref{DeVMax} can be equivalently written as
\begin{equation}\label{DeVMaxp=1}
	t f^{\# *}(t) \lesssim \omega_1(f,t)_{1}, \quad f \in L_1(\R),
\end{equation}
which complements the important inequality
\begin{equation}\label{DeVMaxp=1*}
	f^{**} (t) \lesssim \int_{t^{1/d}}^\infty \frac{\omega_d(f,u)_{1}}{u^d} \frac{du}{u}, \quad f \in  L_1(\R^d).
\end{equation}
The latter admits extensions to the moduli of smoothness based on $L_p(\R^d)$ (see, e.g.,  \cite[Chapter 5, Theorem 4.19]{BennettSharpley}) or, more generally, to r.i. spaces \cite[Corollary 1]{MartinMilman}. In the special case $d=1$,  inequalities \eqref{DeVMaxp=1} and \eqref{DeVMaxp=1*} were improved by Kolyada and Lerner. Namely, they showed that (cf. \cite[(24)]{KolyadaLerner})
\begin{equation}\label{KL}
	t f^{**}(t) \lesssim \omega_1(f,t)_{1}, \quad f \in L_1(\R).
\end{equation}
In the multivariate case  the following estimate was obtained by Kolyada \cite[Theorem 1]{Kolyada}: If $f \in L_1(\R^d)$ then
\begin{equation}\label{KL--}
	t^d f^{**}(t^d) \lesssim t\|f\|_{L_1(\R^d)} + \omega_1(f,t)_{1}, \quad t \in \big(0, 2^{-1/d} 
\big).
\end{equation}

Our first result sharpens  DeVore's inequality \eqref{DeVMax} for $p<\infty$ and Kolyada's estimate \eqref{KL--} for $p=1$.

\begin{thm}\label{ThmDeVoreLorentz}
	Let  $k \in \N$ and $t > 0$.
	
	Assume $f \in L_{p,q}(\R^d)$ with $1 < p < \infty$ and $0 <  q \leq \infty$. Then
	\begin{equation}\label{ThmDeVore*Lorentz}
	t^{-d/p}  \Big(\int_0^{t^d} (u^{1/p} f^{\# *}(u))^q \frac{du}{u} \Big)^{1/q} \lesssim \sup_{ t < u < \infty} u^{-d/p} \omega_k(f,u)_{p,q} \qquad \text{if} \quad k > d/p,
\end{equation}
and
\begin{equation}\label{ThmDeVore*Lorentz2}
	  \Big(\int_0^{t^d} (u^{1/p} f^{\# *}(u))^q \frac{du}{u} \Big)^{1/q}  \lesssim \omega_k(f,t)_{p,q} \qquad \text{if} \quad k\leq  d/p
\end{equation}
(with the standard modifications when $q=\infty$).

 Assume $f \in L_1(\R^d)$. Then 
\begin{equation}\label{ThmDeVore*New}
	t^d f^{**}(t^d) \lesssim \omega_d(f,t)_{1}
\end{equation}
and
\begin{equation}\label{KolyadaNew}
		t^k \int_{t^d}^\infty u^{1-k/d} f^{**}(u) \frac{du}{u} \lesssim  \omega_k(f,t)_1 \qquad \text{if}  \quad k < d.
	\end{equation}

\end{thm}

Note that by \eqref{EquivMax},  $f^{\#}$ can be replaced by
$M^{\#}_{s} f$ for small enough $s>0$
in inequalities \eqref{ThmDeVore*Lorentz} and \eqref{ThmDeVore*Lorentz2}. On the other hand, the corresponding inequalities to \eqref{ThmDeVore*New} and \eqref{KolyadaNew} obtained by replacing $f^{**}$ by $f^{\#*}$ also holds true (cf. \eqref{ProofLem2.1} below). In particular, it is plain to see that  \eqref{KolyadaNew} implies
\begin{equation*}
	  \int_0^{t^d} f^{\# *}(u) \, du   \lesssim \omega_k(f,t)_{1} \qquad \text{if} \quad k <  d
\end{equation*}
for $f \in L_1(\R^d)$, see \eqref{ThmDeVore*Lorentz2}.

Taking $p=q$ in  Theorem \ref{ThmDeVoreLorentz} permits us   to significantly extend DeVore's inequality \eqref{DeVMax}.

\begin{cor}\label{CorollaryThmDeVore}
	Let $k \in \N$ and $t > 0$. Assume $f \in L_p(\R^d)$ with $1 < p < \infty$. Then
	\begin{equation}\label{ThmDeVore*}
	t^{-d/p}  \Big(\int_0^{t^d} (f^{\# *}(u))^p \, du \Big)^{1/p} \lesssim \sup_{ t < u < \infty} u^{-d/p} \omega_k(f,u)_{p} \qquad \text{if} \quad   k > d/p,
\end{equation}
and
\begin{equation}\label{ThmDeVore*newextreme}
	  \Big(\int_0^{t^d} (f^{\# *}(u))^p \, du \Big)^{1/p}  \lesssim \omega_k(f,t)_{p} \qquad \text{if} \quad k \leq d/p.
\end{equation}


\end{cor}

Let us now discuss the optimality of Theorem \ref{ThmDeVoreLorentz} and its extension to cubes.
\begin{rem}\label{RemarkCubes}
(i) Assume $p > 1$.
It is clear that \eqref{ThmDeVore*} strengthens   \eqref{DeVMax}. Below (cf. Proposition \ref{ThmSharpnessAssertion}) we will construct the family of extremal functions for which  \eqref{ThmDeVore*Lorentz} (in particular, \eqref{ThmDeVore*}) becomes equivalence but 
  \eqref{DeVMax}
  is not applicable
 since
		\begin{equation*}
		\frac{f^{\# *}(t)}{\sup_{t < u < \infty} u^{-1/p} \omega_k(f,u)_{p}} \to 0 \quad \text{as} \quad t \to 0.
	\end{equation*}
	\\	
(ii)
Observe that 
 inequality \eqref{ThmDeVore*Lorentz2} for  $k < d/p$ follows from the sharper  estimate
	$$
	\Big(\int_0^{t^d} (u^{\frac{1}{p}} f^*(u))^q \frac{du}{u} \Big)^{1/q} + t^k \Big(\int_{t^d}^\infty (u^{\frac{1}{p} - \frac{k}{d}} f^*(u))^q \frac{du}{u} \Big)^{1/q} \lesssim \omega_k(f,t)_{p,q}.
	$$
	\\	
(iii) Inequalities \eqref{ThmDeVore*Lorentz} and \eqref{ThmDeVore*} do not depend on $k > d/p$ since 
\begin{equation*}
	\sup_{ t < u <\infty} u^{-d/p} \omega_k(f,u)_{p,q} \asymp \sup_{ t < u <\infty} u^{-d/p} \omega_l(f,u)_{p,q}, \qquad k, l > d/p.
\end{equation*}
This is an immediate consequence of the Marchaud inequalities for moduli of smoothness (see \cite{GogatishviliOpicTikhonovTrebels})
	\begin{equation}\label{KL---}
		\omega_l(f,t)_{p,q} \lesssim \omega_k (f,t)_{p,q} \lesssim t^k \int_t^\infty \frac{\omega_l(f,u)_{p,q}}{u^k} \frac{du}{u}, \qquad k < l.
	\end{equation}	
	\\
(iv)  Theorem \ref{ThmDeVoreLorentz} shows the important role played by the relationships between $k, d$ and $p$ in maximal inequalities. 
Comparing the cases $k > d/p$ and $k \leq d/p$, 
 the question is raised whether
 \eqref{ThmDeVore*Lorentz} can be sharpened by removing the  $\sup_{t < u < \infty}$. It will be shown below that this is not the case.
	 \\
(v) Assume $p=1$. Inequality \eqref{ThmDeVore*New}  sharpens  \eqref{KL--} in the sense of the order of moduli of smoothness, cf. \eqref{KL---}. 
  Furthermore, in Proposition \ref{SharpnessThmDeVore*New} below, we will show that \eqref{ThmDeVore*New} is sharp, that is, there exist functions for which the converse inequality is valid. On the other hand, the inequality \eqref{KolyadaNew} and its optimality are known (see \cite[Corollary 6]{Kolyada} and \cite[Theorem 4.3]{DominguezTikhonovB}), but it is incorporated to Theorem  \ref{ThmDeVoreLorentz} for the sake of completeness.
	\\
	(vi) Another interesting approach to estimates for maximal functions in terms of smoothness was proposed by Kolyada \cite{Kolyada87}. Among other results, he showed (cf. \cite[Theorem 1]{Kolyada87}) that if $f \in L^p(\R^d), \, 1 < p < \infty$, then
	\begin{equation}\label{Kolyada}
		\sum_{n \geq j : f^{\#*}(2^{-n d}) > f^{\#*}(2^{-(n+1) d})} 2^{-n d} (f^{\#*}(2^{-n d}))^p \lesssim \omega_1(f,2^{-j})_p^p, \quad j \in \mathbb{N}.
	\end{equation}
	Furthermore, if $1 < p < d$ then this estimate can be improved:
	 \begin{equation}\label{Kolyada2}
		\sum_{n \geq j} 2^{-n d} (f^{\#*}(2^{-n d}))^p \lesssim \omega_1(f,2^{-j})_p^p, \quad j \in \mathbb{N}.
	\end{equation}

To compare \eqref{Kolyada} with Corollary \ref{CorollaryThmDeVore}, it is convenient to rewrite \eqref{ThmDeVore*} with $k=1$ as follows:
	\begin{equation}\label{6656565}
		\sum_{n \geq j} 2^{-n d} (f^{\#*}(2^{-n d}))^p \lesssim 2^{-j d} \sup_{2^{-j} < u < \infty} u^{-d} \omega_1(f,u)_p^p, \qquad j \in \N,  \qquad d < p.
	\end{equation}
	On the one hand, the left-hand side of \eqref{6656565} sharpens the corresponding one in \eqref{Kolyada}. On the other hand, the $\sup_{2^{-j} < u < \infty}$ in the right-hand side of \eqref{6656565} is attained at $u=2^{-j}$ in \eqref{Kolyada}. This raises the natural question whether  the inequality \eqref{6656565} (or equivalently, \eqref{ThmDeVore*}) can be improved by
	 \begin{equation}\label{66565651}
		\sum_{n \geq j} 2^{-n d} (f^{\#*}(2^{-n d}))^p \lesssim \omega_1(f,2^{-j})_p^p, \qquad j \in \N,  \qquad d < p.
	\end{equation}
	However, this is not true. Indeed, assume that \eqref{66565651} holds, or equivalently,
$
		 \bigg(\int_0^{t^d} (f^{\# *}(u))^p \, du \bigg)^{1/p}  \lesssim \omega_1(f,t)_{p},$ $t \in (0,1).$ 
	Since  $\lim_{t \to 0+} \frac{\int_0^{t} (f^{\# *}(u))^p \, du }{t}  = (f^{\#*}(0))^p = \|f\|^p_{\text{BMO}(\R^d)}$, 
 then 
	\begin{equation*}
		 \|f\|^p_{\text{BMO}(\R^d)} \lesssim \lim_{t \to 0+} \frac{ \omega_1(f,t)_{p}}{t^{d/p}}, \quad d < p,
	\end{equation*}
	which is not true even for $f \in C^\infty_0(\R^d)$ (recall that $\omega_1(f,t)_p \asymp t$ for $t$ sufficiently small).
	
	Assume $d \geq 2$ and $p \in (1,d)$. Then inequalities \eqref{ThmDeVore*newextreme} with $k=1$ and \eqref{Kolyada2} coincide. Furthermore, \eqref{ThmDeVore*newextreme} shows that the limiting case of \eqref{Kolyada2} with $p=d \geq 2$ also holds. This was left open in \cite{Kolyada87}.
\\
  (vii) The counterpart of Theorem \ref{ThmDeVoreLorentz} for cubes reads as follows. Let  $k \in \N$ and $t \in (0,1)$. Assume $f \in L_{p,q}(Q_0)$ with $1 < p < \infty$ and $0 <  q \leq \infty$. Then
	\begin{equation}\label{ThmDeVore*LorentzCubes}
	t^{-d/p}  \Big(\int_0^{t^d} (u^{1/p} f_{Q_0}^{\# *}(u))^q \frac{du}{u} \Big)^{1/q} \lesssim  \|f\|_{L_{p,q}(Q_0)} + \sup_{ t < u < 1} u^{-d/p} \omega_k(f,u)_{p,q}, \qquad \text{if} \quad k > d/p,
\end{equation}
and
\begin{equation}\label{ThmDeVore*Lorentz2Cubes}
	  \Big(\int_0^{t^d} (u^{1/p} f_{Q_0}^{\# *}(u))^q \frac{du}{u} \Big)^{1/q}  \lesssim  t^k \|f\|_{L_{p,q}(Q_0)}+ \omega_k(f,t)_{p,q}, \qquad \text{if} \quad k\leq  d/p
\end{equation}
(with the standard modifications when $q=\infty$).

 Assume $f \in L_1(Q_0)$ and $t \in (0,1)$. Then 
\begin{equation}\label{ThmDeVore*NewCubes}
	t^d f^{**}(t^d) \lesssim t^d \|f\|_{L_1(Q_0)} + \omega_d(f,t)_{1}
\end{equation}
and
\begin{equation}\label{KolyadaNew**}
		t^k \int_{t^d}^1 u^{1-k/d} f^{**}(u) \frac{du}{u} \lesssim t^k \|f\|_{L_1(Q_0)} +   \omega_k(f,t)_1, \qquad \text{if}  \quad k < d.
	\end{equation}
Note that the inequalities \eqref{ThmDeVore*LorentzCubes}--\eqref{KolyadaNew**} are no longer true when we omit the terms involving $\|f\|_{L_{p,q}(Q_0)}$ or $\|f\|_{L_1(Q_0)}$ (take, e.g., polynomials).
  This is in sharp contrast with Theorem \ref{ThmDeVoreLorentz} showing
   an important difference  between inequalities  for functions on cubes and the Euclidean space.
\end{rem}

Working  with the Lorentz  norm in  Theorem \ref{ThmDeVoreLorentz} allows us to improve not only   the classical embedding \eqref{EmbD} but also the Stein--Zygmund embedding  \cite{SteinZygmund} (see also \cite[p. 164]{Stein70})
\begin{equation*}
		\dot{H}^{d/p} L_{p,\infty}(\R^d) \hookrightarrow \text{BMO}(\R^d), \quad 1 < p < \infty,
	\end{equation*}
 by
considering a finer scale of  Lorentz--Besov spaces.
 Similar approach has been recently applied in, e.g.,  \cite{Brezis, Brue, GrafakosSlavikova, MartinMilman, MartinMilman14, SeegerTrebels, Stolyarov}. 

	\begin{thm}\label{LemmaEmbBMOLorentzState}
	Assume $1 < p < \infty$. Then
	\begin{equation}\label{LemmaEmbBMOLorentz}
		\dot{B}^{d/p}_\infty L_{p,\infty}(\R^d) \hookrightarrow \emph{BMO}(\R^d).
	\end{equation}
	Consequently,
	\begin{equation}\label{LemmaEmbBMO*}
		\dot{H}^{d/p} L_{p,\infty}(\R^d) \hookrightarrow \emph{BMO}(\R^d)
	\end{equation}
	and
	\begin{equation}\label{LemmaEmbBMO}
		\dot{W}^k L_{d/k,\infty}(\R^d) \hookrightarrow \emph{BMO}(\R^d), 
\quad k < d.
	\end{equation}
\end{thm}

\begin{rem}
(i) The Stein--Zygmund embedding \eqref{LemmaEmbBMO*} follows from \eqref{LemmaEmbBMOLorentz} since $\dot{H}^{d/p} L_{p,\infty}(\R^d) \hookrightarrow \dot{B}^{d/p}_\infty L_{p,\infty}(\R^d)$ (cf. \cite[Theorem 1.2]{SeegerTrebels}).

(ii) Since
\begin{equation*}
\dot{B}^{d/p_0}_\infty L_{p_0,\infty}(\R^d) \hookrightarrow \dot{B}^{d/p_1}_\infty L_{p_1,\infty}(\R^d), \quad 1 < p_0 < p_1 < \infty,
\end{equation*}
 (cf. \cite[Theorem 1.5]{SeegerTrebels}), the domain space in \eqref{LemmaEmbBMOLorentz} becomes larger with $p \to \infty$. A similar comment  also applies to \eqref{LemmaEmbBMO*} and \eqref{LemmaEmbBMO} (cf. \cite[Theorem 1.6]{SeegerTrebels}). The limiting case $p=\infty$ in \eqref{LemmaEmbBMOLorentz} requires special care because the chosen definition of Besov spaces (with smoothness zero) makes important distinctions. Namely, the space $\dot{B}^{0}_\infty L_{\infty,\infty}(\R^d) = \dot{B}^0_{\infty,\infty}(\R^d)$ given in terms of the moduli of smoothness coincides with $C^\infty_0(\R^d)$ equipped with the $L_\infty(\R^d)$ norm. In this case, we obviously have $\dot{B}^{0}_{\infty,\infty}(\R^d) \hookrightarrow \text{BMO}(\R^d).$ Moreover, the converse embedding holds true when the Besov space $\dot{B}^{0}_{\infty,\infty}(\R^d)$ is replaced by its Fourier-analytically defined counterpart (cf. \cite[p. 24]{Triebel20}).
 \end{rem}

 Embedding \eqref{LemmaEmbBMO} sharpens
\begin{equation}\label{MP}
	\dot{W}^k L_{d/k,\infty}(\R^d) \hookrightarrow L(\infty,\infty)(\R^d), \quad k < d,
\end{equation}
(cf. \cite[Theorem 1.2]{Milman04}  and \cite{Milman16}), where $L(\infty,\infty)(\R^d)$ is the so-called weak-$L_\infty$ space. Note that $L(\infty,\infty)(\R^d)$ is the r.i. hull of $\text{BMO}(\R^d)$ (cf. \cite{BennettDeVoreSharpley}). The quantitative version of \eqref{MP} (with $k=1$) is given by the oscillation inequality
\begin{equation}\label{KolyadaGrad}
	f^{**}(t)-f^*(t) \lesssim t^{1/d} |\nabla f|^{**}(t), \quad f \in C^\infty_0(\R^d).
\end{equation}
This inequality plays a prominent role in the theory of Sobolev embeddings and isoperimetry, as can be seen in \cite{Bastero}, \cite{Kolyada89}, \cite{MartinMilman10} and the references quoted therein.  Another goal of this paper is to complement \eqref{KolyadaGrad}, as well as Theorem \ref{ThmDeVoreLorentz}, with the quantitative counterpart of \eqref{LemmaEmbBMO} in terms of $f^{\#}$. More precisely, we show the following.

\begin{thm}\label{ThmDeVoreDer}
	Let $1 < p < \infty,  k \in \N$, and $r= \frac{d p}{d + k p}$. Assume that either  of the following conditions is satisfied:
	\begin{enumerate}[\upshape(i)]
	\item $k < d (1-1/p)$ and $1 \leq q \leq \infty,$
	\item $k= d(1-1/p)$ and $q=1$.
	 \end{enumerate}
	  Then, given any $f \in \dot{W}^k L_{r,q}(\R^d) + \dot{W}^k L_{d/k, \infty}(\R^d)$ and $t > 0$, we have
		\begin{align}
	t^{-1/p} \Big(\int_0^{t} (u^{1/p}f^{\# *}(u))^q \, \frac{du}{u} \Big)^{1/q} & \lesssim t^{-1/p} \Big(\int_0^{t} (u^{1/r} |\nabla^k f|^*(u))^q \, \frac{du}{u} \Big)^{1/q} \nonumber \\
	& \hspace{1cm}+ \sup_{t < u < \infty} u^{k/d} |\nabla^k f|^*(u) \label{ThmDeVoreDer<}
	\end{align}
	(with the usual modification  if $q=\infty$).	The corresponding inequality for cubes reads as
		\begin{align}
	t^{-1/p} \Big(\int_0^{t} (u^{1/p}f^{\# *}_{Q_0}(u))^q \, \frac{du}{u} \Big)^{1/q} & \lesssim \sum_{l=0}^k \Big[ t^{-1/p} \Big(\int_0^{t} (u^{1/r} |\nabla^l f|^*(u))^q \, \frac{du}{u} \Big)^{1/q} \nonumber  \\
	& \hspace{1cm}+ \sup_{t < u < 1} u^{k/d} |\nabla^l f|^*(u)  \Big]  \label{ThmDeVoreDer<Cubes}
	\end{align}
	for $f \in W^k L_{r,q}(Q_0)$ and $t \in (0,1)$.
\end{thm}

\begin{rem}
(i) The two terms given in the right-hand side of \eqref{ThmDeVoreDer<} are independent of each other. More precisely, let
	\begin{equation*}
		I(t) = t^{-1/p} \Big(\int_0^{t} (u^{1/r} g^*(u))^q \, \frac{du}{u} \Big)^{1/q} \quad \text{and} \quad J(t) =  \sup_{t < u < \infty} u^{k/d} g^*(u).
	\end{equation*}
	If $g(x) = g_0(|x|)$ with $g_0(t) = t^{-d/r} (1 + |\log t|)^{-\varepsilon}, \, \varepsilon > 1/q$, then $J(t) \asymp t^{-1/p} (1 + |\log t|)^{-\varepsilon}$ and $I(t) \asymp t^{-1/p} (1 + |\log t|)^{-\varepsilon + 1/q}$. On the other hand, setting $g_0(t) = \chi_{(0,1)}(t)$ we have $I(t) \asymp t^{k/d}$ and $J(t) \asymp C$ for $t$ small enough.
	\\
	(ii) 
Inequality \eqref{ThmDeVoreDer<} is sharp in the sense that there exists a family of functions for which \eqref{ThmDeVoreDer<} becomes an equivalence (see Proposition \ref{ThmDeVoreDerSharpnessAssertion} below for the precise statement.)
 \\
 (iii) Inequalities \eqref{KolyadaGrad} and \eqref{ThmDeVoreDer<} with $k=1$ are not comparable in general. On the one hand, comparing  their left-hand sides, we have
 \begin{equation*}
 	f^{**}(t)-f^*(t) \lesssim t^{-1/p} \Big(\int_0^{t} (u^{1/p}f^{\# *}(u))^q \, \frac{du}{u} \Big)^{1/q}.
 \end{equation*}
 This is an immediate consequence of the Bennett--DeVore--Sharpley inequality (cf. \eqref{BDSOsc} below) and monotonicity properties. On the other hand, concerning right-hand sides,  H\"older inequality yields 
 \begin{equation*}
 	t^{1/d} |\nabla f|^{**}(t) \lesssim t^{-1/p} \Big(\int_0^{t} (u^{1/r} |\nabla f|^*(u))^q \, \frac{du}{u} \Big)^{1/q}.
 \end{equation*}
 Furthermore, it is plain to construct counterexamples showing that the converse estimate fails to be true.
\\
 (iv) Note that \eqref{ThmDeVoreDer<Cubes} is not valid if we replace the right-hand side by the corresponding expression  involving only the last summand with $l=k$.
	\end{rem}

\subsection{Rearrangement inequalities in terms of sharp maximal functions}

To avoid unnecessary technicalities, the results of this section are stated for integrable functions on $Q_0$ but similar local results can be obtained for integrable functions on $\R^d$.
It is well known that
\begin{equation}\label{ProofLem2.1}
(f^{\#}_{Q_0;\mu})^{*}_{\mu}(t) \lesssim f^{**}_\mu(t)
\end{equation}
for $f \in L^1(Q_0, \mu)$, see, e.g.,  \cite{Herz} and  \cite[Theorem 3.8, p. 122]{BennettSharpley}.


In many problems of harmonic analysis the converse estimates to \eqref{ProofLem2.1}
 are of great importance. 
The celebrated Bennett--DeVore--Sharpley inequality \cite[(3.3), p. 605]{BennettDeVoreSharpley} (comp\-lemented in \cite[(3.6), p. 227]{ShagerShvartsman}) partially provides such estimates: if $f \in L_1(Q_0,\mu)$ and $0 < t< \frac{\mu(Q_0)}{6}$ then
\begin{equation}\label{BDSOsc}
	f^{**}_\mu(t) - f^{*}_\mu(t) \lesssim (f^{\#}_{Q_0;\mu})^*_{\mu}(t).
\end{equation}
 In particular, this inequality yields  that $\text{BMO}(Q_0,\mu)$ is contained in the weak $L_\infty(Q_0,\mu)$ space. Furthermore, it can be applied to derive such fundamental results as the John--Nirenberg inequality \cite{JohnNirenberg} and the Fefferman--Stein theorem \cite{FeffermanStein} on the equivalence between the $L_p$-norms of $f$ and $f^{\#}$, as well as various applications in interpolation theory, cf. \cite{BennettSharpley} and \cite{ShagerShvartsman}. As a consequence,  inequality \eqref{BDSOsc} has received a lot of attention over the last few years. In this regard, we mention the work \cite{Lerner05} by Lerner, where he obtained an improvement of \eqref{BDSOsc} for non-doubling measures with the help of centered maximal functions, as well as weighted variants of \eqref{BDSOsc} involving $f^{\#}$ rather than $f^{\#}_{\mu}$. Related inequalities for  the maximal function $M^{\#}_{s,Q_0}$ may be found in \cite{Cwikel, JawerthTorchinsky, Lerner98a}.

In this paper we obtain a logarithmic version of the Bennett--DeVore--Sharpley inequality but, unlike \eqref{BDSOsc}, involving only $f^{*}_\mu$ in the lower bound. Here, we stress that the left-hand side of \eqref{BDSOsc} depends neither on the growth of $f^*_\mu$ nor $f^{**}_\mu$ but rather on the
oscillation, which sometimes causes additional obstacles for applications. Our result reads as follows.


\begin{thm}\label{ThmBDS}
	Let $1 < p < \infty$ and $0 < r \leq \infty$. Assume that $f \in L_{p,r}(Q_0,\mu)$. Then, for each $t \in (0,1)$, we have
	\begin{equation}\label{ThmBDS1}
		\int_0^{t (1-\log t)^{-p}} (u^{1/p} (f-f_{Q_0;\mu})^{*}_{\mu}(u))^r \, \frac{du}{u} \lesssim \int_0^{t} (u^{1/p} (f^{\#}_{Q_0;\mu})^*_\mu(u))^r \, \frac{du}{u}
	\end{equation}
	(with the usual modification if $r=\infty$). Furthermore, this inequality is optimal in the following sense
	\begin{equation}\label{ThmBDS1Optim}
		\int_0^{t (1-\log t)^{-\lambda}} (u^{1/p} (f-f_{Q_0})^{*}(u))^r \, \frac{du}{u} \lesssim  \int_0^{t} (u^{1/p} (f^{\#}_{Q_0})^*(u))^r \, \frac{du}{u} \iff \lambda \geq p.
	\end{equation}
	The same result holds for $M^{\#}_{s, Q_0;\mu}, \, 0 < s\leq s_0,$ where $s_0 > 0$ depends on $d$.
\end{thm}

\begin{rem}\label{Rem1}
	(i) The analogue of \eqref{ThmBDS1} for functions on $\R^d$ reads as follows. Let $1 < p < \infty$ and $0 < r \leq \infty$. Assume $f \in L_{p,r}(\R^d)+ \text{BMO}(\R^d)$. Then, given any cube $Q_0$ in $\R^d$, the following holds
		\begin{equation}\label{757575}
		\int_0^{t (1-\log t)^{-p}} (u^{1/p} ((f-f_{Q_0;\mu}) \chi_{Q_0})^{*}_{\mu}(u))^r \, \frac{du}{u} \lesssim \int_0^{t} (u^{1/p} (f^{\#}_{\mu})^*_\mu(u))^r \, \frac{du}{u}, \quad t \in (0,1).
	\end{equation}
	Note that the equivalence constants behind are independent of $f$ and $t$ (but may depend on $|Q_0|_d$.)
\\	
	(ii)  It is a natural question to investigate the interrelation between our inequality \eqref{ThmBDS1} and good-$\lambda$ inequalities.
 A typical good-$\lambda$ inequality claims that there exists $B > 0$ such that for all $\varepsilon > 0, \lambda > 0$, and all locally integrable function $f$ on $\R^d$, we have
	\begin{equation}\label{GoodLambda}
		\mu \{|f| > B f^{\#} + \lambda\} \leq \varepsilon \mu \{|f| > \lambda\}.
	\end{equation}
	The connection between this good-$\lambda$ inequality and the Bennett--DeVore--Sharpley inequality was shown by
 Kurtz \cite{Kurtz} (with \cite{Bagby} as a forerunner).
   He proved that \eqref{GoodLambda} implies the following variants of \eqref{BDSOsc}:
	\begin{equation*}
		f^*_\mu(t)-f^*_\mu(2t) \leq C (f^{\#}_\mu)^*_\mu\Big(\frac{t}{2}\Big)
	\end{equation*}
	and
	\begin{equation}\label{VariantBDS}
		f^{**}_\mu(t) - f^{*}_\mu(t) \leq 2 B (f^{\#}_{\mu})^{**}_{\mu}(\frac{t}{4}) = 8 B \frac{1}{t} \int_0^{t/4} (f^{\#}_\mu)^*_\mu(u) \,du.
	\end{equation}
	 More general results may be found in \cite{Milman16}.
Further, 
using
 Hardy's inequality and the fact that $f^{**}_\mu(t) = \int_t^\infty (f^{**}_\mu(u) - f^*_\mu(u)) \frac{du}{u}$
 whenever
 $f^{**}_\mu(\infty)=0$ (e.g., if $f \in L_{p}(\R^d,\mu)$ for any $p < \infty$),
\eqref{VariantBDS} yields the following  Fefferman--Stein type inequality
	\begin{equation}\label{p}
		\int_0^\infty (u^{1/p} f^{*}_\mu(u))^r \, \frac{du}{u}  \lesssim  \int_0^{\infty} (u^{1/p} (f^{\#}_{\mu})^*_\mu(u))^r \, \frac{du}{u}, \quad 1 < p < \infty, \quad  r > 0.
	\end{equation}
 Clearly, \eqref{ThmBDS1} also implies the counterpart of \eqref{p} for cubes. Taking $p=r$ and $\mu = |\cdot|_d$ we arrive at  the classical  Fefferman--Stein inequality \cite{FeffermanStein}.
However, the previous argument cannot be applied  to obtain the qualitative estimates of \eqref{p}, 
 i.e., replacing $\int_0^\infty$ by $\int_0^t$. 
 More than that,  in virtue of the sharpness assertion \eqref{ThmBDS1Optim}, the estimate
	\begin{equation}\label{ThmBDS1False}
		\int_0^{t} (u^{1/p} (f-f_{Q_0})^{*}(u))^r \, \frac{du}{u} \lesssim \int_0^{t} (u^{1/p} (f^{\#}_{Q_0})^*(u))^r \, \frac{du}{u}, \quad f \in L_{p,r}(Q_0),
	\end{equation}
	fails to be true. The correct inequality is given by \eqref{ThmBDS1} and involves the additional logarithmic term. Moreover, it is not difficult to see that \eqref{ThmBDS1} implies the weaker estimate
	\begin{equation*}
		\int_0^{t} (u^{1/p} (f-f_{Q_0;\mu})_\mu^{*}(u))^r \, \frac{du}{u} \lesssim \int_0^{t} (u^{1/p} (1-\log u) (f^{\#}_{Q_0;\mu})_\mu^*(u))^r \, \frac{du}{u}, \quad f \in L_{p,r}(Q_0).
	\end{equation*}	
	(iii) For  $p > 1$, by Hardy's inequality,  \eqref{ThmBDS1} can be equivalently written as
		\begin{equation*}
		\int_0^{t (1-\log t)^{-p}} (u^{1/p} (f-f_{Q_0;\mu})^{**}_\mu(u))^r \, \frac{du}{u} \lesssim \int_0^{t} (u^{1/p} (f^{\#}_{Q_0;\mu})_\mu^{*}(u))^r \, \frac{du}{u}, \quad f \in L_{p,r}(Q_0).
	\end{equation*}
This estimate is sharp.
Indeed assume, without loss of generality, that $\mu = |\cdot|_d$, and define $f(u) = |u|^{-d/p} (1-\log |u|)^{-\varepsilon} + f_{Q_0}, \, u \in Q_0 = [-1,1]^d$, where $\varepsilon > 1/r$. Elementary computations show that $(f-f_{Q_0})^*(t) \asymp t^{-1/p} (1 - \log t)^{-\varepsilon}$ for $t$ sufficiently small and thus
	\begin{equation*}
		\int_0^{t(1-\log t)^{-p}} (u^{1/p} (f-f_{Q_0})^{**}(u))^r \, \frac{du}{u} \asymp (1-\log t)^{-\varepsilon r + 1}
	\end{equation*}
	and, by \eqref{ProofLem2.1},
	\begin{equation*}
		 \int_0^{t} (u^{1/p} (f^{\#}_{Q_0})^*(u))^r \, \frac{du}{u} \lesssim \int_0^{t} (u^{1/p} f^{**}(u))^r \, \frac{du}{u}   \asymp (1-\log t)^{-\varepsilon r + 1}.
	\end{equation*}	
	(iv) A weaker estimate than (\ref{BDSOsc}), namely $f^{**}(t) \lesssim \int_t^\infty (f^{\#}(u))^* \frac{du}{u}$ whenever
$f \in L_1(\R^d) +  L_\infty(\R^d)$ and
$f^{**}(\infty)=0$ was obtained in \cite[(4.15)]{BennettSharpley79} (cf. \cite[Chapter 5, Corollary 7.4, p. 379]{BennettSharpley} for its analogue for cubes). However, this inequality does not yield optimal estimates even for smooth functions. More precisely, for any non-zero $f \in C^\infty_0(\R^d),$ we have $f^{**}(t) \to \|f\|_{L_\infty(\R^d)}$ and $\int_t^\infty (f^{\#}(u))^* \frac{du}{u} \to  \infty$ as $t \to 0+$.
 Note that this obstruction is not observed in  inequality \eqref{ThmBDS1}.

\end{rem}

We also establish the endpoint case $p=\infty$ in Theorem \ref{ThmBDS} with the help of the maximal function $\overline{M}^{\#}_{s, Q_0} f$.

\begin{thm}\label{ThmGJ}
	If $f \in \emph{BMO}(Q_0)$, then
	\begin{equation}\label{ThmGJInequal}
		\sup_{0 < u < t}  (1-\log u)^{-1} (f-f_{Q_0})^{**}(u) \lesssim
(1-\log t)^{-1} \|\overline{M}^{\#}_{t, Q_0} f \|_{L_\infty(Q_0)}, \quad t \in (0,1).
	\end{equation}
\end{thm}

\begin{rem}\label{RemLim}
 (i)
 Inequality \eqref{ThmGJInequal} is optimal in the sense that the equivalence holds for a certain function in $\text{BMO}(Q_0)$. To prove this assertion, we make use of the following estimate, which is an immediate consequence of the John--Nirenberg theorem. If $f \in \text{BMO}(Q_0)$, then
	\begin{equation*}
		\|\overline{M}^{\#}_{t, Q_0} f \|_{L_\infty(Q_0)} \lesssim (-\log t) \|f\|_{\text{BMO}(Q_0)}.
	\end{equation*}
	Consider $f(x) = |\log |x||, \, x \in Q_0 = [-1,1]^d$.
Since $f \in  \text{BMO}(Q_0)$ 
and
 $f^{**}(t) \asymp (-\log t)$, it follows that both sides in  \eqref{ThmGJInequal} coincide.
\\	
	(ii)  Let us now show that \eqref{ThmGJInequal} provides a much stronger estimate than 
the inequality
   $(f-f_{Q_0})^{**}(t) \lesssim  \|\overline{M}^{\#}_{t, Q_0} f \|_{L_\infty(Q_0)}$ (without $\sup_{0 < u < t}$). 
   To be more precise, for each $t \in (0,1)$ there exists $f \in \text{BMO}(Q_0)$ (we may assume without loss of generality that $f_{Q_0} = 0$) such that
	 	\begin{equation*}
		\sup_{0 < u < t}  (1-\log u)^{-1} f^{**}(u) \asymp
(1-\log t)^{-1} \|\overline{M}^{\#}_{t, Q_0} f \|_{L_\infty(Q_0)}
	\end{equation*}
	but
	\begin{equation*}
		f^{**}(t) \asymp 1 , \quad \|\overline{M}^{\#}_{t, Q_0} f \|_{L_\infty(Q_0)} \asymp (- \log t).
	\end{equation*}
	Indeed, consider
	\begin{equation*}
		f_0(u) = \left\{\begin{array}{lcl}
                            (-\log t) & ,  & 0 < u  \leq \frac{t}{2}, \\
                            & & \\
                            \big(1- \frac{2 (1-(-\log t)^{-1})}{t} (u - \frac{t}{2})\big) (-\log t)& , & \frac{t}{2} < u < t, \\
                            & & \\
                            0 &, & t \leq u \leq 1,
            \end{array}
            \right.
	\end{equation*}
	and let $f$ be such that $f^{**}(u) \asymp f_0(u)$. We have
	\begin{align*}
		\sup_{0 < u < t}  (1-\log u)^{-1} f^{**}(u) &\asymp \sup_{0 < u < \frac{t}{2}} (-\log u)^{-1} (-\log t) + \\
		&\hspace{1cm} \sup_{\frac{t}{2} < u < t}   \big(1- \frac{2 (1-(-\log t)^{-1})}{t} (u - \frac{t}{2})\big) \asymp 1
	\end{align*}
	and thus, by \eqref{ThmGJInequal},
	\begin{equation*}
		(-\log t) \lesssim \|\overline{M}^{\#}_{t, Q_0} f \|_{L_\infty(Q_0)}  \lesssim \|f\|_{L_\infty(Q_0)} \asymp (-\log t).
	\end{equation*}

\end{rem}

\bigskip
\section{Applications and discussions}

\subsection{Fefferman--Stein's inequality}\label{SectionFS}
As usual, a weight is a non-negative integrable function $w$ on $Q_0$. Given a measurable set $E \subset Q_0$, let $w (E) = \int_E w(x) \, dx$. We say that $w$ belongs to the $A_\infty(Q_0)$ class if there exist positive constants $C_w$ and  $\delta$ such that for all cubes $Q \subseteq Q_0$ and all measurable $E \subseteq Q$ we have
$$
	\frac{w(E)}{w(Q)} \leq C_w \Big(\frac{|E|_d}{|Q|_d}\Big)^\delta.
$$
It is well known that if $w \in A_\infty(Q_0)$ then $w$ is doubling; also $L_\infty(Q_0, w) = L_\infty(Q_0)$.

The Fefferman--Stein inequality \cite{FeffermanStein} asserts that
\begin{equation}\label{FSNew}
 \inf_{c \in \R}\|f-c\|_{L_p(Q_0,w)} \lesssim \|f^{\#}_{Q_0}\|_{L_p(Q_0,w)}, \quad 0 < p < \infty, \quad w \in A_\infty(Q_0);
 \end{equation}
see also \cite{Stromberg}. The stronger version of this inequality, which is obtained by replacing $f^{\#}_{Q_0}$ by $M^{\#}_{s,Q_0} f$, also holds true (cf. \cite{JawerthTorchinsky, Lerner98}). Inequality \eqref{FSNew} plays a central role in Fourier analysis (see, e.g., \cite{Grafakos, Stein}) and interpolation theory (see \cite{BennettSharpley}). In particular, it is strongly related to Coifman--Fefferman inequalities, 
cf. \cite{Kurtz,Lerner, LernerPre1, LernerPre2}. To obtain
\eqref{FSNew}, one can use various methods including 
 duality arguments \cite{Stein}, rearrangement inequalities \cite{BennettSharpley79, DeVoreSharpley}, good-$\lambda$ inequalities \cite{Grafakos}, or Garsia--Rodemich spaces \cite{AstashkinMilman}.

Recently, great interest has been directed toward studies
 of \eqref{FSNew} in more ge\-ne\-ral function spaces, see \cite{AstashkinMilman,Lerner, LernerPre2}. In particular, it was shown  that
the Coifman--Fefferman inequality is equivalent to the Fefferman--Stein inequality on a certain class of Banach function spaces. Moreover,  the Fefferman--Stein inequality holds on r.i. spaces if and only if their lower Boyd index is positive.
  These techniques however  cannot be applied to investigate function spaces that are close to $L_\infty(Q_0)$ (i.e., the lower Boyd index is $0$).
  A natural example of such spaces, widely used in  harmonic analysis, is the Lorentz--Zygmund spaces $L_{\infty, q}(\log L)_b(Q_0, w)$ and, in particular, the exponential classes $\text{exp}\, L^{\lambda}(Q_0,w)$ (see Section \ref{SectionFunctionSpaces}).
  Clearly,  inequality \eqref{FSNew} cannot be true  replacing $L_p(Q_0,w)$ by  $L_{\infty, q}(\log L)_b(Q_0, w)$.

Below we  answer the following question:
 What is the best possible target space $\mathbb{X}= \mathbb{X}(Q_0)$ within  Lorentz--Zygmund spaces so that the  inequality
\begin{equation*}
 \inf_{c \in \R}\|f-c\|_{\mathbb{X}} \le C \|f^{\#}_{Q_0}\|_{L_{\infty, q}(\log L)_b(Q_0, w)}
 \end{equation*}
 holds?

\begin{cor}\label{TheoremSharpLimiting}
	Let $0 < q \leq \infty$ and $b < -1/q$. Assume $w \in A_\infty(Q_0)$ and $f \in L_p(Q_0, w)$ for some $1 < p < \infty$. Then
	\begin{equation}\label{FSLim}
	\inf_{c \in \R} \bigg(\int_0^1 (1 - \log t)^{b q} \Big(\sup_{t <u <1} (1-\log u)^{-1} (f-c)_w^*(u) \Big)^{q} \frac{dt}{t} \bigg)^{1/q} \lesssim \|f^{\#}_{Q_0}\|_{L_{\infty,q}(\log L)_b(Q_0,w)}.
	\end{equation}
	In particular, we have
	\begin{equation*}
	\inf_{c \in \R} \|f-c\|_{L_{\infty,q}(\log L)_{b-1}(Q_0,w)} \lesssim \|f^{\#}_{Q_0}\|_{L_{\infty,q}(\log L)_b(Q_0,w)}
	\end{equation*}
	and
	\begin{equation*}
		\inf_{c \in \R} \|f-c\|_{\emph{exp} \, L^{\frac{\lambda}{\lambda+1}} (Q_0,w)} \lesssim \|f^{\#}_{Q_0}\|_{\emph{exp} \, L^\lambda (Q_0,w) }, \quad \lambda >0.
	\end{equation*}
	The corresponding result for $M^{\#}_{s,Q_0} f$ with $s$ small enough also holds true.
\end{cor}
\begin{rem}
(i) The norm given in the left-hand side of \eqref{FSLim} appears frequently in the study of optimal embedding theorems in limiting cases (cf. \cite{BennettRudnick}, \cite{Pustylnik} and \cite{DominguezHaroskeTikhonov}) and interpolation theorems (cf. \cite{EvansOpic, EvansOpicPick}).

(ii) The sharp version of Corollary  \ref{TheoremSharpLimiting}  (see Proposition \ref{PropOptimFS} below) states that
  for
$0 < q \leq \infty$ and $b < -1/q$  the inequality
	\begin{equation*}
			\inf_{c \in \R}\left(\int_0^1 (1 - \log t)^{b q} \left(\int_t^1 (u^{1/p} (1-\log u)^{\xi} (f-c)_w^*(u))^r \frac{du}{u} \right)^{q/r} \frac{dt}{t} \right)^{1/q}  \lesssim \|f^{\#}_{Q_0}\|_{L_{\infty,q}(\log L)_b(Q_0,w)}
			\end{equation*}
holds			if and only if
			\begin{equation*}
                           \left\{\begin{array}{lcl}
                            p < \infty, & r \leq \infty,  & -\infty < \xi < \infty, \\
                            & & \\
                            p=\infty,& r < \infty, & \xi < -1 - \frac{1}{r}, \\
                            & & \\
                            p=\infty, & r= \infty, & \xi \leq -1.

            \end{array}
           \right.
	\end{equation*}
\end{rem}

\subsection{Calder\'on--Scott type results}
The strong connection between estimates of maximal functions of smooth functions and Sobolev embeddings was established by Calder\'on and Scott
\cite{CalderonScott} and further developed in DeVore and Sharpley \cite{DeVoreSharpley}. Let us illustrate it 
 with a simple example given in \cite[p. 84]{CalderonScott}. In order to guarantee that $f^{\#} \in L_{\xi,r}(\R^d),\, 1 <\xi < \infty, 0 < r \leq \infty$, it suffices to assume that $f \in \dot{B}^{d (\frac{1}{p} - \frac{1}{\xi})}_{p, r}(\R^d)$ where $1 < p < \xi$. Moreover, we have
 \begin{equation}\label{Blowupnew}
 	\|f^{\#}\|_{L_{\xi,r}(\R^d)} \leq C(\xi) \|f\|_{\dot{B}^{d (\frac{1}{p} - \frac{1}{\xi})}_{p, r}(\R^d);k}, \quad k > d \Big(\frac{1}{p} - \frac{1}{\xi}\Big),
 \end{equation}
 where $C(\xi)$ is a positive constant which depends, among other parameters, on $\xi$ but is independent of $f$. Indeed, this is an immediate consequence of \eqref{ProofLem2.1} and the classical Sobolev inequality for Besov functions $\|f\|_{L_{\xi,r}(\R^d)} \lesssim \|f\|_{\dot{B}^{d (\frac{1}{p} - \frac{1}{\xi})}_{p, r}(\R^d);k}$. Furthermore, the same argument but now invoking the Sobolev inequality for Lorentz--Besov spaces
 (cf. \cite{Martin}) shows that \eqref{Blowupnew} also holds true when the Besov space $\dot{B}^{d (\frac{1}{p} - \frac{1}{\xi})}_{p, r}(\R^d)$ is replaced by $\dot{B}^{d (\frac{1}{p} - \frac{1}{\xi})}_{r} L_{p,q}(\R^d), \, 0 < q \leq \infty$. Note that $\dot{B}^{d (\frac{1}{p} - \frac{1}{\xi})}_{p, r}(\R^d) \subsetneq \dot{B}^{d (\frac{1}{p} - \frac{1}{\xi})}_{r} L_{p,q}(\R^d), \, q > p$.

The corresponding analysis for function spaces that are close to $L_\infty$, for example $\text{exp} \, L^\lambda$ or more generally $L_{\infty, r}(\log L)_b$, is much more delicate. Among other obstructions, it is convenient to switch from homogeneous Besov spaces on $\R^d$ to inhomogeneous  Besov spaces on bounded domains. Otherwise, the expected counterpart of \eqref{Blowupnew} will no longer involve only a classical function norm on the left-hand side, but a sum of function norms; this phenomenon already occurs in Sobolev inequalities for limiting cases (cf. \cite[Theorem 8.1]{DeVoreRiemenschneiderSharpley}). Clearly,  inequality \eqref{Blowupnew} does not hold for functions in $\dot{B}^{d (\frac{1}{p} - \frac{1}{\xi})}_{p, r}(Q_0)$ (consider, e.g., polynomials) so that one is forced to work with its inhomogeneous  counterpart $B^{d (\frac{1}{p} - \frac{1}{\xi})}_{p, r}(Q_0)$. Furthermore, we note that the Fefferman--Stein inequality \eqref{FSNew} does not hold for $p=\infty$; we only have the trivial reverse estimate $\|f^{\#}\|_{L_{\infty,r}(\log L)_b(Q_0)} \lesssim \|f\|_{L_{\infty,r}(\log L)_b(Q_0)}$. Then, following a similar reasoning as the one given above for the non-limiting case (i.e., $\xi < \infty$) but now relying on the corresponding Sobolev inequality (cf. \cite[Corollary 5.5]{DeVoreRiemenschneiderSharpley} and \cite[Theorem 2]{Martin})
		\begin{equation}\label{Blowup3}
			\|f\|_{L_{\infty, r}(\log L)_b(Q_0)} \lesssim \|f\|_{B^{d/p, b+ 1/\min\{1,r\}}_{p,r} (Q_0);k} \quad b < -1/r,	\quad k > d/p,
		\end{equation}
		we derive
		\begin{equation}\label{Blowup4}
		\|f^{\#}_{Q_0}\|_{L_{\infty, r}(\log L)_b(Q_0)} \lesssim \|f\|_{B^{d/p, b+ 1/\min\{1,r\}}_{p,r} (Q_0);k}.
			\end{equation}
			This implies a loss of logarithmic smoothness of order $\frac{1}{\min\{1,r\}}$ in order to achieve that $f^{\#}_{Q_0} \in L_{\infty, r}(\log L)_b(Q_0)$.
			
			One of our goals in this paper is to
 show that the standard methods  described above, which reduce estimates for maximal functions (cf. \eqref{Blowup4}) to the Sobolev inequalities (cf. \eqref{Blowup3}), are far from being optimal and can be  considerably improved by using new extrapolation estimates based on Theorem \ref{ThmDeVoreLorentz}. These extrapolation results are interesting by their own sake and we postpone their detailed discussion to Section \ref{SectionExtrapol} below. As application of these extrapolation arguments, we are in a position to improve
 \eqref{Blowup4}.
		
		\begin{cor}\label{CorollaryLimitingBesovMax}
			Let $1 < p < \infty, 0 < q, r \leq \infty, k > d/p$, and $b < -1/r$. Then
			\begin{equation}\label{CorollaryLimitingBesovMax1New*}
			\|f^{\#}_{Q_0}\|_{L_{\infty, r} (\log L)_b (Q_0)} \lesssim \|f\|_{B^{d/p, b+1/r}_{r} L_{p,q}(Q_0);k}.
			\end{equation}
			In particular,
			\begin{equation}\label{CorollaryLimitingBesovMax1}
			\|f^{\#}_{Q_0}\|_{L_{\infty, r} (\log L)_b (Q_0)} \lesssim \|f\|_{B^{d/p, b+1/r}_{p, r}(Q_0);k}
			\end{equation}
			and if $\lambda > 0$ then
			\begin{equation*}
			\|f^{\#}_{Q_0}\|_{\emph{exp} \, L^\lambda (Q_0)} \lesssim \|f\|_{B^{d/p,-1/\lambda}_{p, \infty}(Q_0);k}.
			\end{equation*}
		\end{cor}
	 \begin{rem}
	(i) Since $B^{d/p, b+1}_{p,r}(Q_0) \subsetneq B^{d/p, b+1/r}_{p, r}(Q_0), \, r > 1$, \eqref{CorollaryLimitingBesovMax1} sharpens
   \eqref{Blowup4}.
	\\
	 (ii) We will show in Proposition \ref{CorollaryLimitingBesovMaxOptimal} below that \eqref{CorollaryLimitingBesovMax1New*} is optimal, namely
	 \begin{equation*}
	 \|f^{\#}_{Q_0}\|_{L_{\infty, r} (\log L)_b (Q_0)} \lesssim \|f\|_{B^{d/p, b+\xi}_{r} L_{p,q}(Q_0);k } \iff \xi \geq 1/r.
	 \end{equation*}
	 \end{rem}

	 Our method can  be also applied to work with Sobolev spaces. Let us state the analogue of Corollary \ref{CorollaryLimitingBesovMax}.
	
	\begin{cor}\label{CorollaryLimitingBesovMaxSecond}
	Let $1 \leq r \leq  \infty, k < d$, and $b < -1/r$. Then
			\begin{equation*}
			\|f^{\#}_{Q_0}\|_{L_{\infty, r} (\log L)_b (Q_0)} \lesssim \|f \|_{W^k L_{d/k, r} (\log L)_{b + 1/r}(Q_0)}.
			\end{equation*}
		\end{cor}
		
		\begin{rem}
		(i) In Proposition \ref{CorollaryLimitingBesovMaxSecondOptimal} below, we will establish the optimality of the previous inequality, i.e.,
	 \begin{equation*}
	 	\|f^{\#}_{Q_0}\|_{L_{\infty, r} (\log L)_b (Q_0)} \lesssim \|f \|_{W^k L_{d/k, r} (\log L)_{b + \xi}(Q_0)} \iff \xi \geq 1/r.
	 \end{equation*}
(ii) Comparing Corollaries \ref{CorollaryLimitingBesovMax} and \ref{CorollaryLimitingBesovMaxSecond}, we note that to the best of our knowledge the relationship  between the spaces $W^k L_{d/k,r}(\log L)_{b+1/r}(Q_0)$ and $B^{d/p,b+1/r}_{r} L_{p,q}(Q_0)$ is not known. See  \cite{SeegerTrebels}.
		\end{rem}

\subsection{Extrapolation results}\label{SectionExtrapol}
A natural question in harmonic analysis is to look for sharp bounds of the norms of classical operators in terms of some of the involved parameters (integrability, smoothness, $A_p$ characteristic of weights, etc). This question is not only interesting by itself (cf. \cite{HytonenPerez}) but it is also useful to establish borderline estimates via extrapolation methods (cf. \cite{JawerthMilman}). 

The aim of this section is to apply the pointwise estimates obtained in Section \ref{SectionEstimSmooth} to derive several sharp estimates involving integrability properties of $f^{\#}$. As anticipated above, these estimates will be essential in the proofs of Corollaries \ref{CorollaryLimitingBesovMax} and \ref{CorollaryLimitingBesovMaxSecond}.

According to \eqref{Blowupnew}, we have that
\begin{equation}\label{Blowup2}
			\|f^{\#}\|_{L_{d/\varepsilon, r}(\R^d)} \leq C(\varepsilon) \|f\|_{\dot{B}^{d/p-\varepsilon}_{p, r}(\R^d);k}, \quad \varepsilon \to 0+.
		\end{equation}
	Invoking now Theorem \ref{ThmDeVoreLorentz} we are able to determine the exact behaviour of the constant $C(\varepsilon)$.

	\begin{cor}\label{CorDeVoreExtrapol}
		Let $1 < p < \infty, 0 < q, r \leq \infty,$ and $k \in \N$. Then
	\begin{equation}\label{CorDeVoreExtrapol1}
		\|f^{\#}\|_{L_{d/\varepsilon, r}(\R^d)} \leq C_0 \,  \varepsilon^{-1/r} \|f\|_{\dot{B}^{d/p-\varepsilon}_r L_{p, q}(\R^d); k}, \qquad \varepsilon \to 0+, \qquad k > d/p,
	\end{equation}
	and
	\begin{equation}\label{CorDeVoreExtrapol2}
		\|f^{\#}\|_{L_{d/\varepsilon, r}(\R^d)} \leq C_1 \|f\|_{\dot{B}^{d/p-\varepsilon}_r L_{p, q}(\R^d); k}, \qquad \varepsilon \to 0+, \qquad k = d/p,
	\end{equation}
	where $C_0$ and $C_1$ are positive constants which do not depend on $f$ and $\varepsilon$. The corresponding estimates for cubes read as follows
		\begin{equation}\label{CorDeVoreExtrapol1cube}
		\|f^{\#}_{Q_0}\|_{L_{d/\varepsilon, r}(Q_0)} \leq C_2 \,  \varepsilon^{-1/r} \|f\|_{B^{d/p-\varepsilon}_r L_{p, q}(Q_0); k}, \qquad \varepsilon \to 0+, \qquad k \geq d/p,
	\end{equation}
	where $C_2$ does not depend on $f$ and $\varepsilon$.
	\end{cor}

\begin{rem}	
	(i) The previous result tells us that the asymptotic behaviour of $C(\varepsilon)$ in \eqref{Blowup2} strongly depends on the order $k$ of the fixed Besov (semi-)norm. On the one hand, in the limiting case $k=d/p$ by \eqref{CorDeVoreExtrapol2}, $C(\varepsilon) = O(1)$. This corresponds to the fact that if $\varepsilon \to 0+$ then we approach the  spaces $L_{\infty,r}(\R^d)$ and
		\begin{equation*}
			\Big\{f \in C^\infty_0(\R^d) : \int_0^\infty (t^{-k} \omega_k(f,t)_{p,q})^r \frac{dt}{t} < \infty \Big\},
		\end{equation*}
		which are both  trivial when $r < \infty$; while if $r=\infty$ one can easily check that  inequality \eqref{CorDeVoreExtrapol2} can be rewritten as $\|f\|_{\text{BMO}(\R^d)} \lesssim \||\nabla^k f |\|_{L_{d/k, q}(\R^d) }$ (cf. \eqref{LemmaEmbBMO}). A similar comment also applies to \eqref{CorDeVoreExtrapol1} (i.e., $k > d/p$) with $r= \infty$ (cf. \eqref{LemmaEmbBMOLorentz}.)

On the other hand, a completely different phenomenon arises when $k > d/p$ and $r < \infty$ (see \eqref{CorDeVoreExtrapol1}). In this case, the blow up $C(\varepsilon) = O (\varepsilon^{-1/r})$ is in fact sharp since
(see Proposition \ref{CorDeVoreExtrapolSharp} below) 
		\begin{equation*}
		\|f^{\#}_{Q_0}\|_{L_{d/\varepsilon, r}(Q_0)} \leq C_0 \,  \varepsilon^{-\xi} \|f\|_{B^{d/p-\varepsilon}_r L_{p, q}(Q_0); k} \iff \xi \geq 1/r.
	\end{equation*}	
	(ii) Estimate \eqref{CorDeVoreExtrapol1cube} is formulated in terms of the inhomogeneous  Besov norms rather than their homogeneous counterparts (see \eqref{CorDeVoreExtrapol1} and \eqref{CorDeVoreExtrapol2}). This modification is indeed necessary to obtain meaningful estimates.
Moreover, the limiting case $k=d/p$ shows another distinction between Besov norms on $\R^d$ and cubes. Specifically, inequality \eqref{CorDeVoreExtrapol1cube} with $k=d/p$ is considerably worse than \eqref{CorDeVoreExtrapol2}.
		\end{rem}
	
	The counterpart of Corollary \ref{CorDeVoreExtrapol} for Sobolev spaces reads as follows.
	
		\begin{cor}\label{CorDeVoreDerExtrapol}
		Let $1 \leq r \leq \infty$ and $k < d$. Then
			\begin{equation}\label{CorDeVoreDerExtrapol1}
		\|f^{\#}\|_{L_{d/\varepsilon, r}(\R^d)} \leq C_0 \,  \varepsilon^{-1/r} \|f \|_{\dot{W}^k L_{d/(k+\varepsilon), r} (\R^d)}, \qquad \varepsilon \to 0+,
	\end{equation}
	and
	\begin{equation}\label{CorDeVoreDerExtrapol1Cubes}
		\|f^{\#}_{Q_0}\|_{L_{d/\varepsilon, r}(Q_0)} \leq C_1 \,  \varepsilon^{-1/r} \|f \|_{W^k L_{d/(k+\varepsilon), r} (Q_0)}, \qquad \varepsilon \to 0+,
	\end{equation}
	where $C_0$ and $C_1$ are positive constants which do not depend on $f$ and $\varepsilon$.
	\end{cor}
	
	\begin{rem}
	(i) The assumption $k < d$ guarantees that the right-hand sides of \eqref{CorDeVoreDerExtrapol1} and \eqref{CorDeVoreDerExtrapol1Cubes} are well defined for small enough $\varepsilon$.
\\	
	(ii) Letting  $C(\varepsilon) = \sup_{\|f\|_{W^k L_{d/(k+\varepsilon), r} (Q_0)} \leq 1} \|f^{\#}_{Q_0}\|_{L_{d/\varepsilon, r}(Q_0)}$ and $r < \infty$, the fact that  $C(\varepsilon) \to \infty$ as $\varepsilon \to 0+$ is clear since $L_{\infty, r}(Q_0)$ becomes the trivial space. The novelty of Corollary \ref{CorDeVoreDerExtrapol} is to establish the blow-up $C(\varepsilon) = O (\varepsilon^{-1/r})$. Furthermore, by Proposition \ref{CorDeVoreDerExtrapolOptim} below this estimate is optimal, that is, 
	\begin{equation*}
		\|f^{\#}_{Q_0}\|_{L_{d/\varepsilon, r}(Q_0)} \leq C \,  \varepsilon^{-\xi} \|f\|_{W^k L_{d/(k+\varepsilon), r} (Q_0)} \iff \xi \geq 1/r.
	\end{equation*}
	 On the other hand, note that $C(\varepsilon)$ is uniformly bounded  if $r=\infty$ (see \eqref{CorDeVoreDerExtrapol1}). This corresponds to the fact that $\|f\|_{\text{BMO}(Q_0)} \lesssim \|f\|_{W^k L_{d/k,\infty}(Q_0)}$ (cf. \eqref{LemmaEmbBMO} and \eqref{0030030} below.)
	\end{rem}

\bigskip
\section{Proofs of Theorems \ref{ThmBDS} and \ref{ThmGJ}}\label{section5}

Proofs of Theorems \ref{ThmBDS} and \ref{ThmGJ} are based on a combination of known Fourier analytic tools such as the John--Nirenberg inequality and the Garnett--Jones theorem on $\text{BMO}(Q_0)$ as well as  new arguments involving  limiting interpolation techniques (see \eqref{DefLimInt}) and Holmstedt's reiteration formulas. In particular, the following charac\-te\-ri\-za\-tion of the Lorentz--Zygmund space $L_{\infty,q} (\log L)_{b}(Q_0)$ as a limiting interpolation space will be useful.

\begin{lem}\label{LemInterp1}
Let $0 < p < \infty, 0 < q, r \leq \infty,-\infty < c < \infty$, and $b < -1/q \, (b \leq 0 \text{ if } q=\infty)$. Then we have
	\begin{equation*}
		 L_{\infty,q} (\log L)_{b}(Q_0,\mu) = (L_{p,r}(\log L)_c(Q_0,\mu), L_\infty(Q_0,\mu))_{(1,b),q}
	\end{equation*}
	with equivalent quasi-norms. The corresponding result for periodic spaces also holds true.
\end{lem}
\begin{proof}
	Assume first that $c=0$. Inserting the well-known estimate
	\begin{equation}\label{KFunctLorentz}
		K(t, f; L_{p,r}(Q_0,\mu), L_\infty(Q_0,\mu)) \asymp \left(\int_0^{t^p} (u^{1/p} f_\mu^*(u))^r \frac{du}{u} \right)^{1/r}
	\end{equation}
	(see \cite[Theorem 4.2]{Holmstedt}) into the definition of the limiting interpolation space, we obtain
	\begin{align*}
		\|f\|_{(L_{p,r}(Q_0,\mu), L_\infty(Q_0,\mu))_{(1,b),q}} & \asymp \left(\int_0^1 t^{-q/p} (1-\log t)^{b q} \left(\int_0^{t} (u^{1/p} f_\mu^*(u))^r \frac{du}{u} \right)^{q/r} \frac{dt}{t} \right)^{1/q} \\
		& \asymp \left( \int_0^1 ((1-\log t)^b f_\mu^*(t))^q \frac{dt}{t}\right)^{1/q} = \|f\|_{L_{\infty,q}(\log L)_b(
		Q_0,\mu)},
	\end{align*}
	where the last equivalence follows from Hardy's inequality \cite[Theorem 6.4]{BennettRudnick} and the monotonicity of $f_\mu^*$.
	
	The general case ($c \in \R$) can be reduced to the previous one via the trivial embeddings
	\begin{equation*}
		L_{p_0}(Q_0,\mu) \hookrightarrow L_{p,r}(\log L)_c(Q_0,\mu) \hookrightarrow L_{p_1}(Q_0,\mu), \quad 0 < p_1 < p < p_0 < \infty.
	\end{equation*}
\end{proof}


\begin{lem}[\cite{EvansOpicPick}]\label{LemInterp2}
	 Let $(A_0,A_1)$ be a quasi-Banach pair with $A_1 \hookrightarrow A_0$. Let $K(t,f) = K(t, f; A_0, A_1), \, 0 < t < 1$. If $0 < q \leq \infty$ and $b < -1/q \, (b \leq 0 \text{ if } q= \infty)$, then
	\begin{align}
	K(t(1-\log t)^{-b-1/q}, f; A_0, (A_0, A_1)_{(1,b),q}) & \nonumber\\
	&\hspace{-5cm} \asymp K(t,f) + t (1-\log t)^{-b-1/q} \Big(\int_t^1 (u^{-1} (1 - \log u)^b K(u,f))^q \frac{du}{u} \Big)^{1/q} \label{LemInterp2.1}
	\end{align}
	and
	\begin{equation}\label{LemInterp2.2}
	K((1-\log t)^{b+1/q}, f; (A_0, A_1)_{(1,b),q}, A_1) \asymp  \Big(\int^t_0 (u^{-1} (1 - \log u)^b K(u,f))^q \frac{du}{u} \Big)^{1/q}
	\end{equation}
	(with the usual modifications if $q=\infty$).
\end{lem}

We are now ready to give the

\begin{proof}[Proof of Theorem \ref{ThmBDS}]
	In light if  the John--Nirenberg inequality \cite{JohnNirenberg}, we have
	\begin{equation}\label{JNInequality}
		\|f-f_{Q_0;\mu}\|_{\text{exp} \, L (Q_0,\mu)} \lesssim \|f\|_{\text{BMO}(Q_0,\mu)}
	\end{equation}	which yields
	\begin{equation}\label{ProofThmBDS1}
		K(t,f-f_{Q_0;\mu}; L_{p,r}(Q_0,\mu), \text{exp} \, L (Q_0,\mu)) \lesssim K(t, f; L_{p,r}(Q_0,\mu), \text{BMO}(Q_0,\mu)).
	\end{equation}
	Since $\text{exp} \, L (Q_0,\mu) =  L_\infty (\log L)_{-1}(Q_0,\mu)$ (see Section \ref{SectionFunctionSpaces}), we can apply Lemma \ref{LemInterp1} and relation \eqref{LemInterp2.1} to establish
	\begin{align*}
		K(t (1-\log t), f; L_{p,r}(Q_0,\mu), \text{exp} \, L (Q_0,\mu) ) & \\
		&\hspace{-7cm}\asymp K(t (1-\log t), f; L_{p,r}(Q_0,\mu), (L_{p,r}(Q_0,\mu), L_\infty(Q_0,\mu))_{(1,-1),\infty}) \\
		& \hspace{-7cm} \asymp K(t, f; L_{p,r}(Q_0,\mu), L_\infty(Q_0,\mu)) \\
		& \hspace{-6cm} + t (1-\log t) \sup_{t \leq u \leq 1} u^{-1} (1-\log u)^{-1} K(u,f;L_{p,r}(Q_0,\mu), L_\infty(Q_0,\mu)) \\
		& \hspace{-7cm} \gtrsim K(t, f; L_{p,r}(Q_0,\mu), L_\infty(Q_0,\mu)).
	\end{align*}
This and  \eqref{KFunctLorentz} imply
	\begin{equation}\label{ProofThmBDS2}
	K(t (1-\log t), f; L_{p,r}(Q_0,\mu), \text{exp} \, L (Q_0,\mu))  \gtrsim \left(\int_0^{t^p} (u^{1/p} f_\mu^*(u))^r \frac{du}{u} \right)^{1/r}.
	\end{equation}
	 On the other hand, from \cite[Remark 3.7]{JawerthTorchinsky},
	\begin{equation}\label{JT}
		K(t, f; L_{p,r}(Q_0,\mu), \text{BMO}(Q_0,\mu)) \asymp \left(\int_0^{t^p} (u^{1/p} (f^{\#}_{Q_0;\mu})_\mu^*(u))^r \, \frac{du}{u} \right)^{1/p}.
	\end{equation}
	Combining \eqref{ProofThmBDS1}--\eqref{JT} yields  
 \eqref{ThmBDS1}.

	It remains to show the sharpness assertion \eqref{ThmBDS1Optim}. Suppose that there exists $\lambda > 0$ such that
	\begin{equation*}
		\int_0^{t (1-\log t)^{-\lambda}} (u^{1/p} (f-f_{Q_0})^{*}(u))^r \, \frac{du}{u} \lesssim \int_0^{t} (u^{1/p} f^{\#*}_{Q_0}(u))^r \, \frac{du}{u},
	\end{equation*}
	or, equivalently,
	\begin{equation*}
		\int_0^{t} (u^{1/p} (f-f_{Q_0})^{*}(u))^r \, \frac{du}{u} \lesssim \int_0^{t (1-\log t)^{\lambda}} (u^{1/p} f^{\#*}_{Q_0}(u))^r \, \frac{du}{u}.
	\end{equation*}
	Therefore, using  properties of rearrangements, we obtain
	\begin{align*}
		t^{1/p} (f-f_{Q_0})^*(t) &\asymp \left(\int_{t/2}^{t} (u^{1/p} (f-f_{Q_0})^{*}(u))^r \, \frac{du}{u} \right)^{1/r} \\
		& \lesssim \left(\int_0^{t (1-\log t)^\lambda} (u^{1/p} f^{\#*}_{Q_0}(u))^r \, \frac{du}{u} \right)^{1/r} \\
		& \lesssim t^{1/p} (1-\log t)^{\lambda/p} \|f\|_{\text{BMO}(Q_0)},
	\end{align*}
	which yields that there exist positive constants $c_1, c_2$ such that
	\begin{equation*}
		|\{x \in Q_0 : |f(x)-f_{Q_0}| > \xi\}|_d \leq c_1 e^{-c_2 \big(\frac{\xi}{\|f\|_{\text{BMO}(Q_0)}}\big)^{\frac{p}{\lambda}}}
	\end{equation*}
	The latter embedding implies $\lambda \geq p$ because $\text{exp} \, L (Q_0)$ is the smallest rearrangement invariant space that contains $\text{BMO}(Q_0)$ (cf. \cite{Pustylnik}.)
\end{proof}

	A careful examination of the proof of Theorem \ref{ThmBDS} shows the following
	\begin{thm}
	Let $\lambda > 0$. The following statements are equivalent:
	\begin{enumerate}[\upshape(i)]
	\item John--Nirenberg inequality: for each $f \in \emph{BMO}(Q_0)$, there exist positive constants $c_1$ and  $c_2$ such that
		\begin{equation*}
		|\{x \in Q_0 : |f(x)-f_{Q_0}| > \xi\}|_d \leq c_1 e^{-c_2 \big(\frac{\xi}{\|f\|_{\emph{BMO}(Q_0)}} \big)^{\frac{1}{\lambda}}}.
	\end{equation*}
	\item Oscillation inequality: for $1 < p < \infty$ and $0 < r \leq \infty$ we have
	\begin{equation*}
		\int_0^{t (1-\log t)^{-\lambda p}} (u^{1/p} (f-f_{Q_0})^{*}(u))^r \, \frac{du}{u} \lesssim \int_0^{t} (u^{1/p} f^{\#*}_{Q_0}(u))^r \, \frac{du}{u}, \quad f \in L_{p,r}(Q_0).
	\end{equation*}
	\item $\lambda \geq 1.$
	\end{enumerate}
	\end{thm}
	This characterization fits into the research program developed in Mart\'in, Milman and Pustylnik \cite{MartinMilmanPustylnik} and Mart\'in and Milman \cite{MartinMilman}, where certain Sobolev--Poincar\'e  inequalities for smooth functions are equivalently characterized in terms of oscillation inequalities. Note that $\text{BMO}(Q_0)$ can be considered as the limiting space of the scale of Lipschitz spaces $\text{Lip}^\alpha(Q_0), 0 < \alpha < 1$, via the Campanato--Meyers theorem. For further details, we refer the reader to \cite{DeVoreSharpley} and \cite{MartinMilman14}.

\begin{proof}[Proof of Theorem \ref{ThmGJ}]
	Applying the John--Nirenberg inequality given by \eqref{JNInequality}, we have
	\begin{equation}\label{ThmGJProof1}
		K(t, f-f_{Q_0}; \text{exp} \, L (Q_0), L_\infty(Q_0)) \lesssim K(t, f; \text{BMO}(Q_0), L_\infty(Q_0)).
	\end{equation}
	In light of the Garnett--Jones characterization of the $K$-functional for the couple $(L_\infty(Q_0), \text{BMO}(Q_0))$ (cf. \cite{GarnettJones}, \cite{JawerthTorchinsky} and \cite{Shvartsman})
	\begin{equation*}
		K(t, f;  L_\infty(Q_0), \text{BMO}(Q_0)) \asymp  \|\overline{M}^{\#}_{e^{-t}, Q_0} f \|_{L_\infty(Q_0)}, \qquad t > 1,
	\end{equation*}
	we derive
	\begin{align}
		K(t, f; \text{BMO}(Q_0), L_\infty(Q_0)) & = t K(t^{-1}, f;  L_\infty(Q_0), \text{BMO}(Q_0)) \nonumber \\
		& \hspace{-3.5cm} \asymp t  \big\|\overline{M}^{\#}_{e^{-t^{-1}}, Q_0} f \big\|_{L_\infty(Q_0)}, \quad t \in (0,1). \label{ThmGJProof2}
	\end{align}
	
On the other hand, since $\text{exp} \, L(Q_0) = (L_1(Q_0), L_\infty(Q_0))_{(1,-1), \infty}$ (by Lemma \ref{LemInterp1}), we make use of  estimate \eqref{LemInterp2.2} together with the well-known formula $$K(t,f; L_1(Q_0), L_\infty(Q_0)) = t f^{**}(t)$$ (see, e.g., \cite{BennettSharpley}).
  Namely, we have, for each $t  \in (0,1)$,
	\begin{equation}\label{ThmGJProof3}
		K((1-\log t)^{-1} ,f; \text{exp} \, L(Q_0), L_\infty(Q_0))  \asymp \sup_{0 < u < t}  (1-\log u)^{-1} f^{**}(u).
	\end{equation}
 Inserting  \eqref{ThmGJProof2} and \eqref{ThmGJProof3} into \eqref{ThmGJProof1}, we complete the proof.
\end{proof}

\bigskip
\section{Proofs of Theorems  \ref{ThmDeVoreLorentz}, \ref{LemmaEmbBMOLorentzState} and optimality assertions
}
\label{section4}

To prove Theorem \ref{ThmDeVoreLorentz}, we use 
the  new embeddings between Besov spaces and $\text{BMO}(\R^d)$ given in Theorem \ref{LemmaEmbBMOLorentzState}. These results  
 are closely related to embedding theorems recently obtained by Seeger and Trebels \cite{SeegerTrebels}.

	\begin{proof}[Proof of Theorem \ref{LemmaEmbBMOLorentzState}]
Let $H_1(\R^d)$ be the usual Hardy space. According to \cite[Theorem 1.2]{SeegerTrebels}, we have
\begin{equation}\label{d}
H_1(\R^d) \hookrightarrow \dot{B}^{-d/p}_1 L_{p',1}(\R^d),
\end{equation}
where $\dot{B}^{-d/p}_1 L_{p',1}(\R^d)$ is the Fourier-analytically defined Lorentz--Besov space. It is well known that $\text{BMO}(\R^d)$ is the dual space of $H_1(\R^d)$, i.e., $(H_1(\R^d))' = \text{BMO}(\R^d)$ (see \cite[Theorem 2, p. 145]{FeffermanStein}). On the other hand, we claim that
\begin{equation}\label{d2}
(\dot{B}^{-d/p}_1 L_{p',1}(\R^d))' = \dot{B}^{d/p}_\infty L_{p,\infty}(\R^d).
\end{equation}
Assuming this result momentarily,  embedding  \eqref{LemmaEmbBMOLorentz} follows from \eqref{d} via duality.

 We just outline the proof of \eqref{d2} and leave further details to the interested reader. The claim can be shown by using the retraction method in a similar way as the corresponding duality assertion for classical Besov spaces \cite[Theorem 2.11.2, pp. 178--180]{Triebel83} together with the fact that Fourier--Besov spaces with positive smoothness based on $L_{p,q}(\R^d), \, 1 < p < \infty,$ can be equivalently characterized in terms of the $L_{p,q}(\R^d)$ moduli of smoothness (see the proof of the corresponding result for classical Besov spaces with $p= q$ in \cite[Sections 1.13 and 2.5.1]{Triebel78}, where now  Fourier multiplier assertions for $L_{p,q}(\R^d)$ are covered by interpolation of $L_p(\R^d)$).
	
	Embedding \eqref{LemmaEmbBMO*} is an immediate consequence of \eqref{LemmaEmbBMOLorentz} and the fact
 (see \cite[Theorem 1.2]{SeegerTrebels})
 that $\dot{H}^{d/p} L_{p,\infty}(\R^d) \hookrightarrow \dot{B}^{d/p}_\infty L_{p,\infty}(\R^d)$. In particular, setting $p=d/k > 1$ we obtain \eqref{LemmaEmbBMO}.
	\end{proof}

	\begin{proof}[Proof of Theorem \ref{ThmDeVoreLorentz}]
	\textsc{Case 1: $1 < p < \infty, 0 < q \leq \infty$, and $k > d/p$.} In virtue of the known embedding for Lorentz spaces,  Lemma \ref{LemmaEmbBMOLorentzState} implies that $\dot{B}^{d/p}_\infty L_{p,q}(\R^d) \hookrightarrow \text{BMO}(\R^d), \, 0 <  q \leq \infty$. Then
	\begin{equation}\label{ProofThmDeVoreLorentz1}
		K(t, f; L_{p,q}(\R^d), \text{BMO}(\R^d)) \lesssim K(t, f; L_{p,q}(\R^d), \dot{B}^{d/p}_\infty L_{p,q}(\R^d)).
	\end{equation}
	Next we compute these $K$-functionals. Concerning the left-hand side, by \eqref{JT},
		\begin{equation}\label{ProofThmDeVoreLorentz2}
		K(t, f; L_{p,q}(\R^d), \text{BMO}(\R^d)) \asymp \Big(\int_0^{t^p} (u^{1/p} f^{\#*}(u))^q \frac{du}{u} \Big)^{1/q}.
	\end{equation}
	On the other hand, it follows from
		\begin{equation}\label{ProofLemmaEmbBMOLorentzState1*}
	K(t^k, f; L_{p, q}(\R^d), \dot{W}^k L_{p, q}(\R^d)) \asymp \omega_k(f,t)_{p,q}, \quad f \in L_{p, q}(\R^d),
	\end{equation}
	(cf. \cite[3.9.4]{BrudnyiKrugljak}; see also \cite{GogatishviliOpicTikhonovTrebels} and \cite{MartinMilman14}) that
	\begin{equation}\label{ProofLemmaEmbBMOLorentzState1**}
	\dot{B}^{d/p}_\infty L_{p, q}(\R^d) = (L_{p,q}(\R^d), \dot{W}^k L_{p,q}(\R^d))_{\frac{d}{k p}, \infty}.
	\end{equation}
	Then we apply Holmstedt's reiteration formula \cite[Corollary 2.3, p. 310]{BennettSharpley} to establish
	\begin{align}
		K(t^{d/k p}, f; L_{p,q}(\R^d), \dot{B}^{d/p}_\infty L_{p,q}(\R^d))&\asymp t^{d/k p} \sup_{t^{1/k} < u < \infty} u^{-d/ p} K(u^k, f ; L_{p, q}(\R^d), \dot{W}^k L_{p,q}(\R^d))   \nonumber\\
		& 
 \asymp t^{d/k p} \sup_{t^{1/k} < u < \infty} u^{-d/p} \omega_k (f,u)_{p, q}. \label{ProofThmDeVoreLorentz3}
	\end{align}
	Plugging \eqref{ProofThmDeVoreLorentz2} and \eqref{ProofThmDeVoreLorentz3} into \eqref{ProofThmDeVoreLorentz1}, we derive  the desired estimate \eqref{ThmDeVore*Lorentz}.
	
	\textsc{Case 2: $1 < p < \infty, 0 < q \leq \infty$, and $k=d/p$.} By \eqref{ProofThmDeVoreLorentz2}, Lemma \ref{LemmaEmbBMOLorentzState} and \eqref{ProofLemmaEmbBMOLorentzState1*}, we have
	\begin{align*}
		\Big(\int_0^{t^d} (u^{k/d} f^{\#*}(u))^q \frac{du}{u} \Big)^{1/q} &\asymp K(t^k, f; L_{d/k, q}(\R^d), \text{BMO}(\R^d)) \\
		&\hspace{-2.5cm} \lesssim K(t^k, f; L_{d/k, q}(\R^d), \dot{W}^k L_{d/k, q}(\R^d)) \asymp \omega_k (f, t)_{d/k, q}.
	\end{align*}
	
	\textsc{Case 3: $1 < p < \infty, 0 < q \leq \infty$ and $k < d/p$.} In light of the Sobolev inequality (see, e.g.,  \cite[Theorem 2]{Milman})
	\begin{equation*}
		\dot{W}^k L_{p,q}(\R^d) \hookrightarrow L_{p^*,q}(\R^d), \qquad p^* = \frac{d p}{d-k p},
	\end{equation*}
	we obtain
	\begin{equation}\label{hfhfhhf}
		K(t^k,f; L_{p,q}(\R^d), L_{p^*,q}(\R^d)) \lesssim K(t^k,f; L_{p,q}(\R^d), \dot{W}^k L_{p,q}(\R^d)).
	\end{equation}
	Inserting the Holmstedt's formula \cite[Theorem 4.2]{Holmstedt}
	\begin{equation*}
	K(t^k,f; L_{p,q}(\R^d), L_{p^*,q}(\R^d)) \asymp \Big(\int_0^{t^d} (u^{\frac{1}{p}} f^*(u))^q \frac{du}{u} \Big)^{1/q} + t^k \Big(\int_{t^d}^\infty (u^{\frac{1}{p} - \frac{k}{d}} f^*(u))^q \frac{du}{u} \Big)^{1/q}
	\end{equation*}
	and \eqref{ProofLemmaEmbBMOLorentzState1*} in \eqref{hfhfhhf}, we infer that
	\begin{equation}\label{dhhdshahs}
		\Big(\int_0^{t^d} (u^{\frac{1}{p}} f^*(u))^q \frac{du}{u} \Big)^{1/q} + t^k \Big(\int_{t^d}^\infty (u^{\frac{1}{p} - \frac{k}{d}} f^*(u))^q \frac{du}{u} \Big)^{1/q} \lesssim \omega_k(f,t)_{p,q}
	\end{equation}
	and, in particular,
	\begin{equation*}
		\Big(\int_0^{t^d} (u^{\frac{1}{p}} f^*(u))^q \frac{du}{u} \Big)^{1/q}  \lesssim \omega_k(f,t)_{p,q}.
	\end{equation*}
	Applying now \eqref{ProofLem2.1} and the Hardy's inequality, we get
	\begin{align*}
		 \Big(\int_0^{t^d} (u^{1/p} f^{\# *}(u))^q \frac{du}{u} \Big)^{1/q} &\lesssim  \Big(\int_0^{t^d} (u^{1/p} f^{* *}(u))^q \frac{du}{u} \Big)^{1/q} \\
		 & \hspace{-3cm}\lesssim  \Big(\int_0^{t^d} (u^{1/p} f^{*}(u))^q \frac{du}{u} \Big)^{1/q} \lesssim  \omega_k(f,t)_{p,q}.
	\end{align*}
	
	 \textsc{Case 4:  $p=q=1$ and $k=d$.} It follows from the fundamental theorem of Calculus that $\dot{W}^d_1(\R^d) \hookrightarrow L_\infty(\R^d)$, and thus
	 \begin{equation*}
	 	K(t^d, f; L_1(\R^d), L_\infty(\R^d)) \lesssim K(t^d, f; L_1(\R^d), \dot{W}^d_1(\R^d)).
	 \end{equation*}
	 In light of 
	 $	K(t^d, f; L_1(\R^d), \dot{W}^d_1(\R^d)) \asymp \omega_d(f,t)_{1},$ $f \in L_{1}(\R^d),
	 $ 
	 and
	 \begin{equation*}
	  K(t, f; L_1(\R^d), L_\infty(\R^d)) = t f^{**}(t)
	 \end{equation*}
	 (see, e.g., \cite[Theorem 1.6, Chapter 5, p. 298]{BennettSharpley}) we arrive at \eqref{ThmDeVore*New}.
	
	 \textsc{Case 5: $p=q=1$ and $k < d$.} See Remark  \ref{RemarkCubes}(iv).
	
\end{proof}


The proof of the optimality of Theorem \ref{ThmDeVoreLorentz} relies on
 a study of
  the Fourier series with monotone-type coefficients. We obtain sharp estimates for oscillations and smoothness of such functions. Our method is based on realization results for moduli of smoothness (cf. Lemma \ref{LemModuliLorentz} below), the lower bound for the sharp maximal function given in Theorem \ref{ThmBDS}, and limiting interpolation techniques.

Note that the  characterization of functions from  $\text{BMO}(\T)$ in terms of their Fourier coefficients
was given by the celebrated Fefferman's result  \cite{SleddStegenga} assuming,  additionally, that  
 Fourier coefficients are non-negative.
  Furthermore, various characterizations of $\text{BMO}$ functions under special conditions  (e.g., lacunary Fourier series,  power series)
  are also known (see, e.g., \cite{ChamizoCordobaUbis, KolyadaLeindler} among other). However, these results cannot be applied
  to establish  pointwise estimates of both the oscillation (i.e., the sharp maximal function) and the moduli of smoothness. 

The next result confirms that inequalities \eqref{ThmDeVore*Lorentz} and \eqref{ThmDeVore*Lorentz2} are optimal and provide a non-trivial improvement of \eqref{DeVMax}. 
	\begin{prop}\label{ThmSharpnessAssertion}
	 Let $1 < p < \infty, 0 < q \leq \infty,$ and $k \in \N$. Let $b$ a positive slowly varying function on $(1,\infty)$ such that
	 \begin{equation}\label{ThmSharpnessAssertionAssump0}
	 	\int_1^\infty (b(u))^q \frac{du}{u} < \infty
	 \end{equation}
	  (where the integral should be replaced by the supremum if $q=\infty$). Set
	\begin{equation}\label{Auxb}
	 	\tilde{b}_q(t) = \Big(\int_t^\infty (b(u))^q \frac{du}{u} \Big)^{1/q}, \qquad t > 1.
	 \end{equation}
	 Assume that
	 \begin{equation}\label{ThmSharpnessAssertionAssump2}
	 	\tilde{b}_q(t) \lesssim \tilde{b}_q(t(\log t)^p) \quad \text{as} \quad t \to \infty.
	 \end{equation}
	 Let $f \in L_1(\T)$ and
	\begin{equation}\label{FouSer}
		f (x) \sim \sum_{n=1}^\infty n^{-1 + 1/p} b(n) \cos nx, \qquad x \in \T.
	\end{equation}
	Then
	\begin{equation}\label{ThmSharpnessAssertion1}
		\frac{f^{\# *}(t)}{\sup_{t < u < 1} u^{-1/p} \omega_k(f,u)_{p,q}} \to 0 \quad \text{as} \quad t \to 0
	\end{equation}
	and
		\begin{equation}\label{ThmSharpnessAssertion2}
	t^{-1/p} \Big(\int_0^{t} (u^{1/p} f^{\# *}(u))^q \, \frac{du}{u} \Big)^{1/q} \asymp \sup_{t < u < 1} u^{-1/p} \omega_k(f,u)_{p,q}.
\end{equation}
	\end{prop}

	\begin{rem}\label{remark-slowly-var}
(i) We will show in Appendix A that condition
(\ref{ThmSharpnessAssertionAssump2}) is in fact essential. 
\\
(ii) As examples of slowly varying functions satisfying  \eqref{ThmSharpnessAssertionAssump0} and \eqref{ThmSharpnessAssertionAssump2}, take $b (t) = (\log t)^{-\varepsilon} \psi(\log t)$ with $\varepsilon > 1/p$ and $\psi$ being a broken logarithmic function, or $b(t) = (\log t)^{-1/p} (\log (\log t))^{-\beta}$ with $\beta > 1/p$. See \cite{Bingham} for more examples of slowly varying functions.
\\
(iii) Functions \eqref{FouSer} can be easily extended to consider more general Fourier series $f (x) \sim \sum_{n=1}^\infty n^{-1 + 1/p} (b_1(n)\cos nx+ b_2(n)\sin nx)$.
	\end{rem}

The proof of Proposition \ref{ThmSharpnessAssertion} is based on technical Lemmas \ref{LemModuliLorentz}--\ref{Lem3}, which are interesting for their own sake and deal with Fourier series with general monotone coefficients.


A sequence $a=\{a_n\}$
is called \emph{general monotone}, written $a \in GM$, if there is a
constant $C > 0$ such that
\begin{equation*}
    \sum_{\nu=n}^{2n-1} |\Delta a_\nu| \leq C |a_n| \quad \text{for all} \quad n
    \in \mathbb{N}.
\end{equation*}
Here $\Delta a_\nu = a_\nu - a_{\nu+1}$ and the constant $C$ is
independent of $n$. It is proved in \cite[p. 725]{Tikhonov} that $a \in GM$
if and only if
\begin{equation}\label{GM1}
    |a_\nu| \lesssim |a_n| \quad \text{for} \quad n \leq \nu \leq 2n
\end{equation}
and
\begin{equation}\label{GM2}
    \sum_{\nu=n}^N |\Delta a_\nu| \lesssim |a_n| + \sum_{\nu=n+1}^N
    \frac{|a_\nu|}{\nu} \quad \text{for any} \quad n < N.
\end{equation}

Examples of $GM$ sequences include decreasing sequences, increasing sequences satisfying the condition $a_{2n}\lesssim a_n$,
quasi-monotone sequences, i.e., such that $a_n/n^\tau$ is decreasing for some $\tau \geq 0$ and many others; see \cite{Tikhonov}.

\begin{lem}\label{LemModuliLorentz}
	Let $1 < p < \infty, 0 < q \leq \infty$, and $k \in \N$. Assume that $f \in L_1(\T)$ and
	\begin{equation*}
		f (x) \sim \sum_{n=1}^\infty (a_n \cos nx + b_n \sin nx), \qquad x \in \T,
	\end{equation*}
	where $\{a_n\}, \{b_n\}$ are nonnegative general monotone sequences. For every $j \in \N_0$, we have
	\begin{equation}\label{ModLorentzGM}
	\omega_k(f, 2^{-j})_{p,q} \asymp  2^{-j k}  \Big(\sum_{\nu=0}^j 2^{\nu(k + 1/p') q} (a_{2^\nu}^q + b_{2^\nu}^q) \Big)^{1/q} +\Big(\sum_{\nu=j}^\infty 2^{\nu q/p'} (a_{2^\nu}^q + b_{2^\nu}^q) \Big)^{1/q}
	\end{equation}
	(with the usual modification if $q=\infty$).
\end{lem}

This result is  known  for Lebesgue spaces (i.e., $p=q$)
\cite[Theorem 6.1]{GorbachevTikhonov}.

\begin{proof}[Proof of Lemma \ref{LemModuliLorentz}]
Since the periodic Hilbert transform is bounded in ${L_{p,q}(\T)}$, we may assume that $b_n =0, \, n \in \N$.
 We will apply the following realization result for moduli of smoothness \cite[Lemma 3.1]{GogatishviliOpicTikhonovTrebels}
	\begin{equation}\label{realization}
		\omega_k(f, 2^{-j})_{p,q}  \asymp \|f - S_{2^j} f\|_{L_{p,q}(\T)} + 2^{-j k} \|(S_{2^j} f)^{(k)}\|_{L_{p,q}(\T)},
	\end{equation}
	where $S_{2^j} f (x)  \sim \sum_{n=1}^{2^j} a_n \cos nx$. Although the proof given in \cite{GogatishviliOpicTikhonovTrebels} is only stated  for $1 \leq q \leq \infty$, it also holds in the case $0 < q \leq \infty$. 

To verify \eqref{ModLorentzGM}, in view of \eqref{realization} and \eqref{GM1}, it suffices to show the following estimates:
	\begin{equation}\label{Claim1}
		\Big(\sum_{i=2^{j+1}} ^\infty i^{q/p'-1} a_i^q\Big)^{1/q} \lesssim \|f- S_{2^j} f\|_{L_{p,q}(\T)} \lesssim \Big(\sum_{i=2^{j-1}} ^\infty i^{q/p'-1} a_i^q\Big)^{1/q}
	\end{equation}
	and
	\begin{equation}\label{Claim2}
		\|(S_{2^j} f)^{(k)}\|_{L_{p,q}(\T)} \asymp \Big(\sum_{i=1}^{2^j} i^{(k + 1/p')q - 1} a_i^{q} \Big)^{1/q}.
	\end{equation}

	Suppose that $q < \infty$. To estimate $\|f - S_{2^j} f\|_{L_{p,q}(\T)}$, we  make use of the inequality \cite[Theorem 2.4]{Sagher}
	\begin{equation}\label{sagJMAA}
		\Big\|\Big(\sup_{n \geq i} \Big|\frac{1}{n} \sum_{l=1}^n a_l \Big| \Big)_{i \in \N} \Big\|_{\ell_{p',q}} \lesssim \|f\|_{L_{p,q}(\T)}
	\end{equation}
	for $f(x) \sim \sum_{n=1}^\infty a_n \cos nx$. Here, $\ell_{p',q}$ denotes the Lorentz sequence space (see \cite{BennettRudnick, BennettSharpley}). Applying (\ref{sagJMAA}) for  $f - S_{2^j} f$ we obtain
	\begin{equation}\label{sag}
	\qquad	\quad\quad\Big\|\Big(\sup_{ n \geq i} \frac{1}{n} \sum_{l=1}^n \bar{a}_l \Big)_{i \in \N} \Big\|_{\ell_{p',q}} \lesssim \|f- S_{2^j} f\|_{L_{p,q}(\T)},
\quad
	\bar{a}_l = \left\{\begin{array}{lcl}
	0 & , & l \leq 2^j,\\

                             a_l & ,  & l \geq 2^j + 1.
            \end{array}
            \right.
	\end{equation}
%
	We split the left-hand side of \eqref{sag} into two terms $K_1 + K_2$, where
	\begin{equation*}
		K_1 = \bigg(\sum_{i=1}^{2^j} i^{q/p' - 1} \Big(\sup_{n \geq i} \frac{1}{n} \sum_{l=1}^n \bar{a}_l \Big)^q \bigg)^{1/q} \, \,  \text{and} \, \,  K_2 = \bigg(\sum_{i=2^j + 1}^\infty i^{q/p' - 1} \Big(\sup_{n \geq i} \frac{1}{n} \sum_{l=1}^n \bar{a}_l \Big)^q \bigg)^{1/q}.
	\end{equation*}
	Applying \eqref{GM1}, we obtain
	\begin{equation*}
		K_1 \geq \Big(\sum_{i=1}^{2^j} i^{q/p' - 1} \Big)^{1/q} \sup_{n \geq 2^j + 1} \frac{1}{n} \sum_{l=2^j + 1}^n a_l \gtrsim 2^{j/p'} a_{2^{j+1}}
	\end{equation*}
	and
	\begin{align*}
		K_2  &= \bigg(\sum_{i=2^j + 1}^\infty i^{q/p' - 1} \Big(\sup_{n \geq i} \frac{1}{n} \sum_{l=2^j + 1}^n a_l \Big)^q \bigg)^{1/q} \geq \bigg(\sum_{i=2^{j + 1}+2}^\infty i^{q/p' - 1} \Big( \frac{1}{i} \sum_{l=[i/2]}^i a_l \Big)^q \bigg)^{1/q}\\
		& \gtrsim \Big(\sum_{i=2^{j+1} + 2} ^\infty i^{q/p'-1} a_i^q\Big)^{1/q},
	\end{align*}
	where as usual $[x]$ denotes the integer part of $x \in \R$. Inserting these estimates into \eqref{sag}, we obtain the first inequality in \eqref{Claim1}.

		Next we prove the upper estimate in  \eqref{Claim1}. Using summation by parts, for $x \neq 0$ and $N \geq j$,
	\begin{align*}
		\Big| \sum_{\nu=2^j+1}^\infty a_\nu \cos \nu x \Big| & \lesssim \sum_{\nu=2^j}^{2^N} a_\nu + \frac{1}{|x|} \sum_{\nu=2^N+1}^\infty |\Delta a_\nu| \\
		& \lesssim  \sum_{\nu=2^j}^{2^N} a_\nu + \frac{1}{|x|}  \sum_{\nu=2^{N-1}}^\infty \frac{a_\nu}{\nu},
	\end{align*}
	where the last estimate follows from \eqref{GM1}, \eqref{GM2}. Consequently,
	\begin{equation}\label{EstimRearrange}
		(f- S_{2^j} f)^*(t) \leq C \left\{\begin{array}{lcl}
                               \sum_{\nu=2^j}^{2^N} a_\nu + \frac{1}{t}  \sum_{\nu=2^{N-1}}^\infty \frac{a_\nu}{\nu} & ,  & N \geq j, \, \quad t > 0, \\
                            & & \\
                            \frac{1}{t} \sum_{\nu=2^{j-1}}^\infty \frac{a_\nu}{\nu}& , & t > 0,
            \end{array}
            \right.
	\end{equation}
	where $C$ is a positive constant which is independent of $f, t, j$, and $N$. These estimates imply, together with monotonicity properties, that
	\begin{equation}\label{Estim1}
	\|f - S_{2^j} f\|_{L_{p,q}(\T)} \lesssim \Big(\sum_{i=0}^\infty 2^{-i q/p} ((f- S_{2^j} f)^*(2^{-i-1}))^q \Big)^{1/q} \lesssim D_1 + D_2,
	\end{equation}
	where
	\begin{equation*}
		D_1 = 2^{j/p'} \sum_{\nu= j-1}^\infty a_{2^\nu} \quad \text{and} \quad D_2 = \bigg(\sum_{i=j}^\infty 2^{-i q/p} \Big(\sum_{\nu=j}^{i} 2^\nu  a_{2^\nu} + 2^{i}  \sum_{\nu=i-1}^\infty a_{2^\nu}  \Big)^q \bigg)^{1/q}.
	\end{equation*}
	Further, applying H\"older's inequality if $q \geq 1$ and the embedding $\ell_q \hookrightarrow \ell_1$ if $q < 1$, we have
	\begin{equation}\label{Estim2}
		D_1 \lesssim \Big(\sum_{\nu=j-1}^\infty 2^{\nu q/p'} a_{2^\nu}^q \Big)^{1/q}.
	\end{equation}
	On the other hand, using Hardy's inequality, 
  we estimate
	\begin{equation}\label{Estim3}
		D_2 \lesssim \Big(\sum_{i=j}^\infty 2^{i q/p'} a_{2^{i}}^q \Big)^{1/q}.
	\end{equation}
	Therefore, in virtue of \eqref{Estim1}--\eqref{Estim3}, the second inequality in \eqref{Claim1} is shown.
	
	Finally, we proceed to prove \eqref{Claim2}. Invoking again \eqref{sagJMAA},
	\begin{align*}
		\|(S_{2^j} f)^{(k)}\|_{L_{p,q}(\T)} \gtrsim \bigg(\sum_{i=1}^{2^j} i^{q/p'-1} \Big(\frac{1}{i} \sum_{l=1}^i l^k a_l \Big)^q \bigg)^{1/q} \gtrsim \Big(\sum_{i=1}^{2^j} i^{(k + 1/p')q - 1} a_i^{q} \Big)^{1/q},
	\end{align*}
	where the last step follows from \eqref{GM1}. To establish the converse estimate, one can use a similar argument as above (see \eqref{EstimRearrange}) to show
		\begin{equation*}
		\big((S_{2^j} f)^{(k)}\big)^*(t) \lesssim \left\{\begin{array}{lcl}
                               \sum_{\nu=1}^{2^N} \nu^k a_\nu + \frac{1}{t}  \sum_{\nu=2^{N-1}}^{2^j} \frac{\nu^k a_\nu}{\nu} & ,  & N < j, \, \quad t > 0, \\
                            & & \\
                             \sum_{\nu=1}^{2^j} \nu^k a_\nu& , & t > 0.
            \end{array}
            \right.
	\end{equation*}
	Thus
	\begin{equation*}
		\|(S_{2^j} f)^{(k)}\|_{L_{p,q}(\T)} \lesssim \Big(\sum_{i=0}^\infty 2^{-i q/p} ((S^{(k)}_{2^j} f)^*(2^{-i-1}))^q  \Big)^{1/q}\lesssim E_1 + E_2,
	\end{equation*}
	where
	\begin{equation*}
		E_1 = 2^{-j/p} \sum_{\nu=0}^{j} 2^{\nu (k+1)} a_{2^\nu}  \,\, \text{and} \,\,
		E_2 = \bigg(\sum_{i=0}^j 2^{-i q/p} \Big(\sum_{\nu=0}^{i} 2^{\nu (k+1)} a_{2^\nu} + 2^{i}  \sum_{\nu=i-1}^{j} 2^{\nu k} a_{2^\nu} \Big)^q \bigg)^{1/q}.
	\end{equation*}
	By reasoning in the same way as in \eqref{Estim2} and \eqref{Estim3}, one has   $E_1 + E_2 \lesssim \Big(\sum_{i=0}^j 2^{i (k + 1/p') q} a_{2^{i}}^q\Big)^{1/q}$. This concludes the proof of \eqref{Claim2}.
	Similarly, \eqref{ModLorentzGM} also holds for $q=\infty$.
\end{proof}

As a byproduct of Lemma \ref{LemModuliLorentz} we obtain the following characterization of Lorentz--Besov norms.

\begin{lem}\label{Lem1}
	Let $1 < p < \infty,  0 < q, r \leq \infty, k \in \N$, and $0 < s < k$. Assume that $f \in L_1(\T)$ and
	\begin{equation}\label{Assumption1}
		f (x) \sim \sum_{n=1}^\infty (a_n \cos nx + b_n \sin nx), \qquad x \in \T,
	\end{equation}
	where $\{a_n\}, \{b_n\}$ are nonnegative general monotone sequences. For every $j \in \N$, we have
	\begin{align}
		\Big(\int_{2^{-j}}^1 (u^{-s} \omega_k(f,u)_{p,q})^r \frac{du}{u} \Big)^{1/r} & \asymp \Big(\sum_{\nu=0}^j 2^{\nu (s + 1/p') r} (a_{2^\nu}^r + b_{2^\nu}^r) \Big)^{1/r} \nonumber \\
		& \hspace{1cm}+  2^{j s} \Big(\sum_{\nu=j}^\infty 2^{\nu q/p'} (a_{2^{\nu}}^q + b_{2^\nu}^q) \Big)^{1/q}. \label{Lem1.1}
	\end{align}
\end{lem}


\begin{rem}
Passing to the limit $j \to \infty$ in \eqref{Lem1.1}, we derive that the Besov seminorm
	\begin{equation*}
		\vertiii{f}_{\dot{B}^s_r L_{p,q}(\T);k} := \Big(\int_0^1 (u^{-s} \omega_k(f,u)_{p,q})^r \frac{du}{u} \Big)^{1/r}  \asymp \Big(\sum_{\nu=0}^\infty 2^{\nu (s + 1/p') r} (a_{2^\nu}^r + b_{2^\nu}^r) \Big)^{1/r}.
	\end{equation*}
Thus the fine integrability parameter $q$ does not play a role when working with the global Lorentz--Besov norm. In particular, we have
	\begin{equation*}
		\vertiii{f}_{B^s_r L_{p,q}(\T);k} \asymp \vertiii{f}_{B^s_{p,r}(\T);k}; 
	\end{equation*}
	see also \cite[Theorem 4.22]{DominguezTikhonov}.
However, the quantitative estimate \eqref{Lem1.1} strongly depends on  $q$, since it is easily seen that
  the terms
	\begin{equation*}
		\Big(\sum_{\nu=0}^j 2^{\nu (s + 1/p') r} (a_{2^\nu}^r + b_{2^\nu}^r) \Big)^{1/r}\qquad \text{and} \qquad 2^{j s} \Big(\sum_{\nu=j}^\infty 2^{\nu q/p'} (a_{2^{\nu}}^q + b_{2^\nu}^q) \Big)^{1/q},
	\end{equation*}
	are not comparable.
\end{rem}

\begin{proof}[Proof of Lemma \ref{Lem1}]
	It follows from Lemma \ref{LemModuliLorentz} that
	$$
		\Big(\int_{2^{-j}}^1 (u^{-s} \omega_k(f,u)_{p,q})^r \frac{du}{u} \Big)^{1/r}  \asymp K_1 + K_2 + K_3,
	$$
	where
	\begin{equation*}
		K_1 = \Big(\sum_{\nu=0}^j \big(2^{\nu (s -k) q}\sum_{i = 0}^\nu 2^{i(k + 1/p')q} (a_{2^{i}}^q + b_{2^{i}}^q) \big)^{r/q} \Big)^{1/r},
		\end{equation*}
		\begin{equation*}
			K_2 =  \Big(\sum_{\nu=0}^j\big(2^{\nu s q} \sum_{i=\nu}^j 2^{i q/p'} (a_{2^{i}}^q + b_{2^{i}}^q) \big)^{r/q} \Big)^{1/r}, \quad  K_3 = 2^{j s} \Big(\sum_{\nu=j}^\infty 2^{\nu q/p'} (a_{2^{\nu}}^q + b_{2^\nu}^q) \Big)^{1/q}.
	\end{equation*}
	Applying Hardy's inequality (since $0 < s < k$),
	\begin{equation*}
		K_1  \asymp K_2 \asymp \Big(\sum_{\nu=0}^j 2^{\nu (s + 1/p') r} (a_{2^\nu}^r + b_{2^\nu}^r) \Big)^{1/r}
	\end{equation*}
	and thus 
  we arrive at \eqref{Lem1.1}.
\end{proof}

Our next result provides an upper estimate of the sharp maximal function $f^{\#}$ in terms of
 Fourier coefficients of $f$.

\begin{lem}\label{Lem2}
	Assume that $f \in L_1(\T)$ and
	\begin{equation*}
		f (x) \sim \sum_{n=1}^\infty (a_n \cos nx + b_n \sin nx), \qquad x \in \T,
	\end{equation*}
	where $\{a_n\}, \{b_n\}$ are nonnegative general monotone sequences. Then
	\begin{equation*}
		f^{\#*}(2^{-j}) \lesssim \sum_{n=0}^j 2^n (a_{2^n} + b_{2^n})+ 2^j   \sum_{n=j}^\infty \sum_{\nu=n}^\infty (a_{2^\nu} + b_{2^\nu})
	\end{equation*}
	for  $j \in \N.$
\end{lem}
\begin{proof}
For simplicity, we  assume that $b_n = 0$. 
Taking into account \eqref{ProofLem2.1}, it suffices  to show that
\begin{equation*}
		f^{**}(2^{-j}) \lesssim \sum_{n=0}^j 2^n a_{2^n}+ 2^j   \sum_{n=j}^\infty \sum_{\nu=n}^\infty a_{2^\nu}.
	\end{equation*}
Arguing along similar lines as \eqref{EstimRearrange}, we obtain
	\begin{equation*}
		f^*(t) \leq C \left(  \sum_{n=1}^m a_n + \frac{1}{t} \Big(a_{m+1} + \sum_{n=m+2}^\infty \frac{a_n}{n} \Big) \right), \quad t > 0, \quad m \in \N.
	\end{equation*}
	Applying this estimate together with monotonicity properties (see \eqref{GM1}), we establish
	\begin{align*}
		f^{**}(2^{-j}) &= 2^j \int_0^{2^{-j}} f^*(u) \, du \asymp 2^j \sum_{\nu=j}^\infty 2^{-\nu} f^*(2^{-\nu})  \\
		& \lesssim 2^j \sum_{\nu=j}^\infty 2^{-\nu}
\left(\sum_{n=0}^{\nu} 2^n a_{2^n} + 2^\nu  \sum_{n={\nu-1}}^\infty a_{2^n}  \right)
\\
		& \asymp \sum_{n=0}^j 2^n a_{2^n}+ 2^j   \sum_{n=j}^\infty \sum_{\nu=n}^\infty a_{2^\nu}.
	\end{align*}
\end{proof}

Concerning the lower bound of $f^{\#}$, we obtain the following

\begin{lem}\label{Lem3}
Let $1 < p < \infty$ and $0 < q \leq \infty$. Assume that $f \in L_1(\T)$ and
		\begin{equation*}
		f (x) \sim \sum_{n=1}^\infty (a_n \cos nx + b_n \sin nx), \qquad x \in \T,
	\end{equation*}
	where $\{a_n\}, \{b_n\}$ are nonnegative general monotone sequences. Then
	\begin{equation*}
		 \sum_{n=j}^\infty  2^{-n q/p} n^{-q} \big(2^{n} n^p (a_{[2^n n^p]} + b_{[2^n n^p]})\big)^q \lesssim \int_0^{2^{-j}} (u^{1/p} f^{\#*}(u))^q \, \frac{du}{u}, \quad j \in \N.
	\end{equation*}
\end{lem}
\begin{proof}
	Let $f_-(x)=\frac{f(x)-f(-x)}2$ and $f_+(x)=f(x)-f_-(x)$.
By \eqref{GM1}, making use of the method of \cite{tikhonov1},  it is plain to check that
$
		b_n \lesssim \int_0^{\frac{\pi}{n} } f_-(u) \, du.
$
Moreover, noting that $$\int_0^{t } \int_0^{x} f_+(u) \, du\, dx=2\sum_{n=1}^\infty\frac{a_n}{n^2}\sin^2 \frac{nt}2,
$$
we also derive $
		a_n \lesssim \int_0^{\frac{\pi}{n} } |f_+(u)| \, du.
$ 
Thus we have
	\begin{equation*}
		a_n + b_n \lesssim \int_{-\pi/n }^{\pi/n } |f(u)| \, du.
	\end{equation*}
Using basic properties of rearrangements \cite[Lemma 2.1, Chapter 2, p. 44]{BennettSharpley} implies
	\begin{equation}\label{ProofLem3.1}
	n (a_n + b_n) \lesssim n \int_{-\pi/n }^{\pi/n} |f(u)| \, du \lesssim  n \int_0^{\pi/n} f^*(u) \, du \lesssim f^{**}\Big(\frac{1}{n}\Big).
	\end{equation}
	Invoking Theorem \ref{ThmBDS} (see also Remark \ref{Rem1}(iii)), we obtain
	\begin{align*}
		 \int_0^{2^{-j}} (u^{1/p} f^{\#*}(u))^q \, \frac{du}{u} & \gtrsim \int_0^{2^{-j} j^{-p}} (u^{1/p} f^{**}(u))^q \, \frac{du}{u} \gtrsim \sum_{n=j}^\infty (2^{-n} n^{-p})^{q/p}  (f^{**}(2^{-n} n^{-p}))^q  \\
		 &\gtrsim \sum_{n=j}^\infty  2^{-n q/p} n^{-q} \big(2^{n} n^p (a_{[2^n n^p]} + b_{[2^n n^p]})\big)^q,
	\end{align*}
	where the last step follows from \eqref{ProofLem3.1}.

\end{proof}

We are now in a position to prove Proposition \ref{ThmSharpnessAssertion}.
\begin{proof}[Proof of Proposition \ref{ThmSharpnessAssertion}]
	Let $j \in \N$. Applying Lemma \ref{Lem1} with $s=1/p$ and $r=\infty$ and basic properties of slowly varying functions (see \cite{Bingham}), we have
	\begin{align}
		\sup_{2^{-j} < u < 1} u^{-1/p} \omega_k(f,u)_{p,q} &\asymp \sup_{\nu=0, \ldots, j} 2^{\nu/p} b(2^\nu) +  2^{j/p} \Big(\sum_{\nu=j}^\infty  b(2^\nu)^q \Big)^{1/q} 
\asymp 
                 2^{j/p} \tilde{b}_q(2^j).
                 \label{ProofThmSharpnessAssertion1}
	\end{align}	
Further, in light of Lemma \ref{Lem2},
	\begin{equation}\label{ProofThmSharpnessAssertion2}
	f^{\#*}(2^{-j}) \lesssim \sum_{n=0}^j 2^{n/p} b(2^n) + 2^j   \sum_{n=j}^\infty \sum_{\nu=n}^\infty 2^{\nu(-1+1/p)} b(2^\nu) \asymp 2^{j/p} b(2^j).
	\end{equation}
  On the other hand,	 Lemma \ref{Lem3} yields
	\begin{align}
		 \left(\int_0^{2^{-j}} (u^{1/p} f^{\#*}(u))^q \, \frac{du}{u} \right)^{1/q} &\gtrsim \left(\sum_{n=j}^\infty 2^{-n q/p} n^{-q} (2^{n} n^p (2^n n^p)^{-1 + 1/p}b(2^n n^p))^q \right)^{1/q} \nonumber \\
		 & = \left(\sum_{n=j}^\infty (b(2^n n^p))^q \right)^{1/q} \asymp \tilde{b}_q(2^j j^p). \label{ProofThmSharpnessAssertion3}
	\end{align}
	
	By \eqref{ProofThmSharpnessAssertion1} and \eqref{ProofThmSharpnessAssertion2},
	\begin{equation*}
		\frac{f^{\#*}(2^{-j}) }{\sup_{2^{-j} < u < 1} u^{-1/p} \omega_k(f,u)_{p,q}} \lesssim \frac{b(2^j)}{\tilde{b}_q(2^j)}
	\end{equation*}
	and thus, \eqref{ThmSharpnessAssertion1} holds since
	\begin{equation*}
		\lim_{t \to \infty} \frac{b(t)}{\tilde{b}_q(t)} = 0
	\end{equation*}
	(see \cite[Propositions 1.3.6(ii) and 1.5.9b]{Bingham}). Further, taking into account assumption \eqref{ThmSharpnessAssertionAssump2}, it follows from \eqref{ProofThmSharpnessAssertion3} and \eqref{ProofThmSharpnessAssertion1} that
	\begin{align*}
		2^{j/p}  \left(\int_0^{2^{-j}} (u^{1/p} f^{\#*}(u))^q \, \frac{du}{u} \right)^{1/q} & \gtrsim 2^{j/p} \tilde{b}_q(2^j j^p) \gtrsim 2^{j/p} \tilde{b}_q(2^j) \\
		&\hspace{-3cm} \asymp \sup_{2^{-j} < u < 1} u^{-1/p} \omega_k(f,u)_{p,q}.
	\end{align*}
	Hence, assertion \eqref{ThmSharpnessAssertion2} follows now from \eqref{ThmDeVore*Lorentz}.
\end{proof}

Our next objective is to show the sharpness of \eqref{ThmDeVore*New}, i.e.,
$t f^{**}(t) \lesssim \omega_1(f,t)_{1}$.

\begin{prop}\label{SharpnessThmDeVore*New}
	Let $k\in \N$ and  $f \in L_1(\T)$ be such that
	\begin{equation*}
		f (x) \sim \sum_{n=1}^\infty   a_n \cos n x, \qquad x \in \T,
	\end{equation*}
where
$\{a_n\}$ is a decreasing convex sequence satisfying, for some $\varepsilon > 0$,
	\begin{equation}\label{SharpnessThmDeVore*New1--}a_n n^{k-\varepsilon}\lesssim a_m m^{k-\varepsilon},\qquad n\le m.
\end{equation}
 Then
	\begin{equation}\label{SharpnessThmDeVore*New1}
	\omega_k(f, t)_1 \lesssim t f^{**}(t).
	\end{equation}
\end{prop}

A typical example of $\{a_n\}$ satisfying
\eqref{SharpnessThmDeVore*New1--} is  $a_n=n^{-\varkappa}d_n$ 	with $0 < \varkappa < k$ and
$d_n=d(n)$ with a slowly varying function $d(x)$. Note that in this case, if $0 < \varkappa < 1$, then $f(x)\sim x^{\varkappa-1} d(1/x) \Gamma(1-\varkappa)\sin \frac{\pi \varkappa}2$  as $x\to 0+$ (cf. \cite{Tikhonov}).
The proof of Proposition \ref{SharpnessThmDeVore*New} is based  on the following estimates of moduli of smoothness of Fourier series with convex coefficients.

\begin{lem}[{\cite{Aljancic}}]\label{LemmaAljancic}
	 Assume that  $f \in L_1(\T)$ is such that
	$
		f (x) \sim \sum_{n=1}^\infty a_n  \cos n x,$ $x \in \T,
	$ 
	 where  $\{a_n\}$ is a decreasing convex sequence. Then,  for   $k \in \N$,
	\begin{equation*}
		\omega_k\Big(f, \frac{1}{n}\Big)_1 \lesssim \frac{1}{n^k} \sum_{\nu=1}^n {\nu^{k-1}}{a_\nu}, \qquad n \in \N.
	\end{equation*}
\end{lem}

\begin{proof}[Proof of Proposition \ref{SharpnessThmDeVore*New}]
	Applying Lemma \ref{LemmaAljancic} and taking into account condition \eqref{SharpnessThmDeVore*New1--}, 
 we obtain
	\begin{equation*}
		\omega_k \Big( f, \frac{1}{n}\Big)_1 \lesssim
\frac{1}{n^k} \sum_{\nu=1}^n {\nu^{k-1}}{a_\nu} \asymp a_n.
	\end{equation*}	
On the other hand, by \eqref{ProofLem3.1},
	\begin{equation*}
		\frac{1}{n} f^{**} \Big( \frac{1}{n}\Big) \gtrsim a_n.
	\end{equation*}
	The desired estimate \eqref{SharpnessThmDeVore*New1} now follows from  monotonicity properties of moduli of smoothness and $f^{**}$.
\end{proof}

\begin{proof}[Proof of Remark \ref{RemarkCubes} \emph{(iv)}]
	Let $\text{Ext}$ be the linear extension operator for Sobolev spaces given by Calder\'on--Stein \cite[Chapter VI, Section 3, pp. 180--192]{Stein70}. Then 
	\begin{equation*}
		\text{Ext}: L_p(Q_0) \to L_p(\R^d)  \quad \text{and} \quad \text{Ext}: W^k_p(Q_0) \to W^k_p(\R^d)
	\end{equation*}
	for $k \in \N$ and $1 \leq p \leq \infty$. By the  interpolation properties of Sobolev spaces \cite{DeVoreScherer}, the previous estimates admit extensions to Lorentz--Sobolev spaces
	\begin{equation}\label{112}
		\text{Ext}: W^k L_{p,q}(Q_0) \to W^k L_{p,q}(\R^d), \qquad 1 < p < \infty, \quad 0 < q \leq \infty.
	\end{equation}
	Thus, for each $t \in (0,1)$,
	\begin{equation*}
		K(t^k,\text{Ext} f; L_{p,q}(\R^d), W^k L_{p,q}(\R^d)) \lesssim K(t^k,f; L_{p,q}(Q_0), W^k L_{p,q}(Q_0))
	\end{equation*}
	or, equivalently, (see \eqref{ProofLemmaEmbBMOLorentzState1*})
	\begin{equation}\label{333222}
		t^k \|\text{Ext} f\|_{L_{p,q}(\R^d)} + \omega_k(\text{Ext} f, t)_{p,q} \lesssim t^k \|f\|_{L_{p,q}(Q_0)} + \omega_k(f, t)_{p,q}.
	\end{equation}
	On the other hand, taking into account 
  the trivial estimate $\|\text{Re} f\|_{\text{BMO}(Q_0)} \leq \|f\|_{\text{BMO}(\R^d)}$, we derive
	\begin{equation*}
		K(t, \text{Re} f; L_{p,q}(Q_0), \text{BMO}(Q_0)) \leq K(t, f; L_{p,q}(\R^d), \text{BMO}(\R^d)), \quad t \in (0,1).
	\end{equation*}
	According to \eqref{JT} and \eqref{ProofThmDeVoreLorentz2},
	\begin{equation}\label{7474747}
	\int_0^{t} (u^{1/p} (\text{Re} f)^{\#*}_{Q_0}(u))^q \frac{du}{u} \lesssim \int_0^{t} (u^{1/p} f^{\#*}(u))^q \frac{du}{u}, \quad t \in (0,1).
	\end{equation}
	Assume $k > d/p$. Combining \eqref{ThmDeVore*Lorentz}, \eqref{333222} and \eqref{7474747}, we arrive at
	\begin{align*}
		t^{-d/p}  \Big(\int_0^{t^d} (u^{1/p} f_{Q_0}^{\# *}(u))^q \frac{du}{u} \Big)^{1/q} &\lesssim  t^{-d/p}  \Big(\int_0^{t^d} (u^{1/p} (\text{Ext} f)^{\# *}(u))^q \frac{du}{u} \Big)^{1/q}  \\
		& \lesssim \sup_{ t < u < \infty} u^{-d/p} \omega_k(\text{Ext} f,u)_{p,q} \\
		& \lesssim \| \text{Ext} f\|_{L_{p,q}(\R^d)} + \sup_{ t < u < 1} u^{-d/p} \omega_k(\text{Ext} f,u)_{p,q} \\
		& \lesssim \|f\|_{L_{p,q}(Q_0)} + \sup_{ t < u < 1} u^{-d/p} \omega_k(f,u)_{p,q}.
	\end{align*}
	This proves \eqref{ThmDeVore*LorentzCubes}. The proofs of \eqref{ThmDeVore*Lorentz2Cubes} and \eqref{ThmDeVore*NewCubes} are similar but now using \eqref{ThmDeVore*Lorentz2} and \eqref{ThmDeVore*New}, respectively.
	
\end{proof}

\bigskip
\section{Proof of Theorem \ref{ThmDeVoreDer} and its optimality}

\begin{proof}[Proof of Theorem \ref{ThmDeVoreDer}]
	We start by proving \eqref{ThmDeVoreDer<}. According to \eqref{LemmaEmbBMO}, we have
	\begin{equation}\label{000202002}
		\dot{W}^k L_{d/k, \infty}(\R^d) \hookrightarrow \text{BMO}(\R^d).
	\end{equation}
	On the other hand, in light of the Sobolev embedding (see, e.g., \cite[Theorem 2]{Milman}), we obtain
	\begin{equation}\label{ThmDeVoreDer1*}
		\dot{W}^k L_{r,q}(\R^d) \hookrightarrow L_{p,q}(\R^d), \qquad r= \frac{d p}{d + k p},
	\end{equation}
	provided that either $k < d(1-1/p)$ or $k=d(1-1/p)$ and $q=1$. Hence, we have
	\begin{equation}\label{ThmDeVoreDer1}
		K(t, f; L_{p,q}(\R^d), \text{BMO}(\R^d)) \lesssim K(t, f;\dot{W}^k L_{dp/(d+k p),q}(\R^d), \dot{W}^k L_{d/k, \infty}(\R^d)).
	\end{equation}
	By \eqref{ProofThmDeVoreLorentz2},
	\begin{equation}\label{ThmDeVoreDer2}
		K(t, f; L_{p,q}(\R^d), \text{BMO}(\R^d))   \asymp \left(\int_0^{t^p} (u^{1/p} f^{\#*}(u))^q \, \frac{du}{u}  \right)^{1/q}.
	\end{equation}
	Applying the characterization of the $K$-functional for Sobolev spaces given in \cite{DeVoreScherer} (cf. also \cite{CalderonMilman} and \cite{DeVoreSharpley}) together with Holmstedt's formula \cite[Theorem 4.2]{Holmstedt}, we see that
	\begin{align}
		K(t, f; \dot{W}^k L_{r,q}(\R^d), \dot{W}^k L_{d/k, \infty}(\R^d)) &\asymp K(t, |\nabla^k f|; L_{r,q}(\R^d), L_{d/k, \infty}(\R^d)) \nonumber \\
		&\hspace{-5cm} \asymp \Big(\int_0^{t^p} (u^{1/r} |\nabla^k f|^*(u))^q \, \frac{du}{u} \Big)^{1/q} + t \sup_{t^p < u < \infty} u^{k/d} |\nabla^k f|^*(u).  \label{ThmDeVoreDer3}
	\end{align}
	Inserting the estimates \eqref{ThmDeVoreDer2} and \eqref{ThmDeVoreDer3} into \eqref{ThmDeVoreDer1} we derive
	\begin{align*}
		\left(\int_0^{t} (u^{1/p} f^{\#*}(u))^q \, \frac{du}{u}  \right)^{1/q} & \lesssim  \Big(\int_0^{t} (u^{1/r} |\nabla^k f|^*(u))^q \, \frac{du}{u} \Big)^{1/q} \\
		& \hspace{1cm} + t^{1/p} \sup_{t < u < \infty} u^{k/d} |\nabla^k f|^*(u).
	\end{align*}
	
The proof of \eqref{ThmDeVoreDer<Cubes} follows the same lines as above, but now relying on the inhomogeneous counterparts of \eqref{000202002} and \eqref{ThmDeVoreDer1*}, i.e.,
	 \begin{equation}\label{0030030}
		W^k L_{d/k, \infty}(Q_0) \hookrightarrow \text{BMO}(Q_0)
	\end{equation}
	and the classical embedding
	\begin{equation*}
		W^k L_{r,q}(Q_0) \hookrightarrow L_{p,q}(Q_0), \qquad r= \frac{d p}{d + k p}.
	\end{equation*}
	The proof of \eqref{0030030} is an immediate consequence of \eqref{000202002} and the fact that (see \eqref{112})
	\begin{equation*}
		\text{Ext}: W^k L_{d/k,\infty}(Q_0) \to W^k L_{d/k,\infty}(\R^d).
	\end{equation*}

\end{proof}

	To show the optimality of Theorem \ref{ThmDeVoreDer}, to avoid unnecessary complications, we will only concern ourselves with $k=1$.
	
\begin{prop}\label{ThmDeVoreDerSharpnessAssertion}
	Let $1 < p < \infty$ and $r= \frac{dp}{d+p}$. Assume that either of the following conditions is satisfied:
	\begin{enumerate}[\upshape(i)]
	 \item $1 < d (1-1/p)$ and $1 \leq q \leq \infty$,
	 \item $1= d(1-1/p)$ and $q=1$.
	 \end{enumerate}
	 Let $b$ be a  positive slowly varying function on $(0,1)$ satisfying
		 \begin{equation*}
	 	\int_0^1 (b(u))^q \frac{du}{u} < \infty
	 \end{equation*}
	 (where the integral should be replaced by the supremum if $q=\infty$).
	 Set
	\begin{equation*}
	 	\bar{b}_q(t) = \Big(\int^t_0 (b(u))^q \frac{du}{u} \Big)^{1/q}, \qquad t < 1.
	 \end{equation*}
	 Assume that 
	 \begin{equation}\label{ThmDeVoreDerSharpnessAssertion1}
	 	 \bar{b}_q(t)  \lesssim \bar{b}_q(t(-\log t)^{-p/d})  \quad \text{as} \quad t \to 0.
	 \end{equation}
	Define
	\begin{equation*}
	 f(x) = \int_{|x|}^1 u^{-d/p} b(u) \frac{du}{u}, \qquad 0 < |x| < 1,
	 \end{equation*}
	 and $f(x)=0$ otherwise. Then
	 \begin{align*}
	t^{-d/p} \Big(\int_0^{t^d} (u^{1/p}f^{\# *}(u))^q \, \frac{du}{u} \Big)^{1/q} & \asymp t^{-d/p} \Big(\int_0^{t^d} (u^{1/r} |\nabla f|^*(u))^q \, \frac{du}{u} \Big)^{1/q} \\
	& \hspace{1cm}+ \sup_{t^d < u < \infty} u^{1/d} |\nabla f|^*(u)
	\end{align*}
	for $t$ sufficiently small.

	\end{prop}

Regarding condition \eqref{ThmDeVoreDerSharpnessAssertion1} and examples of
functions $b$, see  Remark \ref{remark-slowly-var}(i) and (ii).

\begin{proof}[Proof of Proposition \ref{ThmDeVoreDerSharpnessAssertion}]
	We define $f_0 : [0,\infty) \to [0, \infty)$ by
	\begin{equation*}
	f_0(t) = \int_t^1 u^{-d/p} b(u) \frac{du}{u}, \qquad t \in (0, 1),
	\end{equation*}
	and $f_0(t)=0$ otherwise. Noting that $f(x) = f_0(|x|)$ and $f_0$ is a decreasing function, it is readily seen that $f^*(t) = f_0 \left((t/\omega_d)^{1/d} \right)$, where $\omega_d$ denotes the volume of the $d$-dimensional unit ball. It follows from basic properties of slowly varying functions that
	\begin{equation*}
		f^*(t) \asymp t^{-1/p} b(t^{1/d}), \qquad t \in (0,1/2).
	\end{equation*}
	Further, elementary computations yield that
	\begin{equation*}
	|\nabla f|^*(t) \asymp t^{-1/p -1/d} b(t^{1/d}), \qquad t \in (0,1).
	\end{equation*}
	Assume $q < \infty$. If $t$ is sufficiently small then
	\begin{equation*}
		 t^{-d/p} \Big(\int_0^{t^d} (u^{1/r} |\nabla f|^*(u))^q \, \frac{du}{u} \Big)^{1/q}  + \sup_{t^d < u < \infty} u^{1/d} |\nabla f|^*(u) \asymp t^{-d/p} \bar{b}_q(t),
	\end{equation*}
	where we have used the fact that $\lim_{t \to 0 +} \frac{\bar{b}_q(t)}{b(t)} = \infty$ (see \cite[Propositions 1.3.6(ii) and 1.5.9b]{Bingham}). On the other hand, applying Theorem \ref{ThmBDS} (see also \eqref{757575} with $Q_0 = [-1,1]^d$) and making use of assumption \eqref{ThmDeVoreDerSharpnessAssertion1},
	\begin{equation*}
	t^{-d/p} \Big(\int_0^{t^d} (u^{1/p} f^{\#*}(u))^q \frac{du}{u} \Big)^{1/q}  \gtrsim t^{-d/p} \bar{b}_q(t (1-\log t)^{-p/d}) \gtrsim t^{-d/p} \bar{b}_q(t).
	\end{equation*}
	Combining the above results and
\eqref{ThmDeVoreDer<}, we complete the proof.
 The case $q=\infty$ is easier and we omit further details.
\end{proof}

\bigskip
\section{Proof of Corollary \ref{TheoremSharpLimiting} and its optimality}

\begin{proof}[Proof of Corollary \ref{TheoremSharpLimiting}]
 According to Theorem \ref{ThmBDS}, there exists $s$ such that  the sharp maximal function $M^{\#}_{s,Q_0;w}f$ obeys the estimate
\begin{equation*}
		\int_0^{t (1-\log t)^{-p}} ((f-f_{Q_0;w})^{*}_w(u))^p \, du \lesssim \int_0^{t}  ((M^{\#}_{s,Q_0;w}f)_w^*(u))^p \, du, \quad
	\end{equation*}
and thus, applying  Hardy's inequality, we derive
\begin{align*}
	\bigg(\int_0^1 t^{-q/p} (1 - \log t)^{b q} \Big(\int_0^{t (1-\log t)^{-p}} ((f - f_{Q_0;w})_w^{*}(u))^p \, du \Big)^{q/p} \frac{dt}{t} \bigg)^{1/q} & \\
	& \hspace{-8cm}\lesssim\bigg(\int_0^1 t^{-q/p} (1 - \log t)^{b q} \Big(  \int_0^{t}  ((M^{\#}_{s,Q_0;w}f)_w^*(u))^p \, du\Big)^{q/p} \frac{dt}{t} \bigg)^{1/q} \\
	& \hspace{-8cm} \lesssim \bigg( \int_0^1 ((1-\log t)^{b} (M^{\#}_{s,Q_0;w}f)_w^*(t))^q \frac{dt}{t} \bigg)^{1/q}.
\end{align*}
Next we show that
\begin{align}
	\bigg(\int_0^1 (1 - \log t)^{b q} \Big(\sup_{t <u <1} (1-\log u)^{-1} (f - f_{Q_0;w})_w^*(u) \Big)^{q} \frac{dt}{t} \bigg)^{1/q} & \lesssim \nonumber \\
	&\hspace{-8cm} \bigg(\int_0^1 t^{-q/p} (1 - \log t)^{b q} \Big(\int_0^{t (1-\log t)^{-p}} ((f-f_{Q_0;w})^{*}_w(u))^p \, du \Big)^{q/p} \frac{dt}{t} \bigg)^{1/q}\label{71}
\end{align}
and
\begin{equation}\label{72}
	\bigg( \int_0^1 ((1-\log t)^{b} (M^{\#}_{s,Q_0;w}f)_w^*(t))^q \frac{dt}{t} \bigg)^{1/q} \lesssim \bigg( \int_0^1 ((1-\log t)^{b} (f_{Q_0}^{\#})^*_w(t))^q \frac{dt}{t} \bigg)^{1/q}.
\end{equation}
Note that in the right-hand side of \eqref{72} we work with the  function $f^{\#}_{Q_0}$ taken with respect to the Lebesgue measure.

Using monotonicity properties and changing the order of summation, we infer that
\begin{align*}
	\bigg(\int_0^1 (1 - \log t)^{b q} \Big(\sup_{t < u < 1} (1-\log u)^{-1} (f- f_{Q_0;w})_w^*(u) \Big)^{q} \frac{dt}{t} \bigg)^{1/q} & \asymp  \\
	& \hspace{-9cm} \bigg(\sum_{j=0}^\infty 2^{j(b+1/q) q} \Big(\sup_{\nu = 0, \ldots, j} 2^{-\nu} (f- f_{Q_0;w})_w^*(2^{-2^\nu})\Big)^q \bigg)^{1/q} \\
	&\hspace{-9cm}  \leq \bigg(\sum_{j=0}^\infty 2^{j(b+1/q) q}  \sum_{\nu=0}^j 2^{-\nu q} ((f- f_{Q_0;w})_w^*(2^{-2^\nu}))^q  \bigg)^{1/q} \\
	&\hspace{-9cm}  \asymp \bigg(\sum_{\nu=0}^\infty 2^{\nu((b-1) q + 1)} ((f- f_{Q_0;w})_w^*(2^{-2^\nu}))^q   \bigg)^{1/q} \\
	&\hspace{-9cm}  \asymp \bigg(\int_0^1 ((1-\log t)^{b-1} (f- f_{Q_0;w})_w^*(t))^q \frac{dt}{t} \bigg)^{1/q} \\
	&\hspace{-9cm}  \lesssim \bigg(\int_0^1 t^{-q/p} (1-\log t)^{b q} \Big(\int_0^{t(1-\log t)^{-p}} ((f- f_{Q_0;w})_w^*(u))^p \, du \Big)^{q/p}  \frac{dt}{t}  \bigg)^{1/q}.
\end{align*}
The proof of \eqref{71} is complete.

Concerning \eqref{72}, we first observe that, by the $A_\infty(Q_0)$-condition,  there are $\delta, C>0$ such that
\begin{equation}\label{anew}
	f^*_w(t) \leq C f^*(t^{1/\delta})
\end{equation}
for all measurable functions $f$ on $Q_0$. 
In particular, this yields $$M^{\#}_{s,Q_0;w}f (x) \leq C M^{\#}_{\frac{(s w(Q_0))^{1/\delta}}{|Q_0|_d},Q_0}f (x).$$ On the other hand, it is plain to see that $M^{\#}_{s,Q_0}f (x) \leq (s |Q_0|_d)^{-1} f^{\#}_{Q_0} (x)$. Combining these two inequalities, we arrive at $M^{\#}_{s,Q_0;w}f (x) \lesssim f^{\#}_{Q_0} (x)$
and thus \eqref{72} follows.
\end{proof}

Our next result establishes the optimality of Corollary \ref{TheoremSharpLimiting}.

\begin{prop}\label{PropOptimFS}
	Let $0 < q \leq \infty$ and $b < -1/q$. Assume $w \in A_\infty(Q_0)$ and $f \in L_p(Q_0,w)$ for some $1 < p < \infty$. Then
	\begin{equation}\label{As1}
			\left(\int_0^1 (1 - \log t)^{b q} \left(\int_t^1 (u^{1/p} (1-\log u)^{\xi} (f-f_{Q_0;w})_w^*(u))^r \frac{du}{u} \right)^{q/r} \frac{dt}{t} \right)^{1/q}  \lesssim \|f^{\#}_{Q_0}\|_{L_{\infty,q}(\log L)_b(Q_0,w)}
			\end{equation}
			if and only if
			\begin{equation*}
                           \left\{\begin{array}{lcl}
                            p < \infty, & r \leq \infty,  & -\infty < \xi < \infty, \\
                            & & \\
                            p=\infty,& r < \infty, & \xi < -1 - \frac{1}{r}, \\
                            & & \\
                            p=\infty, & r= \infty, & \xi \leq -1.

            \end{array}
           \right.
	\end{equation*}
\end{prop}

%
%
%
%

\begin{proof}
The backward implication follows from \eqref{FSLim} and simple computations.
The 
  proof of the forward implication with $w = |\cdot|_d$ hinges on the following estimate
\begin{equation}\label{key}
	\|f^{\#}_{Q_0}\|_{L_{\infty,q}(\log L)_b(Q_0)} \lesssim \sum_{l=0}^1  \bigg(\int_0^1 \Big((1-\log t)^b \sup_{t < u < 1} u^{1/d} |\nabla^l f|^*(u) \Big)^{q} \frac{dt}{t} \bigg)^{1/q}.
	\end{equation}
To show \eqref{key}, we invoke \eqref{ThmDeVoreDer<Cubes} to derive
	\begin{equation*}
		f^{\#*}_{Q_0}(t) \lesssim \sum_{l=0}^1 \Big[t^{-1/p_0} \Big(\int_0^{t} (u^{1/r} |\nabla^l f|^*(u))^{p_0} \, \frac{du}{u} \Big)^{1/p_0} + \sup_{t < u < 1} u^{1/d} |\nabla^l f|^*(u) \Big],
	\end{equation*}
	where $p_0' < d$ and $r = \frac{d p_0}{d+p_0}$. Thus, applying  Hardy's inequality we conclude that
	\begin{align*}
		\|f^{\#}_{Q_0}\|_{L_{\infty,q}(\log L)_b(Q_0)}  & \lesssim \sum_{l=0}^1 \bigg[ \bigg(\int_0^1 t^{-q/p_0} (1-\log t)^{b q}  \Big(\int_0^{t} (u^{1/r} |\nabla^l f|^*(u))^{p_0} \, \frac{du}{u} \Big)^{q/p_0} \frac{dt}{t} \bigg)^{1/q} \\
		& \hspace{1cm}+ \bigg(\int_0^1 \Big((1-\log t)^b \sup_{t < u < 1} u^{1/d} |\nabla^l f|^*(u) \Big)^{q} \frac{dt}{t} \bigg)^{1/q} \bigg] \\
		& \asymp \sum_{l=0}^1 \bigg[ \bigg(\int_0^1 (t^{1/d} (1-\log t)^b |\nabla^l f|^*(t))^q \frac{dt}{t} \bigg)^{1/q} \\
		& \hspace{1cm}+ \bigg(\int_0^1 \Big((1-\log t)^b \sup_{t < u < 1} u^{1/d} |\nabla^l f|^*(u) \Big)^{q} \frac{dt}{t} \bigg)^{1/q} \bigg]  \\
		& \asymp \sum_{l=0}^1  \bigg(\int_0^1 \Big((1-\log t)^b \sup_{t < u < 1} u^{1/d} |\nabla^l f|^*(u) \Big)^{q} \frac{dt}{t} \bigg)^{1/q}.
	\end{align*}

Assume that \eqref{As1} holds.
There is no loss of generality in fixing $Q_0 = [-\frac{1}{2},\frac{1}{2}]^d$. Firstly, we suppose $p=\infty$ and $r=\infty$. Then we will prove that the condition $\xi \leq -1$ is necessary if \eqref{As1} holds. 
Indeed, if $\xi > -1,$ then we choose $\beta$ such that $\max\{-b-1/q-\xi-1, 0\} < \beta < -b-1/q$ and define $f(x) = f_0(|x|)$, where
	\begin{equation*}
	f_0(t) = \int_t^1 (1-\log u)^\beta \frac{du}{u}, \qquad t \in (0, 1/2),
	\end{equation*}
	and $f_0(t)=0$ otherwise. Elementary computations yield that
	\begin{equation*}
		f^*(t) \asymp (1-\log t)^{\beta + 1} \quad \text{and} \quad  |\nabla f|^*(t) \asymp t^{-1/d} (1-\log t)^\beta
	\end{equation*}
	for $t \in (0,1)$. According to \eqref{key}, we have
	\begin{align*}
		\|f^{\#}_{Q_0}\|_{L_{\infty,q}(\log L)_b(Q_0)} & \lesssim  \bigg(\int_0^1 \Big((1-\log t)^b \sup_{t < u < 1} u^{1/d} (1- \log u)^{\beta + 1} \Big)^{q} \frac{dt}{t} \bigg)^{1/q}  \\
		& \hspace{1cm}+   \bigg(\int_0^1 \Big((1-\log t)^b \sup_{t < u < 1} (1-\log u)^\beta \Big)^{q} \frac{dt}{t} \bigg)^{1/q} < \infty
	\end{align*}
	but
	\begin{align*}
		\left(\int_0^1 (1 - \log t)^{b q} \Big(\sup_{t < u <1} (1-\log u)^{\xi} (f-f_{Q_0})^*(u) \Big)^{q} \frac{dt}{t} \right)^{1/q}
= \infty,
	\end{align*}
which contradicts \eqref{As1}.
	
	Secondly, we will show that \eqref{As1} with $p=\infty$ and $r < \infty$ implies $\xi < -1-1/r$. We again argue  by contradiction. Assume that \eqref{As1} holds with $p=\infty, r < \infty$ and $\xi = -1-1/r$. For each $j \in \N$, we consider the function $f_j(x) = f_j(|x|), \, x \in Q_0,$ given by
\begin{equation*}
	f_j(t) = \int_t^1 \frac{du}{u}, \qquad t \in (0, 2^{-j}),
	\end{equation*}
	and $f_j(t)=0$ otherwise.
Denoting by
 $ \chi_{(0, 2^{-j})}$  the characteristic function of the set $(0, 2^{-j})$,
it is readily seen that
	\begin{equation*}
		f_j^*(t) \asymp (-\log t) \chi_{(0, 2^{-j})}(t) \quad \text{and} \quad  |\nabla f_j|^*(t) \asymp t^{-1/d}  \chi_{(0, 2^{-j})}(t)
	\end{equation*}
	uniformly with respect to $j$.
 Therefore, we obtain
	\begin{align*}
		\left(\int_0^1 (1 - \log t)^{b q} \left(\int_t^1 ((1-\log u)^{-1-1/r} (f_j-f_{j;Q_0})^*(u))^r \frac{du}{u} \right)^{q/r} \frac{dt}{t} \right)^{1/q} & \gtrsim \\
		& \hspace{-7cm} \left(\int_0^{2^{-j}} (- \log t)^{b q} \bigg(\int_t^{2^{-j}} \frac{du}{u (-\log u)} \bigg)^{q/r} \frac{dt}{t} \right)^{1/q} \\
		&  \hspace{-7cm} \gtrsim \left(\int_0^{2^{-j}} (-\log t)^{b q} (\log (-\log t))^{q/r} \frac{dt}{t} \right)^{1/q} \asymp j^{b+1/q} (\log j)^{1/r}
	\end{align*}
	and (cf. \eqref{key})
	\begin{align*}
		\|(f_j)^{\#}_{Q_0}\|_{L_{\infty,q}(\log L)_b(Q_0)} & \lesssim 2^{-j/d} j \bigg(\int_0^{2^{-j}} (-\log t)^{b q} \frac{dt}{t} \bigg)^{1/q} +  \bigg(\int_0^{2^{-j}} (-\log t)^{b q} \frac{dt}{t} \bigg)^{1/q} \\
		& \asymp j^{b+1/q} (2^{-j/d} j + 1) \asymp j^{b+1/q}.
	\end{align*}
	Combining these estimates with \eqref{As1}, we arrive at $j^{b+1/q} (\log j)^{1/r} \lesssim j^{b+1/q}$,
	which leads to a contradiction. Furthermore, the failure of \eqref{As1} with $p=\infty, r < \infty$, and $\xi > -1-1/r$ can be obtained from the previous case using the trivial estimates.
	
	The general case $w \in A_\infty(Q_0)$ can be reduced to the Lebesgue setting using \eqref{anew} and a simple change of variables.
	Further details are left to the reader.

\end{proof}

\bigskip
\section{Proofs of Corollaries \ref{CorollaryLimitingBesovMax}, \ref{CorollaryLimitingBesovMaxSecond}, \ref{CorDeVoreExtrapol}, and \ref{CorDeVoreDerExtrapol} and their optimality}

	\begin{proof}[Proof of Corollary \ref{CorDeVoreExtrapol}]
		We concentrate only on \eqref{CorDeVoreExtrapol1cube} and leave to the reader the proofs of \eqref{CorDeVoreExtrapol1} and \eqref{CorDeVoreExtrapol2}. Without loss of generality, we may assume $0 < \varepsilon < d/p$. Suppose $r < \infty$. If $k > d/p$, then we apply monotonicity properties of $f^{\#*}_{Q_0}$, \eqref{ThmDeVore*LorentzCubes} and Fubini's theorem to derive
	\begin{align*}
		\|f^{\#}_{Q_0}\|_{L_{d/\varepsilon, r}(Q_0)} & \lesssim \bigg(\int_0^1 t^{(\varepsilon -d/p) r} \Big(\int_0^{t^d} (u^{1/p} f^{\#*}_{Q_0}(u))^q \, \frac{du}{u} \Big)^{r/q} \frac{dt}{t} \bigg)^{1/r} \\
		& \lesssim  \Big(\int_0^1 t^{\varepsilon r} \frac{dt}{t} \Big)^{1/r} \|f\|_{L_{p,q}(Q_0)} +  \Big(\int_0^1 t^{\varepsilon r} \int_t^1 (u^{-d/p} \omega_k(f,u)_{p, q})^r \frac{du}{u} \frac{dt}{t} \Big)^{1/r} \\
		& \asymp \varepsilon^{-1/r} \|f\|_{L_{p,q}(Q_0)} + \Big(\int_0^1 (u^{-d/p} \omega_k(f,u)_{p, q})^r \int_0^{u} t^{\varepsilon r} \frac{dt}{t} \frac{du}{u} \Big)^{1/r} \\
		& \asymp \varepsilon^{-1/r} \|f\|_{B^{d/p-\varepsilon}_r L_{p,q}(Q_0);k}.
	\end{align*}
	
	Let $k= d/p$. In light of \eqref{ThmDeVore*Lorentz2Cubes}, we have $f^{\#*}_{Q_0} (t^d) \lesssim \|f\|_{L_{p,q}(Q_0)} + t^{-d/p} \omega_k(f,t)_{p,q}$. Therefore,
	\begin{equation*}
		\|f^{\#}_{Q_0}\|_{L_{d/\varepsilon, r}(Q_0)} \lesssim \varepsilon^{-1/r} \|f\|_{L_{p,q}(Q_0)} +  \Big(\int_0^1 (t^{\varepsilon-d/p} \omega_k(f,t)_{p,q})^r \frac{dt}{t} \Big)^{1/r} \leq \varepsilon^{-1/r} \|f\|_{B^{d/p-\varepsilon}_r L_{p,q}(Q_0); k},
	\end{equation*}
	that is, \eqref{CorDeVoreExtrapol1cube} holds true.

	The case $r=\infty$ can be treated similarly. We omit the details.
	\end{proof}

\begin{proof}[Proof of Corollary \ref{CorollaryLimitingBesovMax}]
Assume $r < \infty$. Let $J \in \N$ be such that $2^{-J} < 1/p$. Setting $\varepsilon= 2^{-j} d, \, j \geq J$, in \eqref{CorDeVoreExtrapol1cube}, we have
	\begin{equation*}
		 2^{j b} \|f^{\#}_{Q_0}\|_{L_{2^j, r}(Q_0)} \leq C  2^{j (b+1/r)} \|f\|_{B^{d(1/p-2^{-j})}_r L_{p, q}(Q_0); k},
	\end{equation*}
	where $C > 0$ is independent of $j$. Therefore,
	\begin{equation}\label{ProofCorollaryLimitingBesovMax1}
		\Big(\sum_{j=J}^\infty  ( 2^{j b} \|f^{\#}_{Q_0}\|_{L_{2^j, r}(Q_0)})^r \Big)^{1/r} \lesssim \Big(\sum_{j=J}^\infty  (2^{j (b+1/r)} \|f\|_{B^{d(1/p-2^{-j})}_r L_{p, q}(Q_0); k})^r \Big)^{1/r}.
	\end{equation}
	Applying Fubini's theorem yields
	\begin{equation}\label{ProofCorollaryLimitingBesovMax2}
	\sum_{j=J}^\infty  2^{j b r} \|f^{\#}_{Q_0}\|_{L_{2^j, r}(Q_0)}^r = \int_0^1 V(t) ((f^{\#}_{Q_0})^*(t))^r \frac{dt}{t}
	\end{equation}
and, since $b < -1/r$,
	\begin{equation}\label{ProofCorollaryLimitingBesovMax3}
		\sum_{j=J}^\infty  2^{j (b+1/r) r} \|f\|_{B^{d(1/p-2^{-j})}_r L_{p, q}(Q_0); k}^r \asymp \|f\|_{L_{p,q}(Q_0)}^r +  \int_0^1 t^{-d r/p} W(t) \omega_k (f,t)_{p,q}^r \frac{dt}{t},
	\end{equation}
	where $V(t) = \sum_{j=J}^\infty 2^{j b r} t^{2^{-j} r}$ and $W(t) = \sum_{j=J}^\infty 2^{j(b+1/r) r} t^{2^{-j} d r}$. Next we estimate $V(t)$ and $W(t)$. For a fixed $t \in (0,1)$, changing variables, we  obtain
	\begin{equation}\label{0030303}
		V(t) \asymp \int^{2^{-J}}_0 t^{\sigma  r} \sigma^{-b r}  \frac{d \sigma}{\sigma} = (-\log t)^{b r} \int_0^{2^{-J} (-\log t)} e^{-\sigma r} \sigma^{-b r} \frac{d \sigma}{\sigma} \asymp  (-\log t)^{b r},
	\end{equation}
 and, similarly, $W(t) \asymp (-\log t)^{(b+1/r)r}$. Inserting these estimates into \eqref{ProofCorollaryLimitingBesovMax2} and \eqref{ProofCorollaryLimitingBesovMax3}, we have
	\begin{equation*}
		\sum_{j=J}^\infty  2^{j b r} \|f^{\#}_{Q_0}\|_{L_{2^j, r}(Q_0)}^r \asymp  \|f^{\#}_{Q_0}\|_{L_{\infty,r} (\log L)_b(Q_0)}^r
	\end{equation*}
	and
	\begin{align*}
		\sum_{j=J}^\infty  2^{j (b+1/r) r} \|f\|_{B^{d(1/p-2^{-j})}_r L_{p, q}(Q_0); k}^r & \asymp \|f\|_{L_{p,q}(Q_0)}^r +  \int_0^1 t^{-d r/p} (-\log t)^{(b+1/r) r} \omega_k(f,t)_{p,q}^r \frac{dt}{t} \\
		& \hspace{-4cm}\asymp \|f\|^r_{B_r^{d/p, b+ 1/r} L_{p,q}(Q_0);k}.
	\end{align*}
Thus, by \eqref{ProofCorollaryLimitingBesovMax1}, we conclude the proof of \eqref{CorollaryLimitingBesovMax1New*}.
	
	The case $r=\infty$ can be done in a similar way.
	
	\end{proof}

The optimal statement of Corollary \ref{CorollaryLimitingBesovMax} reads as follows.

\begin{prop}\label{CorollaryLimitingBesovMaxOptimal}
Let $1 < p < \infty, 0 < q, r \leq \infty, k > d/p, b < -1/r$, and $-\infty < \xi < \infty$. Then
			\begin{equation*}
			\|f^{\#}_{Q_0}\|_{L_{\infty, r} (\log L)_b (Q_0)} \lesssim \|f\|_{B^{d/p, b+\xi}_{r} L_{p,q}(Q_0);k} \iff \xi \geq 1/r.
			\end{equation*}
\end{prop}

The proof is based on the limiting interpolation technique and Fefferman--Stein inequalities for Lorentz--Zygmund spaces recently obtained in \cite{AstashkinMilman,Lerner}. In particular, we will need the following interpolation formulas for Lorentz--Zygmund spaces and Lorentz--Besov spaces.

\begin{lem}\label{Lemma8.2}
	Let $0 < p < \infty, 0 < r_0, r_1, r \leq \infty, -\infty < b_0, b < \infty, b_1 < -1/r$, and $0 < \theta < 1$. Then
	\begin{equation*}
		(L_{p,r_0}(\log L)_{b_0}(Q_0), L_{\infty, r_1}(\log L)_{b_1}(Q_0))_{\theta, r;b} = L_{p/(1-\theta), r} (\log L)_{(1-\theta) b_0 + \theta (b_1 + 1/r_1) + b}(Q_0).
	\end{equation*}
\end{lem}

\begin{lem}\label{Lemma8.3}
	Let
$1 < p < \infty, 0 < q \leq \infty,$  $0 < r_0, r_1, r \leq \infty$,
$0 < s_0 < s_1 < \infty$, $-\infty < b_0, b_1, b < \infty,$
and $0 < \theta < 1$. Then
	\begin{equation*}
		(B^{s_0,b_0}_{r_0} L_{p,q}(Q_0), B^{s_1,b_1}_{r_1} L_{p,q}(Q_0))_{\theta,r;b} = B^{(1-\theta) s_0 + \theta s_1, (1-\theta) b_0 + \theta b_1 + b}_r L_{p,q}(Q_0).
	\end{equation*}
	The corresponding formulas for homogeneous/inhomogeneous  Besov spaces on $\R^d$ also hold true.
\end{lem}

The proofs of Lemmas \ref{Lemma8.2} and \ref{Lemma8.3} follow from abstract reiteration formulas (see \cite{EvansOpicPick}).


\begin{lem}
	Let $0 < \theta_0, \theta_1, \theta < 1, \theta_0 \neq \theta_1, 0 < q_0, q_1, q \leq \infty, -\infty < b_0, b_1, b < \infty$. Then
	\begin{equation}\label{ReiterationOptim}
		((A_0, A_1)_{\theta_0, q_0; b_0}, (A_0, A_1)_{\theta_1, q_1; b_1})_{\theta, q; b} = (A_0, A_1)_{(1-\theta) \theta_0 + \theta \theta_1, q; (1-\theta) b_0 + \theta b_1+b}
	\end{equation}
	and if, additionally, $b_1 < -1/q_1$, then
	\begin{equation}\label{ReiterationOptim2}
		(A_0, (A_0, A_1)_{(1, b_1), q_1})_{\theta, q; b} = (A_0, A_1)_{\theta, q; \theta (b_1 + 1/q_1) + b}.
	\end{equation}
\end{lem}

\begin{proof}[Proof of Lemma \ref{Lemma8.2}]
	According to Lemma \ref{LemInterp1}, we have
	\begin{equation*}
		L_{\infty, r_1}(\log L)_{b_1}(Q_0) = (L_{p,r_0}(\log L)_{b_0}(Q_0), L_\infty(Q_0))_{(1,b_1), r_1}.
	\end{equation*}
	Thus, by \eqref{ReiterationOptim2},
	\begin{align*}
		(L_{p,r_0}(\log L)_{b_0}(Q_0), L_{\infty, r_1}(\log L)_{b_1}(Q_0))_{\theta, r;b}  & =  \\
		& \hspace{-4.5cm} (L_{p,r_0}(\log L)_{b_0}(Q_0),  (L_{p,r_0}(\log L)_{b_0}(Q_0), L_\infty(Q_0))_{(1,b_1), r_1})_{\theta, r;b} \\
		& \hspace{-4.5cm} = (L_{p,r_0}(\log L)_{b_0}(Q_0),  L_\infty(Q_0))_{\theta, r; \theta (b_1 + 1/r_1) + b} \\
		& \hspace{-4.5cm} = L_{p/(1-\theta), r; (1-\theta) b_0 + \theta(b_1 + 1/r_1) + b}(Q_0),
	\end{align*}
	where the last step follows from the well-known interpolation properties of Lorentz--Zygmund spaces (see, e.g., \cite[Corollary 5.3]{GogatishviliOpicTrebels}).
\end{proof}

\begin{proof}[Proof of Lemma \ref{Lemma8.3}]
	Let $k \in \N$ be such that $k > s_1$. It follows from the well-known formula (see, e.g., \cite[(3.5)]{Martin} and \cite[Chapter 11]{MartinMilman14}; see also \eqref{ProofLemmaEmbBMOLorentzState1*})
	\begin{equation}\label{747474}
		K(t^k, f;L_{p,q}(Q_0), W^k L_{p,q}(Q_0)) \asymp t^k \|f\|_{L_{p,q}(Q_0)} + \omega_k(f,t)_{p,q}, \quad t \in (0,1),
	\end{equation}
	and $K(t,f;L_{p,q}(Q_0), W^k L_{p,q}(Q_0)) \asymp \|f\|_{L_{p,q}(Q_0)}, \, t > 1$, that
	\begin{equation}\label{3993949}
		B^{s_i,b_i}_{r_i} L_{p,q}(Q_0) = (L_{p,q}(Q_0), W^k L_{p,q}(Q_0))_{\frac{s_i}{k},r_i;b_i}, \quad i=0, 1.
	\end{equation}
	Then, by \eqref{3993949} and \eqref{ReiterationOptim},
	\begin{align*}
		(B^{s_0,b_0}_{r_0} L_{p,q}(Q_0), B^{s_1,b_1}_{r_1} L_{p,q}(Q_0))_{\theta,r;b}  &= \\
		& \hspace{-4cm} ((L_{p,q}(Q_0), W^k L_{p,q}(Q_0))_{\frac{s_0}{k},r_0;b_0}, (L_{p,q}(Q_0), W^k L_{p,q}(Q_0))_{\frac{s_1}{k},r_1;b_1})_{\theta,r;b} \\
		& \hspace{-4cm} = (L_{p,q}(Q_0), W^k L_{p,q}(Q_0))_{((1-\theta) s_0 + \theta s_1)/k, r; (1-\theta) b_0 + \theta b_1 + b} \\
		& \hspace{-4cm} = B^{(1-\theta) s_0 + \theta s_1, (1-\theta) b_0 + \theta b_1 + b}_r L_{p,q}(Q_0).
	\end{align*}
\end{proof}

Now we are in a position to give

\begin{proof}[Proof of Proposition \ref{CorollaryLimitingBesovMaxOptimal}]
	In view of Corollary \ref{CorollaryLimitingBesovMax}, it only remains to show that  the inequality $\|f^{\#}_{Q_0}\|_{L_{\infty, r} (\log L)_b (Q_0)} \lesssim \|f\|_{B^{d/p, b+\xi}_{r} L_{p,q}(Q_0);k}$ implies $\xi \geq 1/r$. Let us denote by $T$ the sublinear operator mapping $f \in L_1(Q_0)$ to $f^{\#}_{Q_0}$, that is, $T (f) = f^{\#}_{Q_0}$. Then, taking into account our assumptions  (see the discussion after \eqref{Blowupnew}), we have
	\begin{equation*}
		T: B^{d/p, b+\xi}_{r} L_{p,q}(Q_0) \to L_{\infty, r} (\log L)_b (Q_0) \quad \text{and} \quad T: B^s_r L_{p,q}(Q_0) \to L_{d p/(d -s p), r}(Q_0),
	\end{equation*}
	where $0 < s < d/p$. Given any $\theta \in (0,1)$, by interpolation, 
  we arrive at
	\begin{equation}\label{ProofCorollaryLimitingBesovMaxOptimal1}
		T : (B^s_r L_{p,q}(Q_0), B^{d/p, b+\xi}_{r} L_{p,q}(Q_0))_{\theta, r} \to (L_{d p/(d -s p), r}(Q_0), L_{\infty, r} (\log L)_b (Q_0))_{\theta,r}.
	\end{equation}
	Invoking Lemmas \ref{Lemma8.2} and \ref{Lemma8.3},
	\begin{equation}
		(B^s_r L_{p,q}(Q_0), B^{d/p, b+\xi}_{r} L_{p,q}(Q_0))_{\theta, r} = B^{s_0, \theta (b + \xi)}_r L_{p,q}(Q_0), \label{ProofCorollaryLimitingBesovMaxOptimal2}
	\end{equation}
	where $s_0 = (1-\theta) s + \theta d/p \in (s, d/p)$, and
	\begin{equation}\label{ProofCorollaryLimitingBesovMaxOptimal3}
	(L_{d p/(d -s p), r}(Q_0), L_{\infty, r} (\log L)_b (Q_0))_{\theta,r} = L_{r_0, r} (\log L)_{\theta (b+1/r)}(Q_0),
	\end{equation}
	where $1/r_0 = (1-\theta) (d-s p)/d p$.
	
	
	According to \eqref{ProofCorollaryLimitingBesovMaxOptimal1}--\eqref{ProofCorollaryLimitingBesovMaxOptimal3}, we derive
	\begin{equation*}
	T: B^{s_0, \theta (b + \xi)}_r L_{p,q}(Q_0) \to L_{r_0, r} (\log L)_{\theta (b+1/r)}(Q_0)
	\end{equation*}
	with $s_0 -d/p = -d/r_0$. Moreover,
   if $s_0 -d/p = -d/r_0$, the r.i. hull of $B^{s_0, \theta (b + \xi)}_r L_{p,q}(Q_0)$ is the space $L_{r_0, r} (\log L)_{\theta (b+ \xi)}(Q_0)$ (see \cite[Theorem 3]{Martin}).
   Thus, in light of the Fefferman--Stein inequality
for the space $L_{r_0, r} (\log L)_{\theta (b+1/r)}(Q_0)$ (cf. \cite{AstashkinMilman,Lerner}),
we derive the embedding  $L_{r_0, r} (\log L)_{\theta (b+ \xi)}(Q_0) \hookrightarrow L_{r_0, r} (\log L)_{\theta (b+1/r)}(Q_0)$, which implies $\xi \geq 1/r$. \end{proof}

A careful examination of the proof of \eqref{CorollaryLimitingBesovMax1New*} given above shows that if
\begin{equation*}
\|f^{\#}_{Q_0}\|_{L_{d/\varepsilon, r}(Q_0)} \lesssim  \varepsilon^{-\xi} \|f\|_{B^{d/p-\varepsilon}_r L_{p, q}(Q_0); k}, \quad \varepsilon \to 0+,
\end{equation*}
holds for some $\xi > 0$, then
\begin{equation*}
	\|f^{\#}_{Q_0}\|_{L_{\infty, r} (\log L)_b (Q_0)} \lesssim \|f\|_{B^{d/p, b+\xi}_{r} L_{p,q}(Q_0);k}.
\end{equation*}
Hence, the optimality of \eqref{CorDeVoreExtrapol1cube} in Corollary \ref{CorDeVoreExtrapol} is a straightforward consequence of Proposition \ref{CorollaryLimitingBesovMaxOptimal}. More precisely, we have established the following

\begin{prop}\label{CorDeVoreExtrapolSharp}
	Let $1 < p < \infty, k > d/p,  0 < q, r \leq \infty$ and $\xi > 0$. Then
	\begin{equation*}
		\|f^{\#}_{Q_0}\|_{L_{d/\varepsilon, r}(Q_0)} \leq C \,  \varepsilon^{-\xi} \|f\|_{B^{d/p-\varepsilon}_r L_{p, q}(Q_0); k} \iff \xi \geq 1/r.
	\end{equation*}
\end{prop}

Next we prove  Corollaries   \ref{CorollaryLimitingBesovMaxSecond} and  \ref{CorDeVoreDerExtrapol}.

\begin{proof}[Proof of Corollary \ref{CorDeVoreDerExtrapol}]
	We begin by proving \eqref{CorDeVoreDerExtrapol1Cubes}. Take any $p > d/(d - k)$. It follows from \eqref{ThmDeVoreDer<Cubes} that
	\begin{equation*}
		f^{\#*}_{Q_0}(t) \lesssim \sum_{l=0}^k  \Big[ t^{-1/p} \Big(\int_0^{t} (u^{1/\nu} |\nabla^l f|^*(u))^r \, \frac{du}{u} \Big)^{1/r} +  \sup_{t < u < 1} u^{k/d} |\nabla^l f|^*(u) \Big],
	\end{equation*}
	where $\nu = \frac{d p}{d + k p}$.
There is no loss of generality in supposing that $0 < \varepsilon < d/p$. Let $r < \infty$. Therefore, applying monotonicity  properties
 of the
 rearrangements
 and changing the order of integration,
	\begin{align*}
		\|f^{\#}_{Q_0}\|_{L_{d/\varepsilon, r}(Q_0)} & \lesssim \sum_{l=0}^k \Big[  \Big(\int_0^1 t^{(\varepsilon/d - 1/p) r} \int_0^t  (u^{1/\nu} |\nabla^l f|^*(u))^r \, \frac{du}{u}  \frac{dt}{t} \Big)^{1/r}  \\
		&\hspace{1cm}+ \bigg(\int_0^1 \Big(t^{\varepsilon/d} \sup_{t < u < 1} u^{k/d} |\nabla^l f|^*(u) \Big)^r \frac{dt}{t} \bigg)^{1/r} \Big] \\
		&\lesssim \sum_{l=0}^k \Big[ \Big(\int_0^1 (u^{1/\nu} |\nabla^l f|^*(u))^r \int_u^1 t^{(\varepsilon/d - 1/p) r} \frac{dt}{t} \frac{du}{u} \Big)^{1/r} \\
		& \hspace{1cm}  +  \Big(\int_0^1 t^{\varepsilon r/d} \int_t^1 (u^{k/d} |\nabla^l f|^*(u))^r \frac{du}{u} \frac{dt}{t} \Big)^{1/r} \Big]  \\
		& \asymp (1 +  \varepsilon^{-1/r} ) \sum_{l=0}^k \Big(\int_0^1 (u^{ (k + \varepsilon)/d}  |\nabla^l f|^*(u))^r \frac{du}{u} \Big)^{1/r} \\
		& \asymp \varepsilon^{-1/r} \|f \|_{W^k L_{d/(k + \varepsilon), r}(Q_0)}.
	\end{align*}

	The case $r=\infty$ is easier and we omit further details. Estimate \eqref{CorDeVoreDerExtrapol1} is proved similarly but now applying \eqref{ThmDeVoreDer<}.
	
\end{proof}

\begin{proof}[Proof of Corollary \ref{CorollaryLimitingBesovMaxSecond}]
	Let $r < \infty$. In view of \eqref{CorDeVoreDerExtrapol1Cubes}, we have
	\begin{equation*}
		\|f^{\#}_{Q_0}\|_{L_{2^j, r}(Q_0)} \leq C \,  2^{j/r} \|f\|_{W^k L_{\frac{d}{k+ 2^{-j} d}, r} (Q_0)},
	\end{equation*}
	where $C$ does not depend on $j$. Hence,
	\begin{equation*}
		\Big(\sum_{j=0}^\infty 2^{j b r} \|f^{\#}_{Q_0}\|_{L_{2^j, r}(Q_0)}^r \Big)^{1/r} \lesssim \Big(\sum_{j=0}^\infty 2^{j (b + 1/r) r} \|f\|_{W^k L_{\frac{d}{k+  2^{-j} d}, r} (Q_0)}^r  \Big)^{1/r}.
	\end{equation*}
	These extrapolation spaces can be easily computed by applying Fubini's theorem. Namely, using
$\sum_{j=0}^\infty 2^{jA} t^{B/2^j}\asymp (-\log t)^A$, $0<t<1/2$, for $A\in \mathbb{R}$ and $B>0$, (cf. \eqref{0030303})
we have
	\begin{equation*}
	\Big(\sum_{j=0}^\infty 2^{j b r} \|f^{\#}_{Q_0}\|_{L_{2^j, r}(Q_0)}^r \Big)^{1/r} \asymp \| f^{\#}_{Q_0}\|_{L_{\infty, r} (\log L)_b(Q_0)}
	\end{equation*}
	and
	\begin{equation*}
	\Big(\sum_{j=0}^\infty 2^{j (b + 1/r) r} \|f\|_{W^k L_{\frac{d}{k+  2^{-j} d}, r} (Q_0)}^r  \Big)^{1/r} \asymp \|f\|_{W^k L_{d/k, r} (\log L)_{b+1/r} (Q_0)}.
	\end{equation*}
Hence $ \| f^{\#}_{Q_0}\|_{L_{\infty, r} (\log L)_b(Q_0)} \lesssim		 \|f\|_{W^k L_{d/k, r} (\log L)_{b+1/r} (Q_0)}.
$
	
	We omit the proof in the case $r=\infty$, since it is similar to that in the case $r < \infty$.
\end{proof}

The next result shows that  Corollary \ref{CorollaryLimitingBesovMaxSecond} is in fact sharp.

\begin{prop}\label{CorollaryLimitingBesovMaxSecondOptimal}
	Let $1 \leq r \leq \infty, k < d, b < -1/r$, and $-\infty < \xi < \infty$. Then
	\begin{equation*}
	\|f^{\#}_{Q_0}\|_{L_{\infty, r} (\log L)_b (Q_0)} \lesssim \|f \|_{W^k L_{d/k, r} (\log L)_{b + \xi}(Q_0) }
 \iff \xi \geq 1/r.
	\end{equation*}
\end{prop}

\begin{proof}
	 Corollary \ref{CorollaryLimitingBesovMaxSecond} implies the ``if part".	
	Let us show that if the inequality
	\begin{equation*}
	\|f^{\#}_{Q_0}\|_{L_{\infty, r} (\log L)_b (Q_0)} \lesssim \|f \|_{W^k L_{d/k, r} (\log L)_{b + \xi}(Q_0) }
	\end{equation*}
holds true,
	then necessarily  $\xi \geq 1/r$. As in the proof of  Proposition \ref{CorollaryLimitingBesovMaxOptimal}, for
$Tf=f^{\#}_{Q_0}$, we have $T: W^k L_{d/k, r} (\log L)_{b + \xi}(Q_0) \to L_{\infty, r} (\log L)_b (Q_0)$. By \eqref{ProofLem2.1} and Hardy's inequality (noting that $d/k > 1$), 
  the operator $T$ acts boundedly on $L_{d/k, r} (\log L)_{b + \xi}(Q_0)$. For $\theta \in (0,1)$ and $q_0 \in (1,\infty)$, we set
	\begin{equation*}
		X = (L_{d/k, r} (\log L)_{b + \xi}(Q_0), W^k L_{d/k, r} (\log L)_{b + \xi}(Q_0))_{\theta, q_0}
	\end{equation*}
	and
	\begin{equation*}
		Y = (L_{d/k, r} (\log L)_{b + \xi}(Q_0), L_{\infty, r} (\log L)_b (Q_0))_{\theta, q_0}.
	\end{equation*}
	Therefore, by the interpolation property, we derive
	\begin{equation}\label{CorollaryLimitingBesovMaxSecondOptimalProof1}
		T :  X \to Y.
	\end{equation}
	It remains to identify the interpolation spaces $X$ and $Y$. It is an immediate consequence of the well-known characterization (cf. \eqref{747474})
	\begin{align*}
		K(t^k, f; L_{d/k, r}(\log L)_{b + \xi}(Q_0), W^k L_{d/k, r}(\log L)_{b + \xi}(Q_0)) &\asymp \\
		& \hspace{-6cm}  \min\{1,t^k\} \|f\|_{ L_{d/k, r}(\log L)_{b + \xi}(Q_0)} + \omega_k(f,t)_{d/k,r, b + \xi}
	\end{align*}
for $f\in L_{d/k, r}(\log L)_{b + \xi}(Q_0)$ and $t > 0$	that $X = B^{\theta k}_{q_0} L_{d/k, r}(\log L)_{b + \xi}(Q_0).$

On the other hand, taking into account Lemma \ref{Lemma8.2}, $Y = L_{d/((1-\theta) k), q_0} (\log L)_{b + \xi +\theta (1/r-\xi)}(Q_0).$ Hence \eqref{CorollaryLimitingBesovMaxSecondOptimalProof1} can be rewritten as
	\begin{equation*}
	 T:  B^{\theta k}_{q_0} L_{d/k, r}(\log L)_{b + \xi}(Q_0) \to L_{d/((1-\theta) k), q_0} (\log L)_{b + \xi +\theta (1/r-\xi)}(Q_0),
	\end{equation*}
	or, equivalently, in light of the Fefferman--Stein inequality for Lorentz--Zygmund spaces, 
	\begin{equation*}
		B^{\theta k}_{q_0} L_{d/k, r}(\log L)_{b + \xi}(Q_0)  \hookrightarrow L_{d/((1-\theta) k), q_0} (\log L)_{b + \xi +\theta (1/r-\xi)}(Q_0).
	\end{equation*}
	Finally, it follows from the fact that $L_{d/((1-\theta) k), q_0} (\log L)_{b+\xi}(Q_0)$ is the r.i. hull of the Besov space $B^{\theta k}_{q_0} L_{d/k, r}(\log L)_{b + \xi}(Q_0)$ (see \cite[Theorem 3]{Martin}) that
	\begin{equation*}
	L_{d/((1-\theta) k), q_0} (\log L)_{b+\xi}(Q_0) \hookrightarrow  L_{d/((1-\theta) k), q_0} (\log L)_{b + \xi +\theta (1/r-\xi)}(Q_0),
	\end{equation*}
	which yields $\xi \geq 1/r$.
\end{proof}

We are now in a position to prove the sharpness of Corollary \ref{CorDeVoreDerExtrapol}. Indeed, assume that there exists $\xi > 0$ such that
\begin{equation*}
		\|f^{\#}_{Q_0}\|_{L_{\frac{d}{\varepsilon}, r}(Q_0)} \leq C \,  \varepsilon^{-\xi} \|f\|_{W^k L_{\frac{d}{k+\varepsilon}, r} (Q_0)}, \quad \varepsilon \to 0+.
	\end{equation*}
	Then, following step by step the proof of Corollary \ref{CorollaryLimitingBesovMaxSecond} given above, we arrive at
	\begin{equation*}
		\|f^{\#}_{Q_0}\|_{L_{\infty, r} (\log L)_b (Q_0)} \lesssim \|f\|_{W^k L_{d/k, r} (\log L)_{b + \xi}(Q_0)}, \quad b < -1/r.
	\end{equation*}
	In light of Proposition \ref{CorollaryLimitingBesovMaxSecondOptimal}, we have shown the following

\begin{prop}\label{CorDeVoreDerExtrapolOptim}
		Let $1 \leq r \leq \infty, k < d$, and $\xi > 0$. Then
	\begin{equation*}
		\|f^{\#}_{Q_0}\|_{L_{\frac{d}{\varepsilon}, r}(Q_0)} \leq C \,  \varepsilon^{-\xi} \|f\|_{W^k L_{\frac{d}{k+\varepsilon}, r} (Q_0)} \iff \xi \geq 1/r.
	\end{equation*}
	\end{prop}

\bigskip
\section{Comparison between Theorems \ref{ThmDeVoreLorentz} and \ref{ThmDeVoreDer}}

Let $1 < p, q < \infty$. Assume $f \in C^\infty_0(\R^d)$. We have shown in Theorems \ref{ThmDeVoreLorentz} and \ref{ThmDeVoreDer} that
	\begin{equation}\label{Comp1}
	t^{-d/p} \Big(\int_0^{t^d} (u^{1/p} f^{\# *}(u))^q \, \frac{du}{u} \Big)^{1/q} \lesssim \sup_{ t < u < \infty} u^{-d/p} \omega_k(f,u)_{p,q}, \quad k \geq d/p,
	\end{equation}
and
	\begin{align}
	t^{-d/p} \Big(\int_0^{t^d} (u^{1/p} f^{\# *}(u))^q \, \frac{du}{u} \Big)^{1/q} & \lesssim t^{-d/p} \Big(\int_0^{t^d} (u^{1/r} |\nabla^k f|^*(u))^q \, \frac{du}{u} \Big)^{1/q} \nonumber\\
	& \hspace{1cm}+ \sup_{t^d < u < \infty} u^{k/d} |\nabla^k f|^*(u), \quad k < d (1-1/p),  \label{Comp2}
	\end{align}
	where $1/r = 1/p + k/d$.  
 The goal of this section is to study  the interrelations between these estimates. 
  We distinguish four cases.
\\	
	\textsc{Case 1:} If $p > 2$ and $k \in \big(\frac{d}{p}, d \big(1 - \frac{1}{p} \big)\big)$, then  \eqref{Comp1} provides a sharper estimate than \eqref{Comp2}. More precisely, there holds 
	\begin{align}
		 \sup_{ t < u < \infty} u^{-d/p} \omega_k(f,u)_{p,q} & \lesssim t^{-d/p} \Big(\int_0^{t^d} (u^{1/r} |\nabla^k f|^*(u))^q \, \frac{du}{u} \Big)^{1/q} \nonumber \\
		 &\hspace{1cm} + \sup_{t^d < u < \infty} u^{k/d} |\nabla^k f|^*(u). \label{Comp3}
	\end{align}
	Let us show \eqref{Comp3}. According to \cite[Theorem 1.2]{SeegerTrebels}, we have
	\begin{equation*}
		\dot{W}^k L_{d/k,\infty}(\R^d) \hookrightarrow \dot{B}^{d/p}_{\infty} L_{p,q}(\R^d)
	\end{equation*}
	since $d/p <k < d$. Then this embedding and \eqref{ThmDeVoreDer1*} imply
	\begin{equation}\label{Comp4}
		K(t, f; L_{p,q}(Q_0), \dot{B}^{d/p}_{\infty} L_{p,q}(\R^d)) \lesssim K(t, f;
		\dot{W}^k L_{r,q}(\R^d) ,\dot{W}^k L_{d/k,\infty}(\R^d) ).
	\end{equation}
	Recall that (see \eqref{ProofThmDeVoreLorentz3} and \eqref{ThmDeVoreDer3})
	\begin{equation}\label{Comp5}
		K(t, f; L_{p,q}(\R^d), \dot{B}^{d/p}_{\infty} L_{p,q}(\R^d))  \asymp t \sup_{t^{p/d} < u <\infty} u^{-d/p} \omega_k(f,u)_{p,q}
	\end{equation}
	and
	\begin{align}
	K(t, f;
		\dot{W}^k L_{r,q}(\R^d) ,\dot{W}^k L_{d/k,\infty}(\R^d) ) & \asymp \Big(\int_0^{t^p} (u^{1/r} |\nabla^k f|^*(u))^q \, \frac{du}{u} \Big)^{1/q} \nonumber \\
		& \hspace{1cm}+ t \sup_{t^p < u < \infty} u^{k/d} |\nabla^k f|^*(u). \label{Comp6}
	\end{align}
	Combining \eqref{Comp4}--\eqref{Comp6}, we conclude the desired estimate \eqref{Comp3}.
	\\
	\textsc{Case 2:} Let $p > 2$ and $k=d/p$ (and so, $k < d(1-1/p)$ and $r=p/2$). Under these assumptions, it turns out that \eqref{Comp1} and \eqref{Comp2} are independent of each other. More precisely, we will show  that
the right-hand side expressions in \eqref{Comp1} and \eqref{Comp2}, i.e.,
		\begin{equation*}
		 I(t) =  t^{-k} \omega_{k}(f,t)_{p,q}
	\end{equation*}
	and
	\begin{equation*}
		J(t) =t^{-d/p} \Big(\int_0^{t^d} (u^{1/r} |\nabla^k f|^*(u))^q \, \frac{du}{u} \Big)^{1/q} + \sup_{t^d < u < \infty} u^{k/d} |\nabla^k f|^*(u),
	\end{equation*}
	are not comparable. This will be shown by contradiction. Assume first that
	\begin{equation}\label{Comp7}
		I(t) \leq C J(t),
	\end{equation}
	where $C$ is a positive constant which is independent of $t \in (0,1)$. Since
	\begin{equation*}
		t^{-d/p} \Big(\int_0^{t^d} (u^{1/r} |\nabla^k f|^*(u))^q \, \frac{du}{u} \Big)^{1/q} \lesssim \sup_{0 < u < \infty} u^{k/d}  |\nabla^k f|^*(u),
	\end{equation*}
	we have
	\begin{equation*}
		\sup_{t^d < u < \infty} u^{k/d} |\nabla^k f|^*(u) \leq J(t) \lesssim \sup_{0 < u < \infty} u^{k/d} |\nabla^k f|^*(u).
	\end{equation*}
Letting $t\to 0$, 
  we derive (recall that $k=d/p$)
	\begin{equation*}
		\lim_{t \to 0+} J(t) \asymp \sup_{0 < u < \infty} u^{k/d} |\nabla^k f|^*(u) = \|\,|\nabla^k f| \,\|_{L_{p,\infty}(\R^d)}.
	\end{equation*}
	On the other hand, by \eqref{ProofLemmaEmbBMOLorentzState1*}, we find that
	\begin{equation*}
	\lim_{t \to 0+} I(t) \asymp \lim_{t \to 0+} t^{-1} K(t, f; L_{p,q}(\R^d), \dot{W}^k L_{p,q}(\R^d)) \asymp \|\,|\nabla^k f| \,\|_{L_{p,q}(\R^d)},
	\end{equation*}
	where in the last step we used the fact that the space $\dot{W}^k L_{p,q}(\R^d)$ is reflexive (see \cite[Theorem 1.4,  p. 295]{BennettSharpley}). Then, by \eqref{Comp7}, we arrive at
	\begin{equation*}
		\dot{W}^k L_{p,\infty}(\R^d) \hookrightarrow \dot{W}^k L_{p,q}(\R^d),
	\end{equation*}
	which fails to be true because $q < \infty$.
	
	Suppose now that
	\begin{equation}\label{Comp8}
		J(t) \leq C I(t).
	\end{equation}
We observe that 
	\begin{equation*}
		\sup_{t^d < u < \infty} u^{k/d} |\nabla^k f|^*(u) \lesssim t^{-d/p}  \Big(\int_0^{\infty} (u^{1/r} |\nabla^k f|^*(u))^q \, \frac{du}{u} \Big)^{1/q}.
	\end{equation*}
	This yields 
	\begin{equation*}
		\Big(\int_0^{t^d} (u^{1/r} |\nabla^k f|^*(u))^q \, \frac{du}{u} \Big)^{1/q} \leq t^k J(t) \lesssim \Big(\int_0^\infty (u^{1/r} |\nabla^k f|^*(u))^q \, \frac{du}{u} \Big)^{1/q}
	\end{equation*}
	and thus
	\begin{equation*}
		\lim_{t \to \infty} t^k J(t) \asymp \Big(\int_0^\infty (u^{1/r} |\nabla^k f|^*(u))^q \, \frac{du}{u} \Big)^{1/q} = \|\,|\nabla^k f| \,\|_{L_{r,q}(\R^d)}.
	\end{equation*}
	Therefore passing to the limit as $t \to \infty$ in \eqref{Comp8}, we obtain
	\begin{equation*}
		\|\,|\nabla^k f| \,\|_{L_{r,q}(\R^d)} \lesssim \lim_{t \to \infty} \omega_k(f,t)_{p,q} \lesssim \|f\|_{L_{p,q}(\R^d)},
	\end{equation*}
	which yields  the desired contradiction.
	\\
	\textsc{Case 3:} Assume that either $p > 2$ and $k \not \in \big[\frac{d}{p}, d \big(1 - \frac{1}{p} \big)\big)$ or $p=2$. Then in the case when $k \geq  d \big(1 - 1/p \big)$ we apply  \eqref{Comp1} and if $k < d/p$ we use \eqref{Comp2}.
	\\
	\textsc{Case 4:} Suppose $1 < p < 2$ (and so, $d(1-1/p) < d/p$). On the one hand, if $k \geq d/p$ then \eqref{Comp1} holds true. On the other hand, \eqref{Comp2} can be invoked whenever $k < d(1-1/p)$. In the remaining parameter range $k \in \big[d(1-1/p), d/p\big)$, \eqref{Comp1} and \eqref{Comp2} cannot be applied. 


\section*{Appendix A. Some integral properties of slowly varying functions}
Here we show that for any $0 < q \leq \infty$ there exists a slowly varying function  $b$ defined on $(B,\infty)$ such that
$\int_B^{\infty} (b(u))^q \frac{du}{u} < \infty$
and, for any  $p>0$,
\begin{equation*}
\frac{\int_t^{\infty} (b(u))^q \frac{du}{u}}{\int_{t (\log t)^p}^{\infty} (b(u))^q \frac{du}{u}}\rightarrow \infty
\quad\mbox{as}\quad t\to\infty.
	\end{equation*}
\begin{proof}
Let $q < \infty$. Applying change of variables, matters reduce to the case $q=1$. Consider
$$b(x)=\exp\left\{-\int_A^{\log x}\frac{1}{\sqrt{\log t}} dt\right\}, \quad x > e^{A}.$$
It is clearly a slowly varying function since
$$1\leqslant \frac{b(x)}{b(cx)}
\leq \exp\left\{\log c \frac{1}{\sqrt{\log\log x}}\right\} \rightarrow  1\quad\mbox{as}\quad x\to\infty$$
for all $c \geq 1$.
Moreover, the condition
$\int_B^{\infty} b(u) \frac{du}{u} < \infty$ follows from
the estimate $b(x)\leq (\log x)^{-2}$ for sufficiently large $x$,
which can be checked straightforwardly.

Now we fix $p>0$ and  $C>0$.
We note that
\begin{equation}\label{sv1}
\frac{b(x)}{b(x (\log x)^p)}\rightarrow \infty
\quad\mbox{as}\quad x\to\infty.
	\end{equation}
 Then there is  $M>0$ (assume that  $M>e^{2^{1/p}}$) such that for $x>M$ there holds
 $b(x)>C b(x (\log x)^p)$.

We denote  $h_0(t)=t,\;h_{k+1}(t)=h_k(t)(\log h_k(t))^p$ for $k\geq 0$. Then  
 $t (\log t)^p >2t$ for large $t$ implies $h_{k+1}(t) > 2 h_k(t)$ for $k\geq 0$ and then
$$\int_t^{\infty} b(u) \frac{du}{u}=\sum_{k=0}^{\infty}\int_{h_k(t)}^{h_{k+1}(t)} b(u) \frac{du}{u}=:\sum_{k=0}^{\infty} H_k(t),$$
$$\int_{t (\log t)^p}^{\infty} b(u)\frac{du}{u}=\sum_{k=1}^{\infty}\int_{h_k(t)}^{h_{k+1}(t)} b(u) \frac{du}{u}=\sum_{k=1}^{\infty} H_k(t).
$$
Further, for $x>M$ we have
$$
\int_x^{x (\log x)^p} b(u) \frac{du}{u}
\geq C
\int_x^{x (\log x)^p} b(u (\log u)^p) \frac{du}{u}\geq
\frac{C}{2} \int_{x (\log x)^p}^{x (\log x)^p (\log(x (\log x)^p))^p} b(u) \frac{du}{u}.
$$
Since  $h_k(t)\geq t$ for any $k\geq 0$, this yields that for $t>\max\{e,e^p,M\}$ there holds
$H_k(t)\geq \frac{C}{2}H_{k+1}(t),$
which implies
$$\int_t^{\infty} b(u) \frac{du}{u}\geq \frac{C}{2} \int_{t (\log t)^p}^{\infty} b(u) \frac{du}{u}.$$
Since, by (\ref{sv1}), $C>0$ is arbitrary, we conclude the proof.

The previous reasoning can be easily adapted to the case $q=\infty$.
\end{proof}

{\bf{Acknowledgements} }
We would like to thank Kristina Oganesyan for useful remarks. The first author was partially supported by MTM 2017-84058-P. The second
author was partially supported by
 MTM 2017-87409-P,  2017 SGR 358, and
 the CERCA Programme of the Generalitat de Catalunya. Part of this work was done during the visit of the authors to the Isaac Newton Institute for Mathematical Sciences, Cambridge, EPSCR Grant no EP/K032208/1.

\end{document}